\documentclass[1p]{elsarticle}
\usepackage{amsfonts}

\usepackage{amsmath,amssymb,comment}
\usepackage{epsfig}
\usepackage{graphicx}
\usepackage{MnSymbol}

\DeclareMathAlphabet{\mathpzc}{OT1}{pzc}{m}{it}
\catcode`@=11
\@addtoreset{equation}{section}
\catcode`@=12

\newtheorem{theorem}{Theorem}[section]

\newtheorem{claim}[theorem]{Claim}

\newtheorem{corollary}[theorem]{Corollary}

\newtheorem{definition}[theorem]{Definition}

\newtheorem{lemma}[theorem]{Lemma}

\newtheorem{remark}[theorem]{Remark}

\newenvironment{proof}[1][Proof]{\textbf{#1.} }{\ $\Box$ \\}

\def\NN{\mathbb{N}}
\def\RR{\mathbb{R}}
\def\ZZ{\mathbb{Z}}
\def\C{\mathbb{C}}
\def\CC{\mathcal{C}}

\def\OO{\mathcal{O}}
\def\oo{\mathpzc{o}}

\def\ee{\text{e}}
\def\ii{\text{i}}
\def\re{\text{\rm Re}\,}
\def\im{\text{\rm Im}\,}
\def\dd{\,d}

\newcommand\Hom[1]{\mathcal{H}^{\geq #1}}
\newcommand\Hog[1]{\mathcal{H}^{#1}}
\newcommand\ho[1]{{\mathcal{H}}^{>#1}}

\def\px{\pi_{x}}
\def\py{\pi_{y}}

\def\a{\alpha}
\def\r{\varrho}
\newcommand{\Vr}{V_{\r}}
\newcommand{\ro}{\r_0}
\newcommand{\Vro}{V_{\ro}}

\def\df{\ell_*}
\def\gdf{r_*}
\def\rpa{r_{\pa}}

\def\Ap{A_{p}}
\def\Bp{B_{p}}
\def\Bq{B_{q}}
\def\ap{a_{p}}
\def\bp{b_{p}}
\def\cp{c_{p}}
\def\ddp{d_{p}}

\def\CIn{a_V}

\def\pa{\mathbf{p}}
\def\Qa{\mathbf{Q}}
\def\T{\mathbf{w}}
\newcommand\Qaj[1]{\Qa^{#1}}
\newcommand\TT[1]{\T^{#1}}

\def\apa{a_{\pa}}
\def\bpa{b_{\pa}}
\def\cpa{c_{\pa}}
\def\dpa{d_{\pa}}
\def\AQ{A_{\Qa}}
\def\BQ{B_{\Qa}}
\def\Apa{A_{\pa}}
\def\Bpa{B_{D\pa}}
\def\CIP{a_V^{\pa}}

\def\fpa{\varphi}
\def\fqa{\psi}

\def\My{M_{q}}

\def\Id{\text{\rm Id}\,}
\def\diag{\text{\rm diag\,}}
\def\id{\iota}

\newcommand{\m}{\mathfrak{m}}

\def\DS{\Sigma_{s,r}}
\newcommand\BH[3]{\mathcal{B}^{#1}_{#2,#3}}

\def\Dl{D_{\lambda}}
\def\Dx{D_{x}}
\def\Dz{D_{z}}

\def\opC{\mathcal{S}}

\newcommand\Di[1]{D_{#1}}
\newcommand\HH[1]{H^{#1}}

\newcommand\IMa[2]{\mathcal{I}_{#1}^{#2}}
\newcommand\BMa[1]{\mathcal{B}_{#1}}

\newcommand\M[1]{\chi^{#1}}
\newcommand\MP{\chi}

\def\Qi{\mathcal{Q}}
\def\DQi{\Qi^1}
\newcommand\Ti[1]{\omega^{#1}}

\def\BDQi{B_{\DQi}}
\newcommand\BQi{B_{\Qi}}
\newcommand\BQil[1]{B_{#1}}

\renewcommand\L[1]{\mathcal{L}^{#1}}

\newcommand\Kl[1]{K^{#1}}
\newcommand\Klp[2]{K^{#1}_{#2}}
\newcommand\KMil[1]{K^{\leq #1}}
\newcommand\Kalg[1]{K^{(#1)}}

\newcommand\kf{\Kl{j}}
\newcommand\kfp[1]{\kf_{#1}}
\newcommand\Kff{\KMil{j}}

\newcommand\RMil[1]{R^{\leq #1}}
\newcommand\Rff{\RMil{j+N-1}}
\newcommand\Rl[1]{ R^{#1}}
\newcommand\rf{\Rl{j+N-1}}

\newcommand\Yl[1]{Y^{#1}}
\newcommand\YMil[1]{Y^{\leq #1}}

\newcommand\El[2]{E^{#1}_{#2}}
\newcommand\Ef[1]{E_{#1}}
\newcommand\Em[1]{E^{>#1}}
\newcommand\Ew[2]{\widehat E^{>#1}_{#2}}
\newcommand\Et[2]{\tilde{E}^{>#1}_{#2}}

\newcommand\ME[1]{\overline{#1}}
\newcommand\ZE[1]{\widetilde{#1}}

\begin{document}

\title{Invariant manifolds of parabolic fixed points (II).
Approximations by sums of homogeneous functions}

\author{Inmaculada~Baldom\'a}
\ead{immaculada.baldoma@upc.edu}
\address{Departament de Matem\`{a}tiques,\\
         Universitat Polit\`{e}cnica de Catalunya,\\
         Diagonal 647, 08028 Barcelona, Spain}

\author{Ernest~Fontich}
\ead{fontich@ub.edu}
\address{Departament de Matem\`{a}tiques i Inform\`atica,\\
        Universitat de Barcelona, \\ Gran Via 585,
        08007 Barcelona, Spain }

\author{Pau~Mart\'{\i}n}
\ead{p.martin@upc.edu}
\address{Departament de Matem\`{a}tiques,\\
         Universitat Polit\`{e}cnica de Catalunya,\\
         Ed.~C3, Jordi Girona 1--3, 08034 Barcelona, Spain}

\begin{abstract}
    We study the computation of local approximations of invariant
    manifolds of parabolic fixed points and parabolic periodic
    orbits of periodic vector fields. If the dimension of these
    manifolds is two or greater, in general, it is not possible to
    obtain polynomial approximations. Here we develop an algorithm
    to obtain them as sums of homogeneous functions by solving
    suitable cohomological equations. We deal with both the differentiable and analytic cases. We also
    study the dependence on parameters. In the companion paper~\cite{BFM2015a}
    these approximations are used to obtain the existence of true
    invariant manifolds close by. Examples are provided.
\end{abstract}

\maketitle

\tableofcontents

\section{Introduction}

This paper is the second part of our study on the invariant
manifolds of parabolic points for $\CC^r$ and analytic maps started
in~\cite{BFM2015a}. We refer to that paper for the motivation and
references concerning such setting.

In this set of two papers we provide conditions that guarantee the
existence of stable invariant manifolds associated of such points.
We use the parametrization method
\cite{CabreFL03a,CabreFL03b,CabreFL05,HdL2006,HdL2007,HaroetAl}. The
operators involved in this method are more regular than the graph
transform, which is an advantage in the present situation, where
only finite differentiability is assumed. Also, it often provides
efficient algorithms to compute explicitly approximations of the
invariant manifolds. In fact, this is the main purpose of the
present paper. To apply this method we need a minimum regularity to
be able to have a polynomial approximation of the map.

We consider maps $F:U\subset \RR^n\times \RR^m \to \RR^n\times
\RR^m$, with $(0,0)\in U$ such that $F(0,0) = (0,0)$, $DF(0,0) =
\Id$. We assume some hypotheses, to be specified later, on the first
non-vanishing nonlinear terms which imply the existence of some
``weak contraction'' in the $(x,0)$-directions, as well as some
hypotheses concerning the $(0,y)$-directions that may imply ``weak
expansion'' in these directions (but not always). The
parametrization method consists of looking for the invariant stable
manifold~$W^s$ of the origin as an immersion $K:V\subset \RR^n\to
\RR^n\times \RR^m$, with $K(0) = (0,0)$, $DK(0) = (\Id,0)^{\top}$,
and satisfying the \emph{invariance equation}
\begin{equation}
\label{eq:inv_equation} F \circ K = K \circ R,
\end{equation}
where $R: V\to V$ is a reparametrization of the dynamics of~$F$
on~$W^s$.

The procedure to find such~$K$ and~$R$ has two steps. First, to find
functions $K^{\le}$ and $R$ solving approximately the invariance
equation, that is, satisfying
\begin{equation}
\label{eq:inv_approx} F \circ K^{\le}(x) - K^{\le} \circ R (x) = \oo
(\|x\|^{\ell}),
\end{equation}
to a high enough order which depends on the first non-vanishing
nonlinear terms of~$F$.

Second, with the reparametrization~$R$ obtained so far to look
for~$K$ as a perturbation of $K^{\le}$. This second step is carried
out in~\cite{BFM2015a} where, assuming that~$R$ and a sufficiently
good approximation~$K^{\le}$ are known, an ``a posteriori'' type
result is obtained.

In this paper we obtain approximate solutions
of~\eqref{eq:inv_equation}. This is accomplished by solving a set of
cohomological equations. In the case that the fixed point is
hyperbolic instead of parabolic, it is possible to find solutions of
the cohomological equations in the ring of polynomials, both for $K$
and $R$ (see~\cite{CabreFL03a,CabreFL03b,CabreFL05}). The same
happens when one looks for one dimensional invariant manifolds
associated to parabolic fixed points~\cite{BFdLM2007}.

However, when the parabolic invariant manifolds have dimension two
or more, a simple computation shows that generically there are no
polynomial approximate solutions of the invariance equation. The
reason is simple: when looking for polynomial solutions, since the
terms of order $k$ are determined in order to kill the terms of
order $k+j$ of some error expression, where $j\ge1$ is related to
the degree of the first non-vanishing monomials in the expansion of
$F$ around the origin, the number of conditions on the coefficients
corresponding to monomials of degree~$k$ is larger than the number
of coefficients if the dimension of the manifold is at least~$2$. In
fact, the number of obstructions increases with the order~$k$. Of
course, it may happen that these obstructions vanish in some
particular examples (like several instances of the three body
problem, see~\cite{BFM2015a}), but generically they are unavoidable.

The cohomological equations for the terms of the approximate
solutions of~\eqref{eq:inv_equation} can be written as a linear PDE
of the form
\[
Dh(x) p(x) - Q(x) h(x) = w(x), \qquad x\in V\subset \RR^n,
\]
where $p$, $Q$ are fixed homogeneous functions that depend on the
first non-vanishing nonlinear terms of the Taylor expansion of~$F$
and $w$ is an arbitrary homogeneous function. Of course, the problem
lies in finding global solutions of this PDE. In this work we prove
that, under suitable hypotheses (see H1, H2, H3 and
\eqref{defconstants} in Section~\ref{sec:mainhypotheses}), the
cohomological equations have homogeneous solutions defined in the
whole domain under consideration. Their order is related to the
order of $w$. This result allows us to find the approximate
solutions of~\eqref{eq:inv_equation} as a sum of homogenous
functions of increasing order. In general, these functions are not
polynomials, not even rational functions. We deal with both the
differentiable and analytic cases. In the differentiable case there
may be a loss of regularity. It is also worth mentioning that the
regularity assumption needed for obtaining~$R$ and the approximation
are sufficient to deal with the second stage of the procedure. We
remark that our conditions allow several characteristic directions
in the domain under consideration (see~\cite{Hakim98,Abate15}).

The structure  of the paper is as follows. In
Section~\ref{sectionhypothesis} we present the hypotheses and main
results of the paper.  In Section~\ref{sectionexamples} we show that
our hypotheses are indeed necessary, that the loss of
differentiability can take place and remark the differences between
the case of one-dimensional and multidimensional parabolic
manifolds. In sections~\ref{subsectionstep2}
and~\ref{formalsolutionsection} we prove the main theorems.
Section~\ref{subsectionstep2} contains the study of the
cohomological equations used in the actual proof of the main
theorems in Section~\ref{formalsolutionsection}.
Section~\ref{sec:parameters} is devoted to the dependence with
respect to parameters.

\section{Main result}\label{sectionhypothesis}
The main result of this work deals with the computation of
approximations of stable manifolds of parabolic points, expressed as the range of a
function~$K$, in such a way that the invariance
condition~\eqref{eq:inv_equation}, $F\circ K - K\circ R=0$, is
satisfied up to a prefixed order (see
equation~\eqref{eq:inv_approx}). We will look for~$K$ and~$R$ as a
finite sum of homogeneous functions not necessarily polynomials.
Each term of these sums  is a homogeneous solution of a so called
cohomological equation. We are forced to look for homogeneous
solutions of the cohomological equations because, in this
multidimensional case $x\in\RR^n$ with $n>1$, as we will see in
Section~\ref{formalsolutionsection}, in general these equations do not admit
polynomial solutions. We also refer to the reader to
Section~\ref{sectionexamples} where several examples are studied.

In addition, we also study the dependence on parameters of the
solutions of the cohomological equations (see
Section~\ref{subsec:result:param}).

At the end of this section, we present the result about approximate
solutions of the invariance equation in the vector field case.

\subsection{Set up and general hypotheses}
\label{sec:mainhypotheses} The context we present here is the same
as the one in~\cite{BFM2015a}, which we reproduce for the convenience of the
reader.

Let $U\subset \RR^{n}\times\RR^{m}$ be an open set such that
$(0,0)\in U$ and let $ F:U \to \RR^{n+m}$ be a map of the form
\begin{equation}\label{defF}
 F(x,y) = \left ( \begin{array}{c} x+  p(x,y) +  f(x,y) \\
y+ q(x,y)+ g(x,y)\end{array}\right ),\qquad x\in \RR^n, \,y\in
\RR^m,
\end{equation}
where $ p$ and $ q$ are homogeneous polynomials of degrees $N\ge 2$
and $M\ge 2$ respectively, $D^l f(x,y)= \OO(\Vert
(x,y)\Vert^{N+1-l})$ and $D^l g(x,y)=\OO(\Vert (x,y)\Vert^{M+1-l})$
for $l=0,1$. Clearly $(0,0)$ is a fixed point of~$F$ and $DF(0,0) =
\Id$.

Since the degrees of $p$ and $q$, $N$ and $M$, respectively, need
not to be the same, we introduce
\begin{equation*}
L = \min\{M,N\}.
\end{equation*}

We denote by $\px (x,y)=x$ and $\py (x,y)=y$ the natural projections
and by $B_{\r}$ the open ball centered at the origin of radius
$\r>0$. However, to simplify notation, we will often denote the
projection onto a variable as a  subscript, i.e., $X_x := \pi_x X$.

Now we state the minimum hypotheses to guarantee that the
cohomological equations we encounter can be solved and consequently,
we are able to find approximate solutions up to the required order.

Given $V\subset \RR^n$ such that $0\in \partial V$ and $\r>0$, we
introduce
\begin{equation}\label{defVr}
\Vr = V\cap B_{\r}.
\end{equation}
In this paper we will say that $V\subset \RR^n$ is star-shaped with
respect to $0$ if $0\in \partial V$ and for all $x\in v$ and
$\lambda \in (0,1]$, $\lambda x \in V$.

Take $\r>0$, norms in $\RR^n$ and $\RR^m$ respectively and consider the following
constants:
\begin{equation}\label{defconstants}
\begin{aligned}
&\ap = - \sup_{ x\in \Vr} \frac{\Vert  x+  p( x,0) \Vert -\Vert
x\Vert }{\Vert  x\Vert^{N}},
&\qquad &\bp  =\sup_{ x\in \Vr} \frac{\Vert  p( x,0)\Vert}{\Vert  x\Vert^{N}}, \\
&\Ap= - \sup_{ x\in \Vr} \frac{\Vert \Id + D_x p( x,0) \Vert
-1}{\Vert  x\Vert^{N-1}},
&\qquad &\Bp = \sup_{ x\in \Vr} \frac{\Vert \Id -D_x p( x,0) \Vert -1}{\Vert  x\Vert^{N-1}}, \\
&
\Bq = -\sup_{ x\in \Vr} \frac{\Vert \Id -D_y q( x,0) \Vert-1}{\Vert  x\Vert^{M-1}}, &&\\
&\cp  =
\begin{cases}
\;\;\ap,  & \text{if}\;\; \Bq \leq 0,\\
\;\;\bp, &  \text{otherwise}
\end{cases},
& \qquad &\ddp  =
\begin{cases}
\;\;\ap,  & \text{if}\;\; \Ap\leq 0,\\
\;\;\bp, &  \text{otherwise},
\end{cases}
\end{aligned}
\end{equation}
where the norms of linear maps are the corresponding operator norms.
We emphasize that all these constants depend on $\r$

We assume that there exist an open set $V\subset \RR^{n}$, $V$
star-shaped with respect to $0$, and appropriate norms in $\RR^{n}$
and $\RR^{m}$ satisfying, taking $\r$ small enough,
\begin{enumerate}
\item[H1]
The homogenous polynomial $ p$ satisfies that
\[
\ap>0.
\]
If $M>N$, we further ask $\Ap/\ddp>-1$.
\item[H2]
The homogenous polynomial $ q$ satisfies $q( x,0)=0$ for
$x\in \Vr$, and
\begin{equation*}
\begin{aligned}
&D_yq(x,0) \text{  is invertible  }\forall x\in \overline{\Vr}\backslash\{0\}, & \qquad &\text{if  } M<N,\\
& 2+\frac{\Bq}{\cp} >\max \left \{1 -\frac{\Ap}{\ddp} ,0
\right\},&\qquad & \text{if  } M=N.
\end{aligned}
\end{equation*}
\item[H3]
There exists a constant $\CIn>0$ such that, for all $ x\in \Vr$,
\begin{equation*}
\text{dist}( x+p( x,0), (\Vr)^{c}) \geq \CIn \Vert  x\Vert^{N}.
\end{equation*}
\end{enumerate}

We emphasize that H1--H3 are asked to be satisfied not in a
neighborhood of the origin but in $\Vr$. As usual in the
parabolic case, a stable invariant manifold is defined over a
subset $V$ such that $0 \in \partial V$. It may happen that the
manifold is not defined in a neighborhood of the origin. However,
some regularity at the origin may be retained. For this reason we
introduce the following definition.
\begin{definition}\label{difrem}
Let $V\subset \RR^l$ be an open set, $x_0 \in \overline{V}$
and $f:V\cup\{x_0\}\subset \RR^l \to \RR^k$. We say that $f$ is
$C^1$ at $x_0$ if $f$ is $C^1$ in $V\cap (B_{\epsilon}(x_0)\setminus\{x_0\})$, for some $\varepsilon>0$ and $
\lim_{x\to x_0, \; x\in V} Df(x) $ exists.
\end{definition}

Finally we introduce some notation. Given $l,k,\ell\in \NN$ and an
open set $\mathcal{U}\subset \RR^{l}$ such that $0\in  \partial
\mathcal{U}\cup \mathcal{U}$, we define
\begin{align*}
\Hom{\ell}&=\{h\in \mathcal{C}^0(\mathcal{U},\RR^{k})\;:\;\;
\text{for}\;u \in \mathcal{U},\;\; \Vert h(u) \Vert = \OO\big(\Vert
u \Vert^{\ell}\big)\},
\\
\ho{\ell}&=\{h\in \mathcal{C}^0(\mathcal{U},\RR^{k})\;:\;\;
\text{for}\;u \in \mathcal{U},\;\;
\Vert h(u) \Vert = \oo\big(\Vert u \Vert^{\ell}\big)\}, \\
\Hog{\ell}&=\{h\in \mathcal{C}^0(\mathcal{U},\RR^{k})\;:\;\; \forall
\lambda\in \RR, \; \forall u\in \mathcal{U},\;\;\text{s.t.}\;\; \lambda u\in \mathcal{U},\;\; h(\lambda u)=\lambda^{\ell} h(u)\}.
\end{align*}
To simplify notation, we skip  the reference to $l, k$ and
$\mathcal{U}$, which will be fixed and clearly understood from the
context.

\subsection{Approximate solutions of the invariance equation for maps}
In this section we present two results. The first one is about the
existence of approximate solutions having the ``simplest" form. The
other one (which can be useful in some applications) is about the
freedom we have for solving the cohomological equations.

As we will prove in an algorithmic way, even when $F$ is an analytic
function, we can not, in general, obtain $\CC^{\infty}$
approximations of the stable manifold, unlike the hyperbolic case.
For instance, if~$\Ap<\ddp$ and $M\geq N$, we obtain $\CC^{\gdf}$-regularity of
these approximations, where~$\gdf$ is given by
\begin{equation}\label{defminimaregularitat}
\gdf = \begin{cases} \max\left \{\displaystyle{k\in \NN : \left(
1- \frac{\Ap}{\ddp} \right) }
k <2+\frac{\Bq}{\cp}\right\}, & \qquad \text{if  } M=N, \\
\max\left \{\displaystyle{k\in \NN : \left( 1- \frac{\Ap}{\ddp}
\right)} k < 2 \right \},& \qquad \text{if  } M>N.
\end{cases}
\end{equation}

\begin{theorem}\label{formalsolutionprop}
Let $ F:U\subset \RR^{n+m} \to \RR^{n+m}$ be  defined in a
neighborhood of the origin and having the form~\eqref{defF}. Assume
that $ F\in \CC^{r}$, with $r\geq N$, and satisfies hypotheses H1,
H2 and H3 for some  $\r_0>0$. Then, for any $N \leq \ell\leq r$
there exist $0<\r \leq \r_0$ and $ K:\Vr\to U$ and $ R:\Vr\to \Vr$ such that
\begin{equation}
 F\circ  K -  K\circ  R\in  \ho{\ell} \label{HIKl}.
\end{equation}
In addition, we can choose $ K$ and $ R$ as a finite sum of
homogeneous functions $\Kl{j}\in \Hog{j}$ and $\Rl{j} \in \Hog{j}$
(not necessarily polynomials), of the form
\begin{equation}\label{formf}
\begin{split}
K_x( x) &= x + \sum_{l=2}^{\ell-N+1} \Kl{l}_x( x), \qquad
K_y( x) = \sum_{l=2}^{\ell-L+1} \Kl{l}_y( x), \\
 R( x) &=  x + \sum_{l=N}^{\min\{\ell,\df\}} \Rl{l}( x)
\end{split}
\end{equation}
with $R^N(x) = p(x,0)$, $L=\min\{N,M\}$ and $\df$ defined by
\begin{equation}\label{degreefreedom}
\df=
\begin{cases}
N-1+\left[\frac{\Bp}{\ap}+\gdf \left(1-\frac{\Ap}{\ddp} \right
)\right], \qquad & \text{if $ \Ap<\bp$ and $M\geq N$}, \\
N-1+\left[\frac{\Bp}{\ap}\right], & \text{if $ \Ap\ge\bp$  and $M\geq N$}, \\
\ell, & \text{$M< N$}.
\end{cases}
\end{equation}

Moreover, $K$ and $R$ extend to $V$
by homogeneity of their terms. The functions $\Kl{l}_x( x)$, with $l=2,\cdots, \df-N+1$, can be chosen arbitrarily, in particular, equal to~$0$.

Concerning the regularity of the approximation of the
parametrization we have that $ K$  and $ R$ are $\CC^1$ at the
origin in the sense of Definition~\ref{difrem}. Finally,
\begin{enumerate}
\item[(1)] if either $\Ap>\ddp$ or $M<N$,  $ K,  R$ are analytic in
a complex neighborhood of~$V$,
\item[(2)] if $\Ap=\ddp$, $ K, R$ are $\CC^\infty$ functions on
$V$,
\item[(3)] if $\Ap<\ddp$ and $M\geq N$, $ K, R$ are $\CC^{\gdf}$ functions on $V$ where $\gdf$ is defined in~\eqref{defminimaregularitat}.
\end{enumerate}
\end{theorem}
\begin{remark}
We will see in Lemma~\ref{defKapwell} that $\Bp/\ap\geq N$. Indeed,
$\Bpa$ in that lemma corresponds to $-\Bp$ in~\eqref{defconstants}.
\end{remark}
\begin{remark}
In \cite{BFM2015a} it is proven that under the hypotheses H1, H2 and
H3, there exists an exact solution $\tilde{K}, R$ of the invariance
equation~\eqref{eq:inv_equation}. In addition, if $\Kl{\leq}$ is the
function provided by Theorem~\ref{formalsolutionprop} for some
$\ell$ big enough, then $K$ has the form $\Kl{\leq}+ \Kl{>}$ with
$\Kl{>}\in \ho{\ell-N+1}$. Even more, assuming that $\Ap,\Bq>0$ and
the hypotheses of the theorem, the stable set is a manifold which is
the graph of a differentiable function $\varphi$ which can be
approximated by $\py K \circ (\px K)^{-1}$ in the sense that
$\varphi- \py K \circ (\px  K)^{-1}\in \ho{\ell-L+1}$.
\end{remark}

\begin{remark}
As we will see in the proof of Theorem~\ref{formalsolutionprop} in
Section~\ref{sec:regularity}, we can choose different strategies in
order to get $ R$ as a sum of homogeneous functions of degree less
than $\df$. However, not for all strategies the obtained regularity
will be optimal.
\end{remark}

\begin{remark}
The results stated in~Theorem~\ref{formalsolutionprop} hold also
true if, instead of assuming that $F$ is a $\CC^r$ function in an
open neighborhood of the origin, we assume that $F$ can be written
as a sum of homogeneous functions which are $\CC^r$ in $V$, that is
$F$ has the form:
\begin{align*}
F_x(x,y) &=  x +  p( x,y)+  F_{ x}^{N+1}( x,y) + \cdots +   F_{ x}^{r}( x,y)+   F_{ x}^{>r}( x,y), \\
F_y( x,y) &= y+ q( x,y) +  F_{y}^{M+1}( x,y) + \cdots+
 F_{y}^{r}( x,y)+   F_{y}^{>r}( x,y),
\end{align*}
where all the functions are $\CC^r$ in $V$, $p\in \Hog{N}$, $q\in
\Hog{M}$, $F_{x}^j , F_{y}^j \in \Hog{j}$ and
$F_x^{>r},F_y^{>r} \in \ho{r}$.
\end{remark}

An alternative point of view is the following result:
\begin{theorem}\label{freeparameterstheorem}
Assume the same hypotheses of~Theorem~\ref{formalsolutionprop}.
Let $N \leq \ell\leq r$ and $\Kl{l}_x\in \Hog{l}$ for $l=2,\cdots,\ell-N+1$.
Then for any function $K_x:V\to \RR^n$ such that
$$
K_x( x)-  x - \sum_{l=2}^{\ell-N+1} \Kl{l}_x( x) \in \ho{\ell-N+1}
$$
satisfying the regularity statements for $K$ of Theorem~\ref{formalsolutionprop},
there exist $0\leq \r \leq \ro$ and  $R:\Vr\to \Vr$ and $K_y:\Vr\to \RR^m$ of the form
$$
R(x) = x + p(x,0) + \sum_{l=N+1}^{\ell} \Rl{l}(x)\qquad K_y( x)=
\sum_{l=2}^{\ell-L+1} \Kl{l}_y( x)
$$
with $R^l\in \Hog{l}, \Kl{l}_y\in \Hog{l}$, such that $F\circ  K -
K\circ  R\in  \ho{\ell}$ with $K=(K_x,K_y)$. Moreover the regularity
statements are the same as the ones in
Theorem~\ref{formalsolutionprop} and $K$ and $R$ can be extended to
$V$.
\end{theorem}
\subsection{Dependence on parameters}\label{subsec:result:param}
Let $\Lambda\subset \RR^{n'}$ be an open set of parameters,
$U\subset \RR^{n+m}$ be an open set and $V$ as in Section~\ref{sec:mainhypotheses}.
Assume that $F: U\times \Lambda\to \RR^{n+m}$ are maps having the form \eqref{defF} for any $\lambda
\in \Lambda$, i.e.:
\begin{equation}\label{defFparam}
 F(x,y,\lambda) = \left ( \begin{array}{c} x+  p(x,y,\lambda) +  f(x,y,\lambda) \\
y+ q(x,y,\lambda)+ g(x,y,\lambda)\end{array}\right ).
\end{equation}

For any fixed $\lambda\in \Lambda$ the constants
in~\eqref{defconstants} are well defined and depend on $\lambda$. We
denote this dependence with a superindex. As we did in
\cite{BFM2015a}, we redefine the constants $\Ap,\ap,$ etc. by taking
the supremum over $\Vr\times \Lambda$ instead of $\Vr$. For
instance,
$$
\Ap=\inf_{\lambda \in \Lambda} \Ap^{\lambda} = -\sup_{(x,\lambda)\in
\Vr \times \Lambda} \frac{\Vert \Id + D_x p( x,0,\lambda) \Vert
-1}{\Vert  x\Vert^{N-1}}.
$$

We note that, assuming H1, H2 and H3 for any $\lambda\in \Lambda$ we
already have the existence of approximate solutions $K_{\lambda}$.
To obtain uniform bounds, and therefore continuity and differentiability, with
respect to $\lambda\in \Lambda$ we need to assume
\begin{enumerate}
\item[H$\lambda$] Hypotheses H1, H2 and H3 hold true uniformly with respect to $\lambda$, namely, all the conditions involving the constants
$\ap, \bp, \Ap, \Bp, \ddp,\cp, \Bq,\CIn$ hold true with the new definition of
these constants.
\end{enumerate}

From now on we will abuse notation and we will write that a function
$h$ depending on a parameter $\mu$, belongs to $\Hom{\ell}$ if
$h(z,\mu)=\OO(\Vert z \Vert^{\ell})$ uniformly in $\mu$. Analogously
if $h\in \ho{\ell}$. Moreover, $h\in \Hog{\ell}$ will mean that $h$ is
homogeneous of degree $\ell$ for any fixed $\mu$.

The differentiability class we work with was introduced in
\cite{CabreFL03b} and is also used in~\cite{BFM2015a}. For any $s,r\in
\ZZ^+=\NN\cup \{0\}$, we define the set
\begin{equation*}
\DS = \big \{ (i,j) \in (\ZZ^+)^2 : i+j\leq r+s,\; i\leq s\big \}
\end{equation*}
and for $\mathcal{U}\subset \RR^l \times \RR^{n'}$, the function space
\begin{equation}
\label{defCDS}
\begin{aligned}
\CC^{\DS} = \big \{ f: \mathcal{U}\to \RR^{k} \;:\; & \forall
(i,j)\in \DS, \\ &D_{\mu}^i \Dz^j f \text{ exists, is continuous and bounded} \big\}.
\end{aligned}
\end{equation}

\begin{theorem}\label{prop:param}
Let $F \in \CC^{\DS}$ be of the form~\eqref{defFparam} with $r\geq
N$ satisfying H$\lambda$ for $\r_0>0$. Let $\ell \in \NN$ be $\ell
\leq r$ as in Theorem~\ref{formalsolutionprop}.

Then the functions $ K:V\times \Lambda \to \RR^{n+m}$ and $R:V\times
\Lambda \to \RR^n$ given by Theorem~\ref{formalsolutionprop} satisfy:
\begin{enumerate}
\item[(1)]  If either $\Ap>\ddp$ or $M<N$, $K,R$ are $\CC^s$ with respect to $\lambda\in \Lambda$ and real analytic with respect to $x\in V$. In addition, if $F$ depends analytically on $\lambda \in \Lambda$, the functions $ K, R$ are real analytic in $V\times \Lambda$.
\item[(2)]  If $\Ap =\ddp$ then $ K, R\in \CC^{\Sigma_{s,\infty}}$ in $V\times \Lambda$.
\item[(3)] If $\Ap<\ddp$ and $M\geq N$, then $ K, R \in \CC^{\Sigma_{s_*,r_*-s_*}}$ in $V\times \Lambda$ where $\gdf$ is defined in~\eqref{defminimaregularitat} and $s_*=\min\{s,\gdf\}$.
\end{enumerate}

If $K_x:V\times \Lambda \to \RR^n$ is of the form given in
Theorem~\ref{freeparameterstheorem} and satisfies the above
regularity statements, then the functions
$R:V\times \Lambda \to \RR^n$ and $K_y:V\times \Lambda \to \RR^m$
provided by Theorem~\ref{freeparameterstheorem} satisfy the same statements.
\end{theorem}

\subsection{Approximate solutions of the invariance equation for flows}
We deduce the analogous results to
Theorems~\ref{formalsolutionprop} and~\ref{prop:param} in the case
of time periodic flows. It is worth to mention that we could deduce
some results for flows from the previous ones using the
Poincar\'{e} map. Nevertheless we prefer to give explicit results
because, as we will see in Section~\ref{formalsolutionsection}, we
can construct the approximate solutions without computing neither
the Poincar\'e map nor the flow, which turns out to be very useful
in applications.

In the case of flows, to shorten the exposition, we deal with the
parametric case, being the free parameter case a straightforward
consequence.

Let $U \subset \RR^{n+m}$ be a neighborhood of the origin,
$\Lambda\subset \RR^{n'}$ and $ X: U \times \RR \times \RR^{n'}\to
\RR^{n+m}$ a $T$-periodic  vector field
\begin{equation}\label{XTperiodic}
\dot{z}=  X(z,t,\lambda) ,\qquad  X(z,t+T,\lambda) =  X(z,t,\lambda)
\end{equation}
such that
\begin{equation}\label{defX}
 X(z,t,\lambda) =  X(x,y,t,\lambda) = \left ( \begin{array}{c}  p(x,y,\lambda) +  f(x,y,t,\lambda)
\\  q(x,y,\lambda)+ g(x,y,t,\lambda)\end{array}\right ),\end{equation}
where $ p$ and $ q$ are homogeneous polynomials of degrees $N\ge 2$
and $M\ge 2$ respectively with respect to $(x,y)$, and
$f(x,y,t,\lambda)= \OO(\Vert (x,y)\Vert^{N+1})$ and $
g(x,y,t,\lambda)=\OO(\Vert (x,y)\Vert^{M+1})$ uniformly in
$(t,\lambda)\in \RR\times \Lambda$.

If we want to deal with the invariant manifolds of parabolic
periodic orbits, we translate the orbit to the origin and we get a
vector field of the form~\eqref{defX}.

From now on, in the case of flows, the spaces $\ho{\ell},
\Hom{\ell}, \Hog{\ell}$ will be the analogous to the ones in
Section~\ref{sec:mainhypotheses}, respectively
Section~\ref{subsec:result:param}, with a $T$-periodic dependence
on~$t$ and with uniform bounds with respect to $\lambda\in \Lambda$.

Let $\varphi(s;t_0, x,y,\lambda)$ be the flow of~\eqref{XTperiodic}.
The condition that the range of a function~$K$, depending on
$(x,t,\lambda)$, is invariant by the flow of the vector
field~\eqref{defX}, analogous to~\eqref{eq:inv_equation} for maps,
is
\begin{equation}
\label{eq:inv_equation_flow} \varphi(s;t,K(x,t,\lambda),\lambda) =
K(\psi(s;t,x,\lambda),s,\lambda),
\end{equation}
for some function $\psi$. In the above equation the unknowns are~$K$
and~$\psi$. However, if $\psi(s;t,x,\lambda)$ is the flow
associated to some vector field $Y(x,t,\lambda)$, the invariance
equation~\eqref{eq:inv_equation_flow} is equivalent to its
infinitesimal version
\begin{equation} \label{homequationflow}
X(K(x,t,\lambda),t,\lambda) = \Dx K(x,t,\lambda)   Y(x,t,\lambda) +
\partial_t K(x,t,\lambda),
\end{equation}
where $\Dx$ denotes the derivative with respect to $x$.

Next theorem asserts that equation~\eqref{homequationflow} can be
solved up to certain order using functions belonging to
$\CC^{\Sigma_{s',r'}}$ for some $s'$ and $r'$. For technical reasons
we will consider separately the differentiability with respect to
$(x,y)$ and $(t,\lambda)$. That is, in the definition~\eqref{defCDS}
of $\CC^{\DS}$ we take $z=(x,y)$ and $\mu=(t,\lambda)$.

\begin{theorem}\label{maintheoremflow}
Let $ X:U \times \RR \times \Lambda \to \RR^{n+m}$ be a vector field
of the form \eqref{defX} with $U$ an open neighborhood of the
origin. Assume that $X\in \CC^{\DS}$ and it satisfies Hypothesis
H$\lambda$ for some $\r_0>0$ and $V$ as in Section~\ref{sec:mainhypotheses}.

Then, for any $N \leq \ell\leq r$ there exist $0<\r \leq \ro$, $ K:\Vr\times \RR\times
\Lambda\to U$, $T$-periodic with respect to~$t$,  and $  Y:\Vr\times
\Lambda \to \RR^n$ such that
\begin{equation}
\label{HIklflow} X(K( x,t,\lambda),t,\lambda)- \Dx  K( x,t,\lambda)
Y( x,\lambda) - \partial_t  K( x,t,\lambda)\in \ho{\ell} .
\end{equation}
In addition, we can choose $ K$ and $  Y$ as a finite sum of
homogeneous functions $\Kl{j}\in \Hog{j}$ and $\Yl{j} \in \Hog{j}$
with respect to~$x$ (not necessarily polynomials), of the form
\begin{equation*}
\begin{split}
K_x( x,t,\lambda) &= x + \sum_{l=2}^{\ell} \Kl{l}_x(x,t,\lambda), \qquad
K_y( x,t,\lambda) = \sum_{l=2}^{\ell} \Kl{l}_y( x,t,\lambda), \\
  Y( x,\lambda) &= \sum_{l=N}^{\min\{\ell,\df\}} \Yl{l}( x,\lambda)
\end{split}
\end{equation*}
with $  Y^N(x,\lambda) = p(x,0,\lambda)$, $L=\min\{N,M\}$ and $\df$
defined in \eqref{degreefreedom}. The functions $\Kl{l}_x(x,\lambda)$,
with $l=2,\cdots, \df-N+1$, can be chosen arbitrarily, in particular, equal to~$0$.
Moreover $K$ and $Y$ can be extended to $V$ by homogeneity.

Concerning regularity we have that $ K$  and $  Y$ are $\CC^1$ at
the origin in the sense of Definition~\ref{difrem} . Finally,
\begin{enumerate}
\item[(1)] If either $\Ap>\ddp$ or $M<N$,  $ K,  Y$ are real analytic with respect to $x$ and $\CC^s$ with respect to $(t,\lambda)$. In addition, if $X$ depends
analytically on $(t,\lambda)\in \RR\times \Lambda$, then $K,Y$ are
real analytic in $V\times \RR\times \Lambda$,
\item[(2)] If $\Ap=\ddp$, $ K, Y$ are $\CC^{\infty}$ with respect to $x$ and $\CC^s$ with respect to $(t,\lambda)$. Moreover, if
$X\in \CC^{\infty}$, then also $ K, Y \in \CC^{\infty}$.
\item[(3)] If $\Ap<\ddp$ and $M\geq N$, $ K, Y$ belong to $\CC^{\Sigma_{s_*,\gdf-s_*}}$ with $s_* = \min\{s,\gdf\}$ and $\gdf$ defined in \eqref{defminimaregularitat}.
\end{enumerate}
\end{theorem}

\begin{remark}
Notice that the vector field $  Y$ can be chosen as a \emph{finite} sum of homogeneous functions independent of~$t$.
\end{remark}

The rest of this paper is devoted to prove all these results. We
first deal with the map case in the non parametric setting. In
Section~\ref{subsectionstep2} we study the existence and regularity
of global homogeneous solutions of a partial differential equation
which is a model for all the cohomological equation we need to
solve. Then, we prove Theorems~\ref{formalsolutionprop},
\ref{freeparameterstheorem} and~\ref{maintheoremflow} by following
an induction procedure with respect to the degree of
differentiability. After that we deal with the dependence with
respect to parameters. Finally we provide several examples to
illustrate that our hypotheses are necessary to obtain approximate
solutions and our results are (in some sense) optimal.

\section{The cohomological equation}\label{subsectionstep2}
Let $V\subset \RR^n$ be an open set, star-shaped with respect to $0$
and $\pa:\RR^n\to \RR^k$, $\Qa:\RR^n \to \L{}(\RR^k,\RR^k)$ and
$\T:V\to \RR^k$ be such that $\pa \in \Hog{N}$, $\Qa \in \Hog{N-1}$,
$\T\in \Hog{\m+N}$ with $N\geq 2$ and $\m \geq 1$.

Note that $\pa,\Qa$ are determined by their restriction to an
arbitrary small neighborhood $U$ of the origin. In particular if
they have some degree of regularity in $U$ they have the same
regularity in the whole space.

The linear partial differential equation
\begin{equation}\label{modellinearequation}
D   h( x) \cdot \pa( x) - \Qa( x) \cdot   h( x) = \T( x)
\end{equation}
for $  h:V  \to  \RR^k$ appears when we try to find approximations
of $K$ and $R$ as sums of homogeneous functions. We are interested
in solutions $ h \in \Hog{\m+1}$.

Let $\Vro$ be defined as in \eqref{defVr}. Along this section we
assume the following conditions for some $\ro>0$:
\begin{itemize}
\item[HP1] $\pa$ is $\CC^1$ in $\Vro$ and
\begin{equation}
\label{defKAS}
\apa =-\sup_{ x \in \Vro} \frac{\Vert   x + \pa ( x)
\Vert -\Vert  x \Vert}{\Vert  x \Vert^{N}}>0.
\end{equation}
\item[HP2] There exists a constant $ \CIP>0$ such that
\begin{equation*}
\text{dist}\big( x + \pa( x), \big (V_{\r_0}\big )^c\big )\geq \CIP
\Vert
 x \Vert^{N},\qquad \forall  x \in \Vro.
\end{equation*}
\end{itemize}
In the applications in this paper, $\pa$ and $\Qa$ will be
polynomial functions.

\begin{remark} If hypotheses HP1 and HP2 are satisfied for some $\ro$, then they also hold for $0<\r<\ro$.
As a consequence, we are always allowed to consider $\r$ small
enough (see Lemma~\ref{remarkrhosmall}).
\end{remark}

We define the constants~$\bpa,\Apa,  \BQ$, $\AQ$, $\cpa$ and $\dpa$
by,
\begin{equation}
\label{defconstantspa}
\begin{aligned}
&\bpa = \sup_{ x \in V_{\r_0}} \frac{\Vert \pa( x) \Vert}{\Vert
x\Vert^{N}}, &\qquad & \Apa= -\sup_{ x \in V_{\r_0}} \frac{\Vert
\text{Id}+ D\pa ( x)\Vert-1}{\Vert  x \Vert^{N-1}},
\\
&\BQ = -\sup_{ x \in V_{\r_0}} \frac{\Vert \text{Id}- \Qa (
x)\Vert-1}{\Vert  x \Vert^{N-1}} , &\qquad &
\AQ = \sup_{ x \in V_{\r_0}}\frac{ \Vert \text{Id} + \Qa( x) \Vert
-1}{\Vert  x \Vert^{N-1}}, \\
&\cpa  =
\begin{cases}
\;\;\apa,  & \text{if}\;\; \BQ\leq 0,\\
\;\;\bpa, &  \text{otherwise}.
\end{cases}
&\qquad&
\dpa= \begin{cases} \;\; \apa ,  & \text{if}\;\; \Apa<0,\\
\;\;\bpa, &  \text{otherwise}.
\end{cases}
\end{aligned}
\end{equation}

Next we introduce two ordinary differential equations which will
play a key role in the proof of the results of this section. The
first one is
\begin{equation}\label{edotau}
\frac{d  x}{d  t}=  \pa ( x).
\end{equation}
We denote by $\fpa( t, x)$ its  flow. The second one is the
homogeneous linear equation
\begin{equation}\label{edoQ}
\frac{d \fqa}{d  t}( t, x)  = \Qa(\fpa( t, x)) \fqa( t, x)
\end{equation}
and we denote by $M( t, x)$ its fundamental matrix such that $M(0,
x)=\Id$.

Using uniqueness of solutions of \eqref{edotau} and homogenity,
\begin{equation}\label{homogeneitatphi}
\fpa( t,\lambda  x) = \lambda \fpa (\lambda^{N-1}  t,  x),\qquad M(
t,\lambda  x) = M(\lambda^{N-1}  t,  x)
\end{equation}
wherever they are defined.

In order to deal with the analytic case, we define the norm $\Vert
\cdot \Vert$ in $\C^n$ as
\begin{equation*}
\Vert  x \Vert = \max\{ \Vert \re  x \Vert , \Vert \im  x \Vert \}.
\end{equation*}
We define complex extensions of $V$ and $\Vr$:
\begin{align*}
\Omega(\gamma)&:=\{x \in \C^n : \re x \in V, \;\; \Vert \im x
\Vert < \gamma \Vert \re x \Vert \}, \\
\Omega(\r,\gamma)&:= \{x \in \C^n : \re x \in \Vr, \;\; \Vert \im x
\Vert <\gamma \Vert \re x \Vert \}.
\end{align*}
Our analytical results will be over solutions defined on a complex
set $\Omega(\gamma)$ with a suitable choice of $\gamma$. We note that, if $ x\in \Omega(\gamma)$ with
$\gamma\leq 1$, then $\Vert  x \Vert =  \Vert \re  x\Vert$. We will
use this fact along this work without explicit mention.

\begin{theorem}\label{solutionlinearequation}
Let $\pa\in \Hog{N}$  and $\Qa\in \Hog{N-1}$ be defined in $\RR^n$
and $\T\in \Hog{\m+N}$ defined on an open set $V$ star-shaped with
respect to $0$, with $N\geq 2$ and $\m\geq 1$. Assume that $\pa$
satisfies hypotheses HP1 and HP2, for some $\r_0>0$ that $\pa,\Qa$
are $\CC^r$, $r\geq 1$, in $U$ and $\T$ is a $\CC^r$ function in
$V$.

Then, if
\begin{equation}\label{hipgraus}
\m+1+\frac{\BQ}{\cpa} > \max\left \{1-\frac{\Apa}{\dpa},0\right \},
\end{equation}
there exists a unique solution $  h\in \Hog{\m+1}$ of
equation~\eqref{modellinearequation} which is given by:
\begin{equation}\label{defh0}
  h( x) = \int_{\infty}^0 M^{-1}( t, x) \T(\fpa( t, x)) \,
d t,\qquad  x \in V.
\end{equation}
Moreover it is of class $\CC^1$ on $V$.

Concerning its regularity we have the following cases:
\begin{enumerate}
\item[(1)]
$\Apa \geq \dpa$. If $1\leq r\leq \infty$, then $  h$ is $\CC^r$ in $V$.
\item[(2)]
$\Apa<\dpa$. Let $r_{0}$ be the maximum of $1 \leq i \leq r$ such
that
\begin{equation}
\label{difcondformal} \m + 1+\frac{\BQ}{\cpa}-i \left (
1-\frac{\Apa}{\dpa}\right)>0.
\end{equation}
Then $  h$ is $\CC^{r_{0}}$ in $V$.
\item[(3)] $\Apa>\dpa$. If $\pa, \Qa,\T$ are real analytic functions in $\Omega(\gamma_0)$
for some $\gamma_0$ then $  h$ is analytic in $\Omega(\gamma)$ for
$\gamma$ small enough. In particular it is real analytic in $V$.
\end{enumerate}
\end{theorem}
\begin{remark}\label{rem:exthom}
By Hypothesis HP2, $\Vro$ is positively invariant by the flow $\fpa$
(see Lemma~\ref{invflux}) but it may happen that $V$ is not. However
since $V$ is star-shaped with respect to the origin, $V\subset
V^{{\rm e}}_{\r_0} =\{tx \;:\; t>0,\; x\in \Vro\}$, $V^{{\rm
e}}_{\r_0}$ is positively invariant by $\fpa$ and the
formula~\eqref{defh0} makes sense with $\T$ understood as the unique
extension of $\T$ to $V^{{\rm e}}_{\r_0}$ by homogeneity.
\end{remark}
\begin{remark}
We notice that the condition $\m+1+\frac{\BQ}{\cpa} > \max\big
\{1-\frac{\Apa}{\dpa},0\big \}$ is automatically satisfied if $\BQ,
\Apa\geq 0$.
\end{remark}
\begin{corollary}\label{cor:theoremlinearequation}
Assume the conditions of Theorem~\ref{solutionlinearequation}. Let
$\nu\in \NN$. If $\nu +\BQ/\cpa \geq 0$, then
equation~\eqref{modellinearequation} has a solution $  h:V\to \RR^k$
belonging to $\Hog{\nu}$, if and only if the integral
$$
\int_{\infty}^0 M^{-1}( t, x) \T(\fpa( t, x)) \, d t
$$
is convergent for $x\in V$.
\end{corollary}
We postpone the proof of these results to
Section~\ref{subsectionproofTheoremauxiliaryequation}. First we establish
some preliminary estimates.
\subsection{Preliminary facts}
This section deals with some basic facts that will be used
henceforth without mention.

\begin{lemma}\label{defKapwell}
The constants $\Apa, \BQ, \apa$, $\bpa$ and $\AQ$ are finite. They
satisfy $|\apa|\leq \bpa$,  $\apa\geq \Apa/N$, $\BQ\leq \AQ$ and
$-\Bpa\geq N \apa>0$.
\end{lemma}
\begin{proof}
The triangular inequality and the homogeneous character of $\pa$ and
 $\Qa$ imply that the constants are finite.
Relation $|\apa|\leq \bpa$ is also a consequence of the triangular
inequality.

From the definition of~$\Apa$, we have that
\begin{align}
\label{KAp>Ap} \Vert  x + \pa( x) \Vert
&\leq \Vert  x \Vert \int_{0}^1 \Vert \text{Id} +  D\pa(\lambda x )\Vert \, d\lambda \leq \Vert  x \Vert \int_{0}^1 \left(1-
\Apa\lambda^{N-1}\Vert  x
\Vert^{N-1} \right )\, d\lambda \notag \\
 &= \Vert  x \Vert \left(1- \frac{\Apa}{N} \Vert  x \Vert^{N-1}\right ),
\end{align}
therefore $\apa\geq \Apa/N$.

As for  $\AQ$ and $\BQ$, we notice that
\begin{equation*}
\Vert \Id -\Qa( x)^2\Vert \leq \Vert \Id + \Qa( x)\Vert \cdot \Vert
\Id - \Qa( x)\Vert \leq \big (1+\AQ\Vert  x \Vert^{N-1}\big)
  \big (1-\BQ\Vert  x \Vert^{N-1}\big ).
\end{equation*}
Since $\Vert \Id -\Qa( x)^2\Vert \ge 1 -\Vert \Qa( x)\Vert^2$, there
exists some constant $K> 0$ such that
\begin{equation*}
1-K\Vert  x \Vert^{2(N-1)} \leq 1 - (\BQ-\AQ) \Vert  x \Vert^{N-1}
-\AQ \BQ \Vert  x \Vert^{2(N-1)}.
\end{equation*}
Then, $\BQ-\AQ\leq (K-\AQ \BQ) \Vert  x\Vert^{N-1}$ and we get
$\BQ-\AQ\le 0$ taking $x\to 0$.

For the last claim, we note that, as we prove in~\eqref{KAp>Ap},
$$
\Vert x -\pa(x)\Vert \leq \Vert  x \Vert \left(1- \frac{\Bpa}{N}
\Vert  x \Vert^{N-1}\right ).
$$
Since $\Vro$ is invariant we apply the above inequality to
$x+\pa(x)$ and we obtain:
\begin{equation}\label{Bp>ap}
\Vert x +\pa(x) - \pa(x+\pa(x)) \Vert \leq \Vert x+\pa(x) \Vert
\left(1- \frac{\Bpa}{N} \Vert  x +\pa(x)\Vert^{N-1}\right ).
\end{equation}
We note that
$\Vert x+\pa(x) \Vert \leq \Vert x \Vert \big(1-\apa \Vert x \Vert^{N-1}\big )$. Hence, by~\eqref{Bp>ap}
$$
\Vert x +\pa(x) - \pa(x+\pa(x)) \Vert\leq \Vert x \Vert \left(1- \left(\apa +\frac{\Bpa}{N}\right) \Vert  x \Vert^{N-1} + K_2 \Vert x \Vert^{2N-2}\right ).
$$
In addition
\begin{equation*}
\Vert x + \pa(x) -\pa(x+\pa(x))\Vert =\left \Vert x -\int_{0}^1
D\pa(x+s\pa(x)) \pa (x) \, ds \right \Vert \geq \Vert x \Vert
\big(1-K_1\Vert x \Vert^{2N-2}\big ).
\end{equation*}
Then, again from~\eqref{Bp>ap}, taking $K=K_1+K_2$ we obtain
$$
-K \Vert x \Vert^{N-1} \leq -\apa -\frac{\Bpa}{N}
$$
which gives the result taking $x\to 0$.
\end{proof}

The following lemma assures that we can take $\r$ as small as we
need.
\begin{lemma}\label{remarkrhosmall}
Let $0<\overline{\r}<\r$. Denoting by $\overline{\Apa},
\overline{\apa},\overline{\bpa},\overline{\AQ}, \overline{\BQ}$ the
values of $\Apa,\apa,\bpa, \AQ, \BQ$ corresponding to
$\overline{\r}$, we have that
$$
\overline{\Apa} \geq \Apa,\;\;\overline{\apa}\geq \apa,\;\;\;
\overline{\bpa}=\bpa,\;\; \overline{\AQ}\leq  \AQ,\;\;
\overline{\BQ}\geq \BQ.
$$
Then, for $x\in V_{\overline{\r}}$,
\begin{align*}
&\Vert \Id + D\pa( x) \Vert \leq 1- \Apa \Vert  x \Vert^{N-1},\qquad
\Vert  x + \pa( x)\Vert \leq  \Vert  x \Vert (1- \apa \Vert  x
\Vert^{N-1}),\\
&\Vert \Id +\Qa( x)\Vert \leq  1+ \AQ \Vert  x \Vert^{N-1},\qquad
\Vert \Id -\Qa( x)\Vert \leq  1- \BQ \Vert  x \Vert^{N-1}.
\end{align*}
In addition, if HP1 and HP2 are satisfied for $\r>0$, they are also
satisfied for all $0<\overline \r < \r$.
\end{lemma}
\begin{proof}
Indeed, let $\overline{\r}<\r$. The relations among the constants
follow from the fact that $V_{\overline{\r}}\subset \Vr$ and (only for $\bpa$) $\pa$
is a homogeneous function. Notice that
$\bpa$ does not depend on $\r$. Hence HP1 is satisfied for
$\overline{\r}$. Now we deal with HP2. Let $ x \in
V_{\overline{\r}}$ and let $z \in
\partial V_{\overline{\r}}$ be such that
\begin{equation*}
\text{dist}\big ( x + \pa( x), (V_{\overline{\r}})^c\big ) = \Vert
 x + \pa( x)-z\Vert.
\end{equation*}
We have two possibilities: either $z \in \partial \Vr$ or $z\in \Vr$
and $|z|=\overline{\r}$. If $z\in \partial \Vr$, then since $ x \in
\Vr$, by HP2 we have $\Vert  x + \pa( x)-z\Vert \geq \CIn \Vert  x
\Vert^N$. Finally, if $|z| = \overline{\r}$ we have that $z=\lambda(
x + \pa( x))$ with $\lambda=\overline{\r} \Vert  x + \pa(
x)\Vert^{-1}$ and by HP1 and the definition of $\apa$
in~\eqref{defKAS},
\begin{equation*}
\Vert  x + \pa( x)-z\Vert = \overline{\r} - \Vert  x + \pa( x)\Vert
\geq \Vert  x \Vert - \Vert  x + \pa( x)\Vert  \geq \apa \Vert  x
\Vert^{N}.
\end{equation*}
\end{proof}

Next lemma will be used in the analytical case.

\begin{lemma}\label{invomega}
Let $\r,\gamma >0$.
\begin{enumerate}
\item[(1)]
If $ x \in \Omega(\r,\gamma)$ and
$\chi:\Omega(\r,\gamma) \to \C^n$ is a real analytic function
belonging to $\Hog{\ell}$ then
\begin{equation*}
 \chi ( x)  = \chi(\re x) + i D_{x} \chi (\re
x) \im  x + \gamma^2 \OO(\Vert  x \Vert^{\ell}).
\end{equation*}
\item[(2)]
If HP2 is satisfied and $\Apa>\bpa$, then there exists $\gamma_0\in
(0,1)$ such that for any $0<\gamma\leq \gamma_0$, the complex set
$\Omega(\r_0,\gamma)$ is an invariant set for the map $ x \mapsto  x
+ \pa( x)$.
\end{enumerate}
\end{lemma}
\begin{proof}
Item (1) follows from Taylor's theorem, Cauchy-Riemann
equations and the fact that $\chi$ is a real analytic function.

A property similar to~(2) was proven in~\cite{BF2004}.
From~(1),  if $ x \in \Omega(\r,\gamma)$,
\[
 x + \pa( x) =  x + \pa(\re x ) + i D \pa(\re  x)
 \im x+ \gamma^2 \OO(\Vert  x \Vert^{N}).
\]
On the one hand we have that, by hypothesis HP2,
\begin{equation}
\label{reinv}
\text{dist}(\re \big( x + \pa( x)\big ), \Vro^c) \geq \CIP \Vert  x
\Vert^{N} - \gamma^2 \OO(\Vert  x \Vert^N) >0
\end{equation}
which implies that $\re\big( x + \pa( x)\big ) \in \Vro$ and on the
other hand, using~\eqref{reinv} and the definitions of $\Apa$ and $\bpa$,
we have
\begin{equation*}
\Vert \im \big ( x  + \pa( x) \big )\Vert - \gamma \Vert \re \big( x
+ \pa( x)\big ) \Vert \leq \gamma  (\bpa -\Apa +\OO(\gamma))\Vert
\re  x \Vert^{N}<0
\end{equation*}
provided $\gamma$ is small enough.
\end{proof}

\subsection{Properties of $\fpa( t, x)$ and $M( t, x)$}

In this section we describe some properties of the solutions of
equations \eqref{edotau} and \eqref{edoQ}. We will denote by $K$ a
generic positive constant, which may take different values at
different places. Also let
\begin{equation*}
\a=\frac{1}{N-1}.
\end{equation*}

\begin{lemma}\label{invflux} Assume hypotheses~HP1 and HP2 for $\ro>0$. Then:
\begin{enumerate}
\item[(1)]
There exists $\r_1 \leq \ro$ such that for all $0<\r\leq \r_1$, $\Vr$ is positively invariant by the flow $\fpa$.
\item[(2)]
Assume that $\Apa>\bpa$ and that $\pa$ has an analytic extension to $\Omega(\gamma_0)$ for some $0<\gamma_0\leq 1$.
Then there exist $0< \r_1\leq \ro$ and $0<\gamma_1 \leq \gamma_0$
such that for any $0<\r\leq \r_1$ and $0\leq \gamma \leq  \gamma_1$,
the set $\Omega(\r,\gamma)$ is invariant by the complexified flow,
i.e. $\fpa( t, x)\in \Omega(\r,\gamma)$, for $t>0$ and $ x \in
\Omega(\r,\gamma)$.
\end{enumerate}
\end{lemma}
\begin{proof}
We first prove item~(2) . Since $\fpa( t,0)\equiv 0$
for all $t$ and $\fpa$ is $\CC^1$, we have that, for some
$\gamma\geq 0$ and $\r$ small enough
\begin{equation}\label{cotafpa0}
\Vert \fpa( t, x) \Vert \leq K \Vert  x \Vert, \qquad t\in[0,1],\;
x\in \Omega(\r,\gamma).
\end{equation}
By Taylor's theorem,
\begin{equation}\label{Taylorfpa}
\fpa( t, x) =  x +  t \pa( x) + \int_{0}^{ t} (t-s) D\pa(\fpa(s, x))
\pa(\fpa(s, x))\, ds
\end{equation}
and using that $\pa \in \Hog{N}$, \eqref{cotafpa0}
and~(1) of Lemma~\ref{invomega} for $\chi=\pa$, we get
for $0\leq t \leq 1$
\begin{equation}\label{cotarefpa3}
\Vert \re \fpa( t, x) - \big (\re  x +  t \pa(\re  x)\big )\Vert
\leq  \gamma^2  K \Vert  x \Vert^{N} t  + K \Vert  x \Vert^{2N-1}
t^2.
\end{equation}

Let $ x \in \Omega(\r,\gamma)$. The fact that $\re  x\in \Vr$,
\eqref{cotarefpa3} and HP2 imply that
\begin{align*}
\text{dist}\big (\re \fpa(1, x),(\Vr)^c\big) \geq &\text{dist}\big
(\re  x+\pa(\re  x), (\Vr)^c\big ) \\ &- \Vert \re x + \pa(\re  x) - \re
\fpa(1, x)\Vert
\\ \geq &\text{dist}\big (\re  x +\pa(\re x), (\Vr)^{c}\big ) - \gamma^2 K \Vert  x \Vert^{N} - K\Vert  x \Vert^{2N-1} \\
 \geq &\CIP \Vert  x \Vert^{N} -  \gamma^2 K \Vert  x \Vert^{N}- K\Vert  x \Vert^{2N-1} \geq \frac{\CIP}{2} \Vert  x \Vert^{N}
\end{align*}
if$\r,\gamma$ are small enough. We have proven that if $ x \in
\Omega(\r,\gamma)$ then $\re \fpa(1, x) \in \Vr$.

From \eqref{homogeneitatphi}, taking $t=1$ and then
$\lambda= t^{\a}$ with $ t\in (0,1]$ and $ x \in \Omega(\r,\gamma)$,
\begin{equation*}
\fpa( t, x) =   t^{-\a} \fpa (1,  t^{\a}  x).
\end{equation*}
Since $ t^{\a}  x \in \Omega( \r,\gamma)$ if $ x \in
\Omega(\r,\gamma)$, we already know that~$\re \fpa( t, x) \in V$.
Moreover, by \eqref{cotarefpa3}, taking $\r,\gamma$ small enough
and using that $\Vert \re x\Vert =\Vert x\Vert$,
\begin{equation*}
\Vert \re \fpa(1, t^{\a}  x ) \Vert \leq \Vert  t^{\a}  x \Vert
\big (1-  t \apa \Vert  x \Vert^{N-1} + K  t \gamma^2 \Vert  x
\Vert ^{N-1} + K  t^2 \Vert  x \Vert^{2(N-1)} \big ) \leq  t^{\a}
\Vert  x \Vert,
\end{equation*}
and consequently $\Vert \re \fpa( t, x) \Vert \leq \Vert   x \Vert
=\Vert \re  x \Vert \leq \r$. This implies that $\re \fpa( t, x) \in
\Vr$ if $ t\in [0,1]$. Now, from identity \eqref{Taylorfpa}, using~(1)
of Lemma~\ref{invomega} and the definitions of
$\bpa$ and $\Apa$, we deduce that
\begin{align}
\Vert \re \fpa( t, x)\Vert &\geq \Vert \big (\re  x +  t \pa(\re
x)\big ) \Vert  -  \gamma^2 K \Vert  x \Vert^{N} t-  K \Vert  x
\Vert^{2N-1} t^2
\notag \\
&\geq \Vert \re  x \Vert (1-  t\bpa \Vert \re  x \Vert^{N-1}) -
\gamma^2  K \Vert  x \Vert^{N} t-  K \Vert  x \Vert^{2N-1} t^2,
\label{cotarefp4} \\
\Vert \im \fpa( t, x) \Vert &\leq  \Vert \big (\Id +  t D_{ x} \pa (\re  x) \big) \im  x \Vert +  \gamma^2 K \Vert  x \Vert^{N} t+  K \Vert  x \Vert^{2N-1} t^2 \notag \\
&\leq \Vert \im  x \Vert (1 -  t\Apa \Vert \re  x \Vert^{N-1} ) +
\gamma^2 K \Vert  x \Vert^{N} t+  K \Vert  x \Vert^{2N-1} t^2.
\notag
\end{align}
Therefore, since $\Apa>\bpa$, taking $\r,\gamma$ small enough,
\begin{equation*}
\gamma \Vert \re \fpa( t, x)\Vert - \Vert \im \fpa( t, x) \Vert \geq 0.
\end{equation*}
As a consequence $\fpa( t, x) \in \Omega(\r,\gamma)$ for all $ t\in
[0,1]$. Finally we extend this property to $ t>1$ by using
inductively that $\fpa( t, x)= \fpa(1,\fpa( t-1, x))$. Note that in
this part we have not to reduce the values of $\r,\gamma$.

A shorter but completely analogous argument
proves~(1) assuming neither that $\pa$ is analytic
nor $\Apa>\bpa$.
\end{proof}

\begin{lemma}\label{cotafpa}Assume that HP1 and HP2 are satisfied for some $\r_0>0$.
Let $0< a\le \apa$ and $ b \ge \bpa$. Then, for any $ t\geq 0$ and $
x \in V$,
\begin{equation*}
\frac{\Vert  x \Vert}{\big (1+ (N-1)  b  t \Vert  x \Vert^{N-1}\big
)^{\a}} \leq \Vert \fpa ( t,  x) \Vert \leq \frac{\Vert  x
\Vert}{\big (1+ (N-1)  a  t \Vert  x \Vert^{N-1}\big )^{\a}}.
\end{equation*}

If  $\Apa>\bpa$ and $\pa$ has an analytic extension to
$\Omega(\gamma_0)$ for some $\gamma_0\leq 1$, for any $0< a<\apa$
and $ b > \bpa$ there exists $\gamma\leq \gamma_0$ such that for $
t\geq 0$, $\fpa$ is analytic in $\Omega(\gamma)$ and the previous
bounds are true for $ x \in \Omega(\gamma)$.
\end{lemma}
\begin{proof}
The definitions of $\apa$ and $\bpa$ in \eqref{defKAS} and
\eqref{defconstantspa}, respectively, imply that for any $ x \in
\Vr$ and $ t \in [0,1]$,
\begin{equation}\label{cotafpa1}
\Vert  x \Vert \big (1 -  t \bpa \Vert  x \Vert^{N-1}\big)\leq \Vert
x +  t \pa( x) \Vert \leq \Vert  x \Vert \big (1 -  t \apa \Vert  x
\Vert^{N-1}\big).
\end{equation}
Indeed, the inequality involving $\bpa$ follows from the triangular
inequality. For the right hand side inequality, let $ x \in \Vr$.
Since $\Vr$ is a star-shaped set, for any $t\in (0,1]$, $t^{\a}
x\in \Vr$ and hence,
\begin{equation*}
-\apa \geq \frac{\Vert t^{\a}  x +\pa(t^{\a} x) \Vert -\Vert t^{\a}
x \Vert }{\Vert t^{\a}  x \Vert^{N}} = \frac{\Vert  x + t^{\a(N-1)}
\pa( x) \Vert -\Vert  x \Vert}{t^{\a(N-1)} \Vert  x \Vert^{N}}.
\end{equation*}
The result follows because $\a(N-1)=1$.

Let now $ x \in \Omega(\r,\gamma)$, where $\rho$ and $\gamma$ given
by Lemma~\ref{invflux}. The real case, $ x \in \Vr$, is obtained
taking $\gamma=0$. By Lemma~\ref{invflux}, $\fpa( t, x) \in
\Omega(\r,\gamma)$ and hence $\Vert \fpa( t, x)\Vert =\Vert \re
\fpa( t, x)\Vert$. Then from \eqref{cotarefp4},
\begin{equation*}
\Vert \fpa( t, x)\Vert \geq \Vert  x \Vert \big (1- \bpa t \Vert   x
\Vert^{N-1} - t \gamma^2 K \Vert  x \Vert^{N-1} - t^2 K \Vert  x
\Vert^{2N-2}\big)
\end{equation*}
and from \eqref{cotarefpa3} and \eqref{cotafpa1}
\begin{equation*}
\Vert \fpa( t, x)\Vert \leq \Vert  x \Vert \big (1- \apa t \Vert   x
\Vert^{N-1} + t \gamma^2 K \Vert  x \Vert^{N-1} + t^2 K \Vert  x
\Vert^{2N-2}\big).
\end{equation*}
To obtain the bound for $\Vert \fpa( t, x)\Vert$, $ t\in [0,1]$, we
only have to take into account that, since $\apa> a$ and $\bpa<  b$,
if $\r,\gamma$ are small enough,
\begin{align*}
\Vert  x \Vert \big (1- \apa t \Vert   x \Vert^{N-1} + t \gamma^2 K
\Vert  x \Vert^{N-1} +  t^2 K \Vert  x \Vert^{2N-2}\big) &\leq
\frac{\Vert  x \Vert}{\big (1+  a(N-1)  t \Vert  x \Vert^{N-1}\big )^{\alpha}},\\
\Vert  x \Vert \big (1- \bpa t \Vert   x \Vert^{N-1} - t \gamma^2 K
\Vert  x \Vert^{N-1} - t^2 K \Vert  x \Vert^{2N-2}\big) &\geq
\frac{\Vert  x \Vert}{\big (1+  b(N-1)  t \Vert  x \Vert^{N-1}\big
)^{\alpha}}.
\end{align*}

Finally we are going to check that the results follow for any $ t
\geq 0$ and $ x\in \Omega(\r,\gamma)$. In fact we will check the
inequality involving $ a$, being the other one analogous. We have
already seen that if $ t\in [0,1]$ the inequalities are true so we
can proceed by induction assuming that the result is true for $ t
\in [0,l]$ with $l\in \NN$. We introduce the auxiliary differential
equation $\dot{\chi} = - a \chi^N$, $\chi \in \RR$, and its flow
$\chi( t,\xi)$, $\xi \in \RR$.  By induction hypothesis $\Vert \fpa(
t, x) \Vert \leq \chi( t, \Vert  x \Vert )$ if $ t \in [0,l]$.
Moreover, by Picard's theorem, if $\xi_1 < \xi_2$ then for all $
t\geq 0$, $\chi( t,\xi_1) < \chi( t,\xi_2)$. Consequently, by using
that $\Omega(\r,\gamma)$ is invariant by the flow $\fpa$, for any $s
\in [0,1]$ and $ t\in[0,l]$, we have that
\begin{equation*}
\Vert \fpa( t+s,  x) \Vert = \Vert \fpa( t,\fpa(s, x))\Vert \leq
\chi\big ( t , \Vert \fpa(s, x) \Vert\big ) \leq \chi \big ( t ,
\chi(s,\Vert  x \Vert)\big ) = \chi( t+s, \Vert  x \Vert)
\end{equation*}
and the induction is completed.

Let $ x\in \Omega(\gamma)$ and $\lambda>0 $ small enough such that
$\lambda  x \in \Omega(\r,\gamma)$. From \eqref{homogeneitatphi},
$$
\fpa( t, x)= \frac{1}{\lambda}\fpa\left (\frac{
t}{\lambda^{N-1}},\lambda  x\right )
$$
and from this expression, the bounds for $\Vert \fpa(s,\cdot) \Vert$
in $\Omega(\r,\gamma)$ extend to $\Omega(\gamma)$.

In the real case since $\gamma=0$, the result is valid for any $0< a < \apa$
and $ b > \bpa$ and we obtain the same bounds with $ a = \apa$ and $ b
= \bpa$.
\end{proof}
\begin{lemma}\label{cotaMQ} Assume that HP1 and HP2 are fulfilled for some $\r_0>0$.
Let $0< a\le \apa$, $ b \ge \bpa$, $A\geq \AQ$ and $B\leq \BQ$.
Then,
for all $ x \in V$ and $ t \geq 0$, we have the following bounds
\begin{align*}
 \big (1+c(N-1) t \Vert  x \Vert^{N-1}\big )^{\a\frac{B}{c}}\leq
\Vert M(t,x)\Vert \leq \big (1+\delta(N-1) t \Vert  x \Vert^{N-1}\big )^{\a\frac{A}{\delta}} \\
\big (1+\delta(N-1) t \Vert  x \Vert^{N-1}\big )^{-\a\frac{A}{\delta}}\leq
\Vert M^{-1}(t,x) \Vert \leq \big (1+c(N-1) t \Vert  x
\Vert^{N-1}\big )^{-\a\frac{B}{c}}
\end{align*}
with
\begin{equation}\label{defcd}
c=\begin{cases} \;\;a,   &\text{if}\;\; B\leq 0,\\
\;\;b,   &\text{otherwise}.
\end{cases}
\qquad
\delta= \begin{cases} \;\; a ,   &\text{if}\;\; A\geq 0,\\
\;\;b,   &\text{otherwise}.
\end{cases}
\end{equation}

If $\pa$ and $\Qa$ have an analytic extension  to $\Omega(\gamma_0)$
for some $\gamma_0\leq 1$, and $\Apa>\bpa$, then for any $0<a<\apa$,
$b>\bpa$, $ A > \AQ$ and $ B<\BQ$ there exists $\gamma \leq
\gamma_0$ such that, for $ t\geq 0$, $M( t, x)$ is analytic in
$\Omega(\gamma)$ and the previous bounds are also true for $ x \in
\Omega(\gamma)$.
\end{lemma}
\begin{proof}
By Lemma \ref{invflux}, the condition $\Apa>\bpa$ implies that there
exist $\r>0$ and $\gamma>0$ such that the set $\Omega(\r,\gamma)$ is
invariant by $\fpa$ if $\gamma$ is small enough provided that $\pa$
has an analytic extension to $\Omega(\r,\gamma_0)$. This will be the
only place where we use the condition $\Apa>\bpa$. For that reason
we will perform our computations in the analytic case, the real
case being just a direct consequence by taking $\gamma=0$.

Let $ x \in \Omega(\r,\gamma)$. First consider the auxiliary
differential equation
\begin{equation*}
\dot{\zeta} = \big (\Id + \Qa(\fpa( t, x))\big ) \zeta
\end{equation*}
and denote by $\chi( t, x)$ its fundamental matrix satisfying
$\chi(0, x)=\Id$. We notice that $\chi( t, x) = \ee^{ t} M( t, x)$.
Moreover,
\begin{equation*}
\chi( t, x) = \Id + \int_{0}^{ t} \big (\Id + \Qa(\fpa(s, x))\big )
\chi(s, x)\, ds.
\end{equation*}
Hence, by the definition of $\AQ$ and Lemma~\ref{invomega}, we have
that
\begin{align*}
\Vert \chi( t, x)\Vert &\leq 1+ \int_{0}^{ t} \Vert \Id + \Qa(\fpa(s, x))\Vert \Vert \chi(s, x)\Vert \, ds \\
&\leq   1+ \int_{0}^{ t} (1+ (\AQ +K\gamma)\Vert \fpa(s, x)
\Vert^{N-1}) \Vert \chi(s, x)\Vert \, ds.
\end{align*}
Writing $ A = \AQ + K \gamma$ and using Gronwall's Lemma,
\begin{align*}
\Vert \chi( t, x) \Vert \leq \text{exp}\left(\int_{0}^{ t}\big ( 1 +
 A\Vert \fpa(s, x) \Vert^{N-1}\big) \, ds\right )  =
\ee^{ t}\text{exp}\left( A \int_{0}^{ t} \Vert \fpa(s, x)
\Vert^{N-1} \, ds \right ).
\end{align*}
By using that $\chi( t, x) = \ee^{ t} M( t, x)$, we obtain that
\begin{equation}\label{prevboundedoM1}
\Vert M( t, x)\Vert \leq \text{exp}\left( A \int_{0}^{ t} \Vert
\fpa(u, x)\Vert^{N-1}\, du \right ).
\end{equation}
In the real case, i.e. when $ x \in \Vr=\Omega(\r,0)$, we can
take $ A = \AQ$.

Let us consider the differential equation
\begin{equation*}
\dot{\zeta} = \big (\Id - \Qa^{\top}(\fpa( t, x)) \big )\zeta .
\end{equation*}
We have that its fundamental matrix $\psi( t, x)$ such that $\psi(0,
x)=\Id$ is $\psi( t, x)= \ee^{ t} M^{-\top}( t, x)$, where here we
have written $M^{-\top} = [M^{-1}]^{\top}$. Indeed,
\begin{equation*}
\dot{\psi}( t, x) = \ee^{ t} M^{-\top}( t, x) + \ee^{ t}
\dot{M}^{-\top}( t, x) = \psi( t, x) -  \Qa^{\top}(\fpa( t, x))
\psi( t, x).
\end{equation*}
Now we have that
\begin{equation*}
\psi( t, x) = \Id + \int_{0}^{ t} \big (\Id -  \Qa^{\top}(\fpa(s,
x)) \big )\psi(s, x)\, ds.
\end{equation*}
We transpose the above equality and take norms to obtain
\begin{equation*}
\Vert \psi^{\top}( t, x) \Vert \leq 1 + \int_{0}^{ t} \Vert  \Id -
\Qa(\fpa(s, x)) \Vert \Vert \psi^{\top}(s, x)\Vert \, ds .
\end{equation*}
Finally using the definition of $\BQ$, Lemma~\ref{invomega} and
Gronwall's Lemma we conclude that
\begin{align*}
\Vert \psi^{\top}( t, x)\Vert &\leq \text{exp}\left(\int_{0}^{ t} 1 - (\BQ -K\gamma)\Vert \fpa(s, x) \Vert^{N-1} \, ds\right )\\
&= \ee^{ t}\text{exp}\left(-(\BQ -K\gamma)\int_{0}^{ t} \Vert
\fpa(s, x) \Vert^{N-1} \, ds \right )
\end{align*}
and, as a consequence, since $\psi^{\top}( t, x)= \ee^{ t} M^{-1}(
t, x)$ we have that
\begin{equation}\label{prevboundedoM2}
\Vert M^{-1}( t, x)\Vert \leq \text{exp}\left(- B \int_{0}^{ t}
\Vert \fpa(u, x)\Vert^{N-1}\, du \right ),
\end{equation}
where we have taken $ B = \BQ - K\gamma$. In order to bound
$\int_{0}^{ t} \Vert \fpa(u, x)\Vert^{N-1}\, du$ we use the bounds in
Lemma~\ref{cotafpa} obtaining
\begin{align*}
\int_{0}^{ t} \Vert \fpa(u, x)\Vert^{N-1}\, du &\leq \Vert  x \Vert^{N-1} \int_{0}^{ t} \frac{1}{1 +  a(N-1) u \Vert  x \Vert^{N-1}}\, du \\ &=\frac{1}{ a(N-1)}\log \big (1+ a(N-1) t \Vert  x \Vert^{N-1}\big ),\\
\int_{0}^{ t} \Vert \fpa(u, x)\Vert^{N-1}\, du & \geq \Vert  x
\Vert^{N-1} \int_{0}^{ t} \frac{1}{1 +  b(N-1) u \Vert  x
\Vert^{N-1}}\, du \\ &=\frac{1}{ b(N-1)}\log \big (1+ b(N-1) t \Vert
x \Vert^{N-1}\big ).
\end{align*}
By Lemma \ref{defKapwell}, $\BQ\leq \AQ$. To obtain the inequalities
in the statement from \eqref{prevboundedoM1} and
\eqref{prevboundedoM2} we distinguish three cases according to the
signs of $\AQ, \BQ$. The first case is $\BQ> 0$. Let $0< B < \BQ$
and $ A
>\AQ$. We take  $0<\gamma_1\leq \gamma_0$ such that $0<  B \le
\BQ-K\gamma_1$ and $ A \ge \AQ + K\gamma_1$. Then, if $0\leq
\gamma\leq \gamma_1$,
\begin{align*}
\Vert M( t, x) \Vert &\leq \big (1+ a(N-1) t \Vert  x \Vert^{N-1}\big )^{\frac{ A}{ a(N-1)}}, \\
\Vert M^{-1}( t, x) \Vert & \leq \big (1+ b(N-1) t \Vert  x
\Vert^{N-1}\big )^{\frac{- B}{ b(N-1)}}.
\end{align*}
The remaining inequalities follow from $\Vert M^{-1}( t, x) \Vert
\geq \Vert M( t, x) \Vert^{-1}$. The other two cases, $\AQ<0$ and
$\BQ \leq 0 \le \AQ$, follow analogously.

Using the identity~\eqref{homogeneitatphi} $ M( t, x) = M\left(
\lambda^{-N+1}  t, \lambda x \right) $, the inequalities extend
to $\Omega(\gamma)$. Note that in the real case we can take $
A=\AQ$, $ B= \BQ$, $ a = \apa$ and $ b = \bpa$.
\end{proof}
\subsection{Proof of Theorem~\ref{solutionlinearequation}} \label{subsectionproofTheoremauxiliaryequation}
We begin by checking that if $  h:V\to \RR^k$ is a differentiable
solution of~\eqref{modellinearequation} in $\Hog{\m+1}$, it has to
be given by formula~\eqref{defh0} given in
Theorem~\ref{solutionlinearequation}, i.e.
$$
h( x) = \int_{\infty}^0 M^{-1}( t, x) \T{}(\fpa( t, x))\, d t.
$$
Indeed, let $  h\in \mathcal{H}^{\m+1}$ be such that
$$
Dh(x) \pa(x) - \Qa(x) h(x)=\T(x).
$$
We define $\mu(t,x)=  h({\fpa}(t,x))$ and we have that
$$
\dot{\mu}( t, x) = D  h(\fpa( t, x)) \pa(\fpa( t, x)) = \Qa(\fpa( t,
x))\mu( t, x) + \T(\fpa( t, x))
$$
and then, since $\mu(0,x)=h(x)$,
$$
\mu( t, x) = {M}( t, x) \left (   h(x) + \int_{0}^{ t} {M^{-1}}(s,x)
\textbf{w}(\fpa(s,x))\, ds\right ).
$$
Note that, with $\r$ given by Lemma~\ref{invflux}, if $x\in \Vr$,
$\fpa(s,x)\in \Vr$ for all $s\ge 0$. The hypothesis~\eqref{hipgraus},
Lemmas~\ref{cotafpa} and~\ref{cotaMQ} and the fact that $\Vert
  h( x) \Vert \leq  K \Vert  x \Vert^{\m+1}$, imply that
${M^{-1}}( t,x) \mu(t,x) = {M^{-1}}(t,x)  h(\fpa( t,x))\to 0$ as
$t\to \infty$. Then we obtain the desired expression for $  h$.

This provides the uniqueness statement in $\Vr$. The fact that $  h$ belongs
to $\Hog{\m+1}$ will be proven in the next lemma in a slightly more
general setting. The homogeneity of $h$ determines uniquely the extension of $h$ to $V$ which satisfies~\eqref{modellinearequation} in $V$.
Then it remains to prove that actually $  h$ is
well defined, it is a solution and its regularity. Our strategy to
prove the regularity stated in Theorem~\ref{solutionlinearequation}
follows three steps. The first one deals with the continuity
(resp. analyticity) of functions defined by integrals of the form
\begin{equation}\label{defgind}
 g( x):=\int_{\infty}^0 \MP^{-1}( t, x) \Ti{}(\fpa( t, x))\, d t
\end{equation}
with $\M{}$ and $\Ti{}$ satisfying appropriate conditions. Note that
definition \eqref{defh0} of $  h$ fits in this setting. This is done
in Lemma \ref{lemma:hanalytic} below.

Secondly, we deal with the $\CC^1$ regularity, proving both: i) that
$ g\in \CC^1$ and ii) that $\Di{} g$ can be expressed as
$$
\int_{\infty}^0 (\M{1})^{-1}( t, x) \Ti{1}(\fpa( t, x))\, d t
$$
with $\M{1}$ and $\Ti{1}$ having the conditions required in the
previous step for $ g$ to be a continuous function. This is proven
in Lemma \ref{lemma:hdif}.

Finally, the third step consists of an inductive procedure with
respect to the degree of differentiability.

In what follows we will use the constants introduced at the
beginning of~Section~\ref{subsectionstep2} depending on the
homogeneous functions indicated in their subscripts without further notice.

\begin{lemma}\label{lemma:hanalytic}
Let $\pa \in \Hog{N}$ be defined on $V$ and satisfying hypotheses
HP1 and HP2 for $\r_0$, $\Qi{} \in \Hog{N-1}$ and $\Ti{}\in
\Hog{\nu+N}$ on $V$, with $\nu \ge 1$. We denote by $\M{}$ the
fundamental matrix of
$$
\frac{d}{d t} \psi( t, x) = \Qi{}(\fpa( t, x)) \psi( t, x),\qquad
\text{such that }\quad  \M{}(0, x)=\Id.
$$
If $\nu+1+\frac{\BQi}{\cpa} > 0$, with $\cpa$ defined in~\eqref{defconstantspa} taking $\Qa=\Qi{}$,
then the function $ g:V \to \RR^k$
defined by \eqref{defgind} belongs to $\Hog{\nu+1}$ being, in
particular, a $\CC^0$ function on $V$.

Moreover, if we also have $\Apa>\bpa$, then, there exists $\gamma>0$
small enough such that the function $ g$ is analytic in
$\Omega(\gamma)$ provided $\pa$, $\Qi$ and $\Ti{}$ have analytic
extensions to $\Omega(\gamma_0)$ for some $\gamma_0>\gamma$.
\end{lemma}
\begin{proof}
If $\pa, \Qi$ and $\Ti{}$ have analytic extensions to
$\Omega(\gamma_0)$, let $0<a<\apa$, $b>\bpa$ and $B<\BQi$ be
such that $\nu+1+\frac{B}{c} >0$ where $c$ is defined in~\eqref{defcd}.
We fix $\r$ and $\gamma$ satisfying the conditions of Lemmas~\ref{invflux},
\ref{cotafpa} and \ref{cotaMQ}. In this case we have that
$\Omega(\r,\gamma)$ is invariant by $\fpa$ provided $\Apa>\bpa$. Since
$\Vr=\Omega(\r,0)$, we make the convention that in the real case, we take
$\gamma=0$. This allows us to deal with both cases (real and
complex) at the same time. If $\Ti{}$ is a $\CC^0$ function on $V$
we take $U=V$ and if $\Ti{}$ has an analytic extension to
$\Omega(\r,\gamma)$ for some $\gamma>0$, we take $U=\Omega(\r,\gamma)$.
With this convention, we define
\begin{equation*}
\Vert \Ti{} \Vert = \sup_{ x \in U}\frac{\Vert \Ti{}( x)\Vert
}{\Vert  x \Vert^{\nu+N}}.
\end{equation*}

We begin by proving that the function $ g$ is well defined and $\CC^{0}$ in $\Omega(\r,\gamma)$.
Indeed, we only need to
check that the integral is convergent. For that we use
Lemmas~\ref{cotafpa} and~\ref{cotaMQ} applied to $\Qi$. Let $ x \in
\Omega(\r,\gamma)$
\begin{align*}
\Vert \MP^{-1}( t, x) \Ti{}(\fpa( t, x) )\Vert & \leq \Vert \Ti{} \Vert \Vert \fpa( t, x) \Vert^{\nu+N} \Vert \MP^{-1}( t, x)\Vert \\
&\leq \Vert \Ti{} \Vert \frac{\Vert  x \Vert^{\nu+N}}{\big (1 +
a(N-1)  t \Vert  x \Vert^{N-1} \big )^ {\a\big(\nu+N
+\frac{B}{c}\big )}}
\end{align*}
because $c \geq a$ and $\kappa:=\a\big(\nu+N +\frac{B}{c}\big ) = \a(N-1) + \a\big(\nu+1 +\frac{B}{c}\big )>1$ by hypothesis.
Therefore,
\begin{equation}\label{hanalytic1}
\Vert \MP^{-1}( t, x) \Ti{}(\fpa( t, x))\Vert \leq  \Vert \Ti{}
\Vert \Vert  x \Vert^{\nu+N} \big (1 + a(N-1)  t \Vert  x
\Vert^{N-1} \big )^{-\kappa}
\end{equation}
which implies that
\begin{equation*}
\Vert  g( x) \Vert \leq \Vert \Ti{} \Vert \Vert  x \Vert^{\nu+N}
\int_{0}^{\infty} \frac{d  t}{\big (1 + a(N-1)  t \Vert  x
\Vert^{N-1} \big )^ {\kappa}} \leq K \Vert \Ti{} \Vert \Vert  x
\Vert^{\nu+1}.
\end{equation*}
Now we prove that $ g$ belongs to $\Hog{\nu+1}$. As we mentioned in
\eqref{homogeneitatphi}, for any $\lambda > 0$, one has that
$\fpa(t,\lambda  x) = \lambda \fpa(\lambda^{N-1} t,  x)$ and
$\MP^{-1}(t,\lambda  x) = \MP^{-1}(\lambda^{N-1} t , x)$. Then,
\begin{align*}
 g({\lambda} x) &= \int_{\infty}^0 {\MP^{-1}} ( t,{\lambda}  x)
\Ti{}(\fpa( t,{\lambda} x)) \, d t =\int_{\infty}^0 {\MP^{-1}}
({\lambda^{N-1}} t,x)  \Ti{}({\lambda} \fpa({\lambda^{N-1}}  t, x))
\, d t
\\
&={\lambda^{1-N}} \int_{\infty}^0  {\MP^{-1}} (t,x)
\Ti{}({\lambda}\fpa(  t, x))\,d t = {\lambda^{1-N}}
{\lambda^{\nu+N}}  \int_{\infty}^0  {\MP^{-1}} (t,x) \Ti{}(\fpa(
 t, x))\,d t
\\ &={\lambda^{\nu+1}}  g(x).
\end{align*}

Finally we check the regularity. We first check that $ g$ is
analytic if $\Ti{}$, $\Qi$ and $\pa$ have analytic extensions to
$\Omega(\r,\gamma)$. Let $ x_0 \in \Omega(\r,\gamma)$ be a given point.
Since $\Omega(\r,\gamma)$ is an open set, there exists $0<r<\Vert
x_0\Vert$ such that the open ball $B_{r}( x_0)$ is contained in
$\Omega(\r,\gamma)$. Then, if $ x \in B_r( x_0)$, $\Vert  x \Vert \geq
\Vert  x_0 \Vert - r$ and consequently, using \eqref{hanalytic1},
\begin{align*}
\Vert \MP^{-1}( t, x) \Ti{}(\fpa( t, x)) \Vert &\leq \Vert \Ti{}
\Vert \frac{\Vert  x \Vert^{\nu+N}}{\big (1 + a(N-1)  t \Vert  x
\Vert^{N-1} \big )^{\kappa}}
\\ &\leq \Vert \Ti{} \Vert \frac{\big (\Vert  x_0 \Vert+ r_0\big )^{\nu+N}}{\big (1 + a(N-1)  t (\Vert  x_0 \Vert -r)^{N-1} \big )^{\kappa}}
\end{align*}
and the analyticity follows from the dominated convergence theorem
because the right hand side of the above bound does not depend on $
x$ and it is integrable.

Since $g$ is homogeneous we can extend it uniquely to an analytic homogeneous function in $\Omega(\gamma)$. Considering $\Ti{}$ extended by homogeneity as indicated in
Remark~\ref{rem:exthom} the extension of $g$ has expression~\eqref{defgind}.

In the real case, when $\pa$ is $\CC^1$ and $\Ti{}$, $\Qi$ are continuous homogeneous
functions, the same argument as the one given in the analytic case,
leads to the proof that $ g$ is a continuous function.
\end{proof}

Now we are going to deal with the differentiable case. If $ g$
is a solution of
\begin{equation}
\label{eq:homhi} D g(x) \pa(x) - \Qi(x) g(x)=\Ti{} (x),
\end{equation}
then $D g$, if it is $\CC^1$, should satisfy
\[
D^2 g(x) \pa(x) - [\Qi(x)D g(x) - D g(x) D\pa(x) ]=D\Ti{} (x) +
D\Qi(x)  g(x)
\]
which is an equation for~$D g$ analogous to~\eqref{eq:homhi} except
that the second term, due to the lack of commutativity is more
involved. Continuing in this way would imply to consider linear
equations of the form
\[
\dot \chi = \Qi(\fpa( t,{\lambda} x)) \chi - \chi D\pa(\fpa(
t,{\lambda} x)).
\]
However we have chosen to consider the equivalent equation for a
vector which contains all elements $D_{ij}  g$ ordered one column
after the other. This forces the introduction of the following
notation.

We denote by $\Di{j}$ the derivative with respect to the variable
$x_j$.

We define the linear operator $\opC:\L{}(\RR^{n},\RR^{k}) \to
\RR^{n\cdot k}$:
\begin{equation}\label{defopC}
\opC(A)=\big ( (A e_1)^\top, \cdots, (A e_n)^\top\big )^\top, \qquad
\text{being   }\{e_1 ,\cdots, e_n\} \; \text{the canonical basis},
\end{equation}
and the functions $\BMa{\Qi}, \IMa{D\pa}{k}, \DQi: \Vr \to
\L{}(\RR^{n\cdot k},\RR^{n\cdot k})$:
\begin{align}
\BMa{\Qi}(x) &=\diag \big(\Qi(x),\cdots, \Qi(x)\big ),\label{defBlockQ}\\
\IMa{D\pa}{k}(x)&=\left ( \begin{array}{ccc} \Di{1}\pa_1(x) \Id_{k} &\cdots &\Di{1}\pa_n(x) \Id_{k} \label{defDpext}\\
\vdots & \vdots & \vdots \\
\Di{n}\pa_1(x) \Id_{k} & \cdots & \Di{n} \pa_n(x) \Id_{k}
\end{array}\right ), \\
\DQi(x)&=\BMa{\Qi}(x)-\IMa{D\pa}{k}(x). \notag
\end{align}
with $\pa=(\pa_1,\cdots,\pa_n)^{\top}$ and $\Id_{k}$ the identity in
$\L{}(\RR^k,\RR^k)$.

For any $w \in \RR^{n\cdot k}$, we also write
$$
w=(w_1,\cdots, w_n),\quad \text{with} \quad w_i \in \RR^k.
$$
Finally we define the norm in $\RR^{n\cdot k}$
$$
\Vert w \Vert = \sup_{u\in \RR^n\backslash \{0\}} \frac{\Vert u_1
w_1+ \cdots + u_n w_n\Vert}{\Vert u \Vert} =\sup_{\Vert u \Vert =1}
\Vert u_1 w_1+ \cdots + u_n w_n\Vert,
$$
where the norms in $\RR^n$ and $\RR^k$ are such that HP1 and HP2
hold.

Let $\M{1}( t, x)$ be the fundamental solution of
\begin{equation}\label{edoQdif}
\frac{d \fqa}{d  t}( t, x)  = \DQi(\fpa( t, x)) \fqa( t, x) \qquad
\text{such that} \quad \M{1}(0, x)=\Id.
\end{equation}

\begin{lemma}\label{boundQ1}
Let $0<\r \leq \r_0$. Then
\begin{enumerate}
\item[(1)] we have that
\begin{equation}\label{boundBQdif}
\BDQi := -\sup_{ x \in \Vr} \frac{\Vert \Id- \DQi( x)\Vert-1}{\Vert
x \Vert^{N-1}} \geq \BQi+\Apa.
\end{equation}
\item[(2)] The fundamental matrix $\M{1}$ of \eqref{edoQdif} satisfies
$$(\M{1})^{-1}( t, x) = \IMa{D\fpa}{k}( t, x) \cdot \BMa{\MP^{-1}}( t, x)$$
with
\begin{align*}
\BMa{\MP^{-1}}( t, x) &= \diag \big(\MP^{-1}( t, x),\cdots, \MP^{-1}( t, x)\big ),\\
\IMa{D\fpa}{k}( t, x)&=\left ( \begin{array}{ccc} \Di{1}\fpa_1( t, x) \Id_{k} &\cdots &\Di{1}\fpa_n( t, x) \Id_{k} \\
\vdots & \vdots & \vdots \\
\Di{n}\fpa_1( t, x) \Id_{k} & \cdots & \Di{n} \fpa_n( t, x) \Id_{k}
\end{array}\right ).
\end{align*}
\end{enumerate}
\end{lemma}
\begin{proof}
Let $w \in \RR^{n\cdot k}$, $w=(w_1,\cdots, w_n)$ with $\Vert w
\Vert =1$. We have that
\begin{align}\label{boundKDQ1}
\left \Vert \left (\frac{1}{2}\Id - \BMa{\Qi} ( x)\right )w\right
\Vert &=\sup_{\Vert u \Vert=1} \left \Vert
\left (\frac{1}{2}\Id -\Qi( x)\right )\big ( w_1 u_1+ \cdots + w_n u_n\big )\right \Vert \notag\\
&\leq\left \Vert \frac{1}{2}\Id -\Qi( x) \right \Vert  \sup_{\Vert u
\Vert=1} \Vert w_1  u_1 + \cdots + w_nu_n\Vert \notag\\ &=\left
\Vert \frac{1}{2}\Id -\Qi( x) \right \Vert \le \frac{1}{2} -
\BQi\Vert x\Vert^{N-1},
\end{align}
where we have used that
\[
\left \Vert \frac{1}{2} \Id - \Qi ( x)\right \Vert = \left \Vert
\frac{1}{2}\left(\Id - \Qi(2^{1/(N-1)} x ) \right) \right \Vert \le
\frac{1}{2}\left(1 - \BQi \big (2^{1/(N-1)} \Vert x \Vert\big)^{N-1}
\right) .
\]

In addition, we can decompose $\big (\frac{1}{2} \text{Id} +
\IMa{D\pa}{k}( x)\big ) w = (\bar{w}_1,\cdots, \bar{w}_n)^{\top}$,
with $\bar{w}_i\in \RR^k$ and
\begin{align*}
\bar{w}_i -\frac{1}{2}w_i=\Di{i}\pa_1( x) w_1 + \cdots
+\Di{i}\pa_n( x) w_n.
\end{align*}
Given $u = (u_1,\dots, u_n) \in \RR^n$, letting $\bar{u}= \big
(\frac{1}{2} \Id + D\pa( x) \big )u$ we have
$$
u_1 \bar{w}_1 + \cdots + u_n \bar{w}_n = \bar{u}_1 w_1 + \cdots
+\bar{u}_n w_n.
$$
As a consequence,
\begin{align*}
\sup_{u\in \RR^n \backslash\{0\}} \frac{\Vert u_1 \bar{w}_1 + \cdots
+ u_n \bar{w}_n\Vert}{ \Vert u\Vert}  & \leq \left \Vert \left
(\frac{1}{2} \Id + D\pa( x) \right ) \right \Vert \sup_{\bar{u}\in
\RR^n \backslash\{0\}} \frac{\Vert \bar{u}_1 w_1 + \cdots +
\bar{u}_n w_n\Vert} {\Vert \bar{u} \Vert}
\\ & =\left \Vert \left (\frac{1}{2} \Id + D\pa( x) \right ) \right \Vert
\leq\frac{1}{2} - \Apa \Vert  x \Vert^{N-1}.
\end{align*}
The above bound jointly with \eqref{boundKDQ1} gives that
$$
\Vert \Id - \DQi( x) \Vert \leq \left \Vert \frac{1}{2} \Id -
\BMa{\Qi}( x) \right \Vert + \left \Vert \frac{1}{2} \Id +
\IMa{D\pa}{k}( x) \right\Vert \leq 1 - (\BQi+\Apa)\Vert  x
\Vert^{N-1}
$$
and \eqref{boundBQdif} is proven.

To obtain the expression for $(\M{1})^{-1}( t, x)$ is a straightforward
computation.
\end{proof}
\begin{lemma}\label{lemma:hdif}
Assume that $\pa, \Qi$ and $\Ti{}$ are $\CC^1$ functions on
$V$. Let $\M{}$ be the fundamental matrix
of $\frac{d}{dt} \psi( t, x) = \Qi(\fpa( t, x))\psi( t, x)$
satisfying $\M{}(0, x)=\Id$.

If hypotheses HP1 and HP2 are satisfied for $\r_0$ and
\begin{equation}\label{hypnuhdif}
\nu+1+\frac{\BQi}{\cpa} > \max\left \{1
-\frac{\Apa}{\dpa},0\right\},
\end{equation}
with~$\cpa,\dpa$ defined in~\eqref{defconstantspa} taking $\Qa=\Qi{}$, then the
function $ g:V \to \RR^k$ defined in \eqref{defgind} belongs to
$\Hog{\nu+1}$ and is a $\CC^1$ function on $V$.

Moreover
\begin{equation}\label{lemma:expDh}
\opC (D g( x)) = \int_{\infty}^0 (\M{1})^{-1}( t, x) \Ti{1}(\fpa( t,
x))\, d t,
\end{equation}
where $\M{1}$ is the fundamental matrix of \eqref{edoQdif} such that
$\M{1}(0, x)=\Id$ and
\begin{equation}\label{defDifT}
\Ti{1}( x)=\opC(D \Ti{}( x))+ \left (( \Di{1}\Qi( x)  g( x)
)^{\top}, \cdots, (\Di{n}\Qi( x)  g( x))^{\top}\right )^{\top}.
\end{equation}
\end{lemma}
\begin{proof} Let $\r>0$ satisfying Lemma~\ref{invflux}.
We claim that for any $\tau\geq 0$ and $x\in \Vr$,
\begin{align}
\int_{\tau}^0 \Di{j} \big [\MP^{-1}( t, x) \Ti{}(\fpa( t, x))\big]\, d t =& -D_j \MP^{-1}(\tau, x)  g(\fpa(\tau, x)) \notag\\
&+\int_{\tau}^0  \big [(\M{1})^{-1} ( t, x) \Ti{1}(\fpa( t, x))\big
]_j \, d  t .\label{difhHtau}
\end{align}
We recall here that the subscript in a vector in $\RR^{n\cdot k}$
identifies a vector in $\RR^k$.

We will use the following properties related to $\M{}$:
\begin{align}
&\frac{d}{d t} \big (\MP^{-1}( t, x) \Di{j} \M{}( t, x)\big ) =
\MP^{-1}( t, x) \Di{j}\big (\Qi(\fpa( t, x))\big )\M{}( t, x),\label{firstpropertychi} \\
&\M{}(u+v, x) = \M{}(u,\fpa(v, x)) \M{}(v, x) \label{secondpropertychi}, \\
&\MP^{-1}(t, x) \Di{j}\M{}(t, x)= -\Di{j}\MP^{-1}(t, x) \M{}(t, x).
\label{thirdpropertychi}
\end{align}
Expression~\eqref{firstpropertychi} follows by using the variational
equation for $\M{}$. The second one follows from the uniqueness of
solutions of $\dot{\psi}( t, x) = \Qi(\fpa( t, x)) \psi( t, x)$ and
the last one taking derivatives in $\MP^{-1}( t, x) \M{}( t,
x)=\Id$.

From Lemma~\ref{boundQ1} and definition~\eqref{defDifT} of $\Ti{1}$
we obtain that
\begin{equation}\label{fourpropertychi}
\begin{aligned}
\big [(\M{1})^{-1} ( t, x)& \Ti{1}(\fpa( t, x))\big ]_j \\ &= \MP^{-1} (
t, x) \big [\Di{j}\big (\Ti{} (\fpa( t, x))\big )+
 \Di{j} \big (\Qi(\fpa( t, x))\big )  g(\fpa( t, x))\big ].
\end{aligned}
\end{equation}
Using properties~\eqref{secondpropertychi} in the definition of $
g$, we obtain that
\begin{align}
 g(\fpa( t, x))&= \int_{\infty}^{0} \MP^{-1}(s,\fpa( t, x))
\Ti{}(\fpa(s,\fpa( t, x))) \, ds \notag
\\&= \M{}( t, x) \int_{\infty}^{ t} \MP^{-1} (s, x) \Ti{}(\fpa(s, x))\, ds,\label{exphfpa}
\end{align}
and by \eqref{fourpropertychi}, \eqref{exphfpa} and
\eqref{firstpropertychi} we get
\begin{align*}
\big [(\M{1})^{-1} ( t, x) &\Ti{1}(\fpa( t, x))\big ]_j = \MP^{-1} ( t, x) \Di{j}\big (\Ti{} (\fpa( t, x))\big ) \\
&+ \frac{d}{d t} \big (\MP^{-1}( t, x) \Di{j} \M{}( t, x)\big )
\int_{\infty}^{ t} \MP^{-1} (s, x) \Ti{}(\fpa(s, x))\, ds .
\end{align*}
Integrating by parts and using $\Di{j}\MP^{-1}(0, x)=0$:
\begin{align*}
\int_{\tau}^0  \big [(\M{1})^{-1} ( t, x)& \Ti{1}(\fpa( t, x))\big
]_j \, d  t  = \int_{\tau}^0 \MP^{-1} ( t, x) \Di{j}\big (\Ti{}
(\fpa( t, x))\big ) \, d t \\ &-
\MP^{-1}(\tau, x) \Di{j} \M{}(\tau, x)\int_{\infty}^{\tau} \MP^{-1} ( t, x) \Ti{}(\fpa( t, x))\,d t \\
&-\int_{\tau}^0 \MP^{-1}( t, x) \Di{j} \M{}( t, x)\MP^{-1} ( t, x)
\Ti{}(\fpa( t, x))\,d t .
\end{align*}
Finally, using \eqref{thirdpropertychi} and expression
\eqref{exphfpa}, we obtain
\begin{align*}
\int_{\tau}^0  \big [(\M{1})^{-1} &( t, x) \Ti{1}(\fpa( t, x))\big
]_j \, d  t  =
\Di{j}\MP^{-1}(\tau, x)  g(\fpa(\tau, x)) \\
&+\int_{\tau}^0 \big[ \MP^{-1} ( t, x) \Di{j}\big (\Ti{} (\fpa( t,
x))\big ) +\Di{j}\MP^{-1}( t, x) \Ti{}(\fpa( t, x))\big ]\, d t
\end{align*}
from which \eqref{difhHtau} follows immediately.

We notice that, from \eqref{exphfpa} we get that $ g(\fpa(\tau, x))$
is differentiable with respect to $\tau$ even if $ g$ is not.
Moreover, let $\tilde{G}$ be the first term in the right hand side
of~\eqref{difhHtau}. Then
\begin{align}
\tilde{ g}(\tau, x):= & -\frac{d}{d\tau} \tilde{G}(\tau, x)=\frac{d}{d\tau} \big [\Di{j}\MP^{-1}(\tau, x)  g(\fpa(\tau, x))\big] \notag \\
=& -\MP^{-1}(\tau, x) \Di{j}\big (\Qi(\fpa(\tau, x))\big )
 g(\fpa(\tau, x)) + \Di{j}\MP^{-1}(\tau, x) \Ti{}(\fpa(\tau, x)).
\label{DjMh}
\end{align}
Therefore differentiating with respect to $\tau$ both sides of
\eqref{difhHtau}:
\begin{equation}
\label{expdifhbona}
\Di{j}\big [\MP^{-1}(\tau, t) \Ti{}(\fpa(\tau, x))\big ] =\tilde{
g}(\tau, x) +(\M{1})^{-1} (\tau, x) \Ti{1}(\fpa(\tau, x)).
\end{equation}
To prove the differentiability of $g$ we need to check that $\Di{j}\big [\MP^{-1}(\tau, t) \Ti{}(\fpa(\tau, x))\big ]$
is locally uniformly integrable with respect to $x$.
In order to prove this fact and expression \eqref{lemma:expDh} for $\opC({D g(
x)})$ in Lemma~\ref{lemma:hdif}, we prove the locally uniformly boundedness (with respect to $ x$) by an
integrable function of the right hand side of~\eqref{expdifhbona}.
Indeed, we have that $\Ti{1}\in \Hog{\nu-1+N}$
and that by Lemma~\ref{boundQ1}, $\BDQi \geq \BQi+\Apa$. We apply
Lemma~\ref{lemma:hanalytic} with $\nu-1$, $\M{1}$ and $\Ti{1}$
instead of $\nu$, $\M{}$ and $\Ti{}$ respectively and we obtain that
the function
$$
G^{1}( x) := \int_{\infty}^0 (\M{1})^{-1}( t, x)\Ti{1}(\fpa( t,
x))\, d t
$$
belongs to $\Hog{\nu}$ provided $\nu+\frac{\BQi}{\cpa} + \frac{\Apa}{\dpa}>0$.
In fact, in the proof of Lemma~\ref{lemma:hanalytic} we checked that
$(\M{1})^{-1}( t, x) \Ti{1}(\fpa( t, x))$ is locally uniformly bounded with
respect to $ x$ by an integrable function.

Now we deal with $\tilde{ g}$. We first bound the first term in
\eqref{DjMh}. Since $\Qi \in \Hog{N-1}$, there exists a constant
$K>0$ such that
\begin{equation}\label{boundDjQ}
\Vert D_j\big (\Qi(\fpa(s,x))\big) \Vert \leq K \Vert \fpa(s, x)
\Vert^{N-2} \Vert D_j \fpa(s, x)\Vert.
\end{equation}
We recall that $D \fpa(\tau, x)$ is the fundamental solution of the
linear system $\dot{\psi} = D\pa(\fpa(\tau, x)) \psi$ such that
$D\fpa(0, x)=\Id$. Hence we apply Lemma~\ref{cotaMQ} to $D\fpa$ to
obtain:
\begin{equation}\label{Variationalbound}
\Vert D \fpa (\tau, x) \Vert \leq \frac{1}{\big ( 1+ \dpa (N-1) \tau
\Vert  x \Vert^{N-1}\big )^{\a \frac{\Apa}{\dpa}}}
\end{equation}
(compare definition of $\Apa$ and definition of $\AQ$
in~\eqref{defconstantspa}). Using \eqref{Variationalbound} and the
bound of $\Vert \fpa( t, x)\Vert$ given by Lemma~\ref{cotafpa}
in~\eqref{boundDjQ}, we get
\begin{equation}\label{DMbound}
\Vert D_j\big (\Qi(\fpa(s,x))\big) \Vert \leq K \frac{\Vert  x
\Vert^{N-2}}{\big(1+ \apa(N-1)s \Vert  x \Vert^{N-1}\big )^
{\a\big((N-2) + \frac{\Apa}{\dpa}\big )}}.
\end{equation}

By Lemma \ref{lemma:hanalytic}, $\Vert  g( x) \Vert \leq K\Vert  x
\Vert^{\nu+1}$ for some constant $K>0$. Using the bounds of $\Vert
\MP^{-1}( t, x)\Vert$ and $\Vert \fpa( t, x)\Vert$ given by
Lemmas~\ref{cotaMQ} and~\ref{cotafpa} respectively, we obtain:
\begin{equation}\label{firstboundhtilde}
\Vert \MP^{-1}(\tau, x) \Di{j}\big (\Qi(\fpa(\tau, x))\big )
 g(\fpa(\tau, x))\Vert \leq K\frac{\Vert  x \Vert^{\nu+N-1}} {\big
(1+ \apa(N-1)\tau \Vert  x \Vert^{N-1}\big )^{\kappa_0}}
\end{equation}
with $\kappa_0=\a \left ( \nu+1 + \frac{\BQi}{\cpa} +N-2 +
\frac{\Apa}{\dpa}\right )$ and $\kappa_0 >1$ by hypothesis.

We deal with $\Vert \Di{j}\MP^{-1}(\tau, x)\Vert $ for $\tau\geq 0$
and $ x \in \Vro$. $\Di{j}\MP^{-1}(\tau, x)$ is the solution of
\begin{equation*}
\frac{d}{d\tau} D_j \MP^{-1}(\tau, x) = - D_j \MP^{-1}(\tau, x)
\Qi(\fpa(\tau, x)) - \MP^{-1}(\tau, x) D_j \big
(\Qi(\fpa(\tau,x))\big )
\end{equation*}
satisfying the initial condition $\Di{j}\MP^{-1}(0, x)=0$. We have
then
\begin{align}\label{DMdef}
D_j \MP^{-1}(\tau, x) &= -\left (\int_{0}^{\tau} \MP^{-1}(s, x) D_j\big (\Qi(\fpa(s,x))\big) \MP(s, x) \,ds\right ) \MP^{-1}(\tau, x) \notag\\
&=-\int_{0}^{\tau} \MP^{-1}(s, x) D_j\big (\Qi(\fpa(s,x))\big)
\MP^{-1}(\tau-s,\fpa(s, x))\,ds,
\end{align}
where we have used \eqref{secondpropertychi} again.

For $\tau>s$, by Lemmas~\ref{cotafpa} and~\ref{cotaMQ}, a
calculation (distinguishing the cases $\BQi\geq 0$ and $\BQi<0$)
gives
\begin{align}\label{boundXtXs}
\Vert  \MP^{-1}(\tau-s,\fpa(s, x))\Vert \Vert \MP^{-1} (s, x) \Vert
&\leq \frac{\Vert \MP^{-1} (s, x) \Vert }{\big (1+ \cpa(N-1)(\tau-s) \Vert \fpa(s, x) \Vert^{N-1}\big )^{\a \frac{\BQi}{\cpa}}}\notag \\
&\leq \frac{1}{ \big (1+ \cpa(N-1)\tau \Vert  x \Vert^{N-1}\big
)^{\a \frac{\BQi}{\cpa}}}.
\end{align}
Note that the bound is independent of~$s$. If $\dpa \neq \Apa$,
using bound \eqref{DMbound} for $\Vert \Di{j}\big (\Qi(\fpa(s,
x))\big ) \Vert$:
\begin{equation*}
\int_{0}^{\tau} \Di{j}\big (\Qi(\fpa(s, x))\big )\,ds  \leq K \Vert
 x \Vert ^{-1}\big(1+ \apa(N-1) \tau \Vert  x \Vert^{N-1}\big )^{\a
\max\left \{0, 1-\frac{\Apa}{\dpa}\right \}}.
\end{equation*}
Using previous computations for bounding the terms in formula
\eqref{DMdef}, we obtain that
\begin{equation*}
\Vert \Di{j} \MP^{-1}(\tau, x) \Vert \leq \frac{K \Vert  x
\Vert^{-1}}{ \big (1+ \apa(N-1)\tau \Vert  x \Vert^{N-1}\big )^{\a
\left (\frac{\BQi}{\cpa} - \max\left \{0,1-\frac{\Apa}{\dpa} \right
\} \right )}}.
\end{equation*}
In addition, using that $\Ti{} \in \Hog{\nu+N}$ and the bound for
$\Vert \fpa( t, x)\Vert$ in Lemma~\ref{cotafpa}:
\begin{equation}\label{secondboundhtilde}
\Vert \Di{j}\MP^{-1}(\tau, x) \Ti{} (\fpa(\tau, x))\Vert \leq K
\Vert  x \Vert^{\nu+ N -1}\big (1 + \apa(N-1) \tau \Vert  x
\Vert^{N-1} \big ) ^{-\kappa}
\end{equation}
with $\kappa= \a\left( \nu +N + \frac{\BQi}{\cpa} - \max\left
\{0,1-\frac{\Apa}{\dpa} \right \}\right)$. By hypothesis $\kappa
>1$. Also, $\kappa_0\ge \kappa$.

Now, to bound $\tilde{ g}$ defined in \eqref{DjMh}, we use
\eqref{firstboundhtilde} and \eqref{secondboundhtilde} and we get:
$$
\Vert \tilde{ g}(\tau, x)\Vert \leq K \Vert  x \Vert^{\nu+ N -1}\big
(1 + \apa(N-1) \tau \Vert  x \Vert^{N-1} \big ) ^{-\kappa}
$$
which can be locally uniformly bounded with respect to $ x$ by an absolutely
integrable function.

If $\dpa = \Apa$, an analogous argument leads to
\begin{equation*}
\Vert \tilde{ g}(\tau, x)\Vert \leq K \Vert  x \Vert^{\nu+ N -1}\big
(1 + \apa(N-1) \tau \Vert  x \Vert^{N-1} \big ) ^{-\kappa} \log \big
(1 + \apa(N-1) \tau \Vert  x \Vert^{N-1} \big ).
\end{equation*}

Then $g$ is differentiable and
$$
\Di{j} g(x) = \int_{\infty}^0 \Di{j} \big (\MP^{-1}(t,x) \Ti{}(\fpa(t,x))\big )\, dt.
$$
Using~\eqref{difhHtau}, and the fact that $\lim_{\tau\to \infty} \Di{j}\MP^{-1} (\tau,x) g(\fpa(\tau,x)) =0$ we get~\eqref{lemma:expDh}.

Using again the homogeneity of $g$ we extend the regularity properties of $g$ from the domain $\Vr$ to $V$.
\end{proof}

\begin{proof}[End of the proof of Theorem~\ref{solutionlinearequation}]
Once Lemma~\ref{lemma:hdif} is proven, we can apply it to $  h$ with
$\nu=\m$, $\Qi=\Qa$ and $\Ti{}=\T$ to get that $  h$ is $\CC^1$.
Then we are ready to prove that indeed $  h$ is a solution
of~\eqref{modellinearequation}. From the expression of $  h$ and the
fact that $\Vr$ is positively invariant by $\fpa$ we can write
$$
h(\fpa(s, x))=M(s, x)\int_{\infty}^s M^{-1}( t, x)
\T(\fpa( t, x))\, d  t, \qquad x\in \Vr,
$$
where we have used~\eqref{secondpropertychi} with $\MP = M$. Taking
derivatives with respect to~$s$ we obtain
\begin{equation}\label{edohfpa}
 D  h(\fpa(s, x)) \pa(\fpa(s, x)) = \Qa(\fpa(s, x)) h(\fpa(s, x)) + \T(\fpa(s, x))
\end{equation}
and evaluating at $s=0$ we get~\eqref{modellinearequation}.

It remains to check the regularity of $  h$. Note that the analytic
case follows directly from Lemma~\ref{lemma:hanalytic}. For the
differentiable case, we proceed by induction. Assume then that
$\pa,\Qa$ and $\T$ are $\CC^r$. Let $\rpa\leq r $ be the degree
of differentiability stated in Theorem \ref{solutionlinearequation}
depending on the values of $\BQ, \Apa, \cpa$ and $\dpa$.

We introduce some notation. Let $\Qaj{0}=\Qa$, $\TT{0}=\T$, $\HH{0}=
h$ and for $l\geq 1$
\begin{equation*}
\Qaj{l}( x) = \BMa{\Qaj{l-1}}( x)-\IMa{D\pa}{n^{l-1}\cdot k}( x)
=\text{diag}(\Qaj{l-1}( x), \dots, \Qaj{l-1}( x)) -
\IMa{D\pa}{n^{l-1}\cdot k}( x),
\end{equation*}
where $\BMa{\Qaj{l-1}}$ and  $\IMa{D\pa}{n^{l-1}\cdot k}$ were
defined in~\eqref{defBlockQ} and~\eqref{defDpext} respectively. We
denote by $M^l( t, x)$ the fundamental matrix of
$$
\frac{d}{d t} \psi = \Qaj{l}(\fpa( t, x)) \psi,\qquad \text{such
that } \quad M^{l}(0, x)=\Id.
$$
In addition we set
\begin{align*}
\TT{l}( x) &= \opC(D \TT{l-1}( x))+ \left ( (\Di{1}\Qaj{l-1}( x)
\HH{l-1}( x) )^{\top}, \dots, (\Di{n}\Qaj{l-1}( x) \HH{l-1}(
x))^{\top}\right )^{\top},
\\\HH{l}( x)&= \opC(D\HH{l-1}( x)),
\end{align*}
provided the derivative exists, where the linear operator $\opC$ is
defined in~\eqref{defopC}. It is clear that
$$
\Qaj{l}( x)\in \L{}(\RR^{n^l\cdot k},\RR^{n^l \cdot k}),\quad
\Qaj{l}\in \Hog{N-1}\cap\CC^{r-1},\quad \HH{l}( x)\in
\RR^{n^l \cdot k},\quad \TT{l}( x) \in \RR^{n^l \cdot k}.
$$
We claim that for $0\leq i\leq \rpa$ we have
\begin{enumerate}
\item [$(a)_i$] $\BQil{\Qaj{i}} \geq \BQ + i \Apa$.
\item [$(b)_i$] $\TT{i}\in \Hog{\m+N-i}$ and $\TT{j}\in \CC^{i+1-j}$ for $0\leq j \leq i$.
\item [$(c)_i$] $\HH{i}\in \Hog{\m+1-i}$, $\HH{j}\in \CC^{i-j}$ for $0\leq j\leq i$ and
\begin{equation}\label{expHHl}
\HH{i}( x)=\int_{\infty}^0 (M^{i})^{-1}( t, x) \TT{i}(\fpa( t, x))\,
d t.
\end{equation}
\end{enumerate}
We prove the claim by induction on $i$. The case $i=0$ follows
directly from the definitions and Lemma~\ref{lemma:hanalytic}.
Assume the claim holds for $i-1$, $1\leq i\leq \rpa-1$. Item
$(a)_i$ follows from Lemma~\ref{boundQ1} applied to
$\Qi=\Qaj{i}=\BMa{\Qaj{i-1}}( x)-\IMa{D\pa}{n^{i-1}\cdot k}( x)$
which gives, together with the induction hypothesis $\BQil{\Qaj{i}}
\geq \BQil{\Qaj{i-1}} + \Apa \geq \BQ + i \Apa$.

Item $(b)_i$. Since, by the induction hypothesis, $\TT{i-1}$ is at least
$\CC^2$, from the definition of $\TT{i}$ we have that $\TT{i}\in
\Hog{\m+N-i}$. From $j=0$, $\TT{0}=\T\in \CC^{r}\subset \CC^{i+1}$.
If $1\leq j \leq i$, using $(b)_{i-1}$ and $(c)_{i-1}$,
\begin{align*}
\TT{j}( x) &= \opC(D \TT{j-1}( x))+ \left ( (\Di{1}\Qaj{j-1}( x)
\HH{j-1}( x) )^{\top}, \dots, (\Di{n}\Qaj{j-1}( x) \HH{j-1}(
x))^{\top}\right )^{\top} \\ &\in \CC^{i+1-j}.
\end{align*}

Item $(c)_i$. We apply Lemma~\ref{lemma:hdif} with
$\Qi=\Qaj{i-1}$, $\Ti{}=\TT{i-1}$ and $\nu=\m-i+1$ so that $\DQi
=\Qaj{i}$, $\M{1}=M^i$ and $\Ti{1}=\TT{i}$. We have to
check~\eqref{hypnuhdif}. For that we will use that $i\leq \rpa$
and~\eqref{difcondformal}. Let $\cpa^{i-1}$ be the
constant $\cpa$ corresponding to $\Qaj{i-1}$ (see definition~\eqref{defconstantspa}).

When $\Apa<\dpa$,
$$
\nu+1+ \frac{\BQil{\Qaj{i-1}}}{\cpa^{i-1}} \geq  \m-i+2 + \frac{\BQ}{\cpa} +
(i-1) \frac{\Apa}{\dpa} >  1-\frac{\Apa}{\dpa}>0.
$$
When $\Apa\geq \dpa$,
$$
\nu +1+ \frac{\BQil{\Qaj{i-1}}}{\cpa} \geq  \m-i+2 + \frac{\BQ}{\cpa} +
(i-1) \frac{\Apa}{\dpa} > (i-1)\left ( \frac{\Apa}{\dpa} -1\right )\geq 0.
$$
Then $\HH{i-1}\in \CC^{1}$ and $\HH{i}=\opC\big (D\HH{i-1}( x)\big
)$ can be written as~\eqref{expHHl}. Therefore, by the definition of
$\HH{j}$, $\HH{j}\in \CC^{i-j}$, $0\leq j \leq i$, and the claim is
proven.

As a consequence of the claim, we have that $  h \in \CC^{\rpa}$ in
$\Vr$ in all cases. By the homogeneity we extend the regularity from
$\Vr$ to $V$. When $\Apa\geq \bpa$, if $r=\infty$, we also obtain $
h\in \CC^{\infty}$.
\end{proof}

\begin{proof}[Proof of Corollary~\ref{cor:theoremlinearequation}]
Assume that we have a homogeneous solution $  h\in \Hog{\nu}$ of
equation~\eqref{modellinearequation}. Then, it has to satisfy the
ordinary differential equation~\eqref{edohfpa} so that
\[
M^{-1}( t, x)  h(\fpa( t, x)) =   h( x) + \int_{0}^{ t} M^{-1}(s, x)
\T(\fpa(s, x))\, d s.
\]
Since $  h \in \Hog{\nu}$, by Lemmas~\ref{cotafpa} and~\ref{cotaMQ},
$$
\Vert M^{-1}( t, x)   h(\fpa( t, x)) \Vert \leq \big (1+\apa (N-1)
t \Vert  x \Vert^{N-1}\big )^{-\a\left (\frac{\BQ}{\cpa}+\nu\right
)}
$$
which is bounded as $ t \to \infty$ provided $\BQ/\cpa +\nu\geq 0$.
Thus, the result is proven.
\end{proof}

\section{Proof of Theorems~\ref{formalsolutionprop} and~\ref{maintheoremflow}}\label{formalsolutionsection}

As we will see in Section~\ref{formalsolutionflow} below,
Theorem~\ref{maintheoremflow} can be deduced following the same
lines as Theorem~\ref{formalsolutionprop}. For that reason we first
focus on the maps case.

We first notice that, for $ R$ such that $ R ( x) -(  x + p(
x,0))\in \Hom{N+1}$ then, by Lemma~\ref{invomega}, $
R(\Vr)\subset \Vr$ (taking $\r$ slightly smaller if necessary). Hence, if the domain of $ K$ is $\Vr$ (as we
will see), the composition $ K \circ  R$ is always well defined.
Moreover, for $ K$ such that $ K( x) - ( x,0) \in \Hom{2}$, if $ x
\in \Vr$ then $ K( x) \in U$ and consequently $F\circ  K $ is well
defined as well.

For~$h$ such that its projections have different orders, we will
write $h \in \Hom{l_1} \times \Hom{l_2}$ if $h_x\in \Hom{l_1}$ and
$h_y\in \Hom{l_2}$. We will use the same notation for the spaces
$\ho{l}$ and $\Hog{l}$.

\subsection{Preliminaries of the induction procedure: the cohomological equations}\label{subsectionstep1}
Given $N\leq \ell \leq r$ and $j\in \NN$ such that $1\leq j \leq
\ell-N+1$ we proceed by induction over $j$ to prove first that there
exist $\KMil{j}$ and $\RMil{j+N-1}$ of the form~
\begin{equation}\label{Step1KfRf}
\KMil{j}( x) = \sum_{l=1}^{j}  \Kl{l}( x), \qquad \RMil{j+N-1}( x) =
 x+\sum_{l=N}^{j+N-1} \Rl{l}( x),
\end{equation}
with $\Kl{1}(x)=(x,0)^{\top}$ and $\Rl{N}( x)= p( x,0)$, satisfying
\begin{equation}\label{HIKlbis}
\Em{j} := F\circ \KMil{j}  - \KMil{j} \circ \RMil{j+N-1} \\
 =(\Em{j}_x, \Em{j}_y)  \in \ho{j+N-1} \times
\ho{j+L-1}.
\end{equation}

Concerning property~\eqref{HIKl} in
Theorem~\ref{formalsolutionprop}, if $L=N$, it is a consequence
of~\eqref{HIKlbis} taking $j=\ell-N+1$. If $L=M<N$, we have to
perform an extra induction procedure for values of $j$ such that
$\ell-N+2 \leq j \leq \ell-L+1$.

The case $j=1$ follows immediately taking $\KMil{1}( x) = (
x,0)^{\top}$ and $\RMil{N}( x) =  x+p(x,0)$. Indeed:
\begin{align*}
\Em{1}_x( x)&=  x+p( x,0)+f( x,0) -  \RMil{N}( x)=f( x,0)\in \Hom{N+1} \subset \ho{N}, \\
\Em{1}_y( x) &=g( x,0)\in \Hom{M+1} \subset \ho{L},
\end{align*}
where we have used that, by hypothesis H2, $q( x,0)=0$.

Suppose that \eqref{HIKlbis} holds true for $j-1\geq 1$,
$\KMil{j-1}$ and $\RMil{j+N-2}$. We will find the condition that $\Kl{j}\in \Hog{j}$ and $\Rl{j+N-1} \in
\Hog{j+N-1}$ have to satisfy in order to ensure that \eqref{HIKlbis}
holds for $j$, $\KMil{j} = \KMil{j-1}+\Kl{j}$ and $\RMil{j+N-1}=
\RMil{j+N-2}+\Rl{j+N-1}$.

We claim that, since $j-1+N\leq \ell\leq r$, there exists
$\Ef{}=(\Ef{x}^{j+N-1}, \Ef{y}^{j+L-1})$ with $\Ef{x}^{j+N-1}\in
\Hog{j+N-1}$ and $\Ef{y}^{j+L-1}\in \Hog{j+L-1}$ such that
\begin{equation}\label{decompresidue}
\Em{j-1}_x - \Ef{x}^{j+N-1} \in \ho{j+N-1} , \qquad \Em{j-1}_y -
\Ef{y}^{j+L-1} \in \ho{j+L-1}.
\end{equation}
Indeed, by Taylor's theorem
\begin{equation}
\label{Taylor:homog}
\begin{aligned}
F_x( x,y) &=  x +  p( x,y)+  F_{ x}^{N+1}( x,y) + \cdots +   F_{ x}^{r}( x,y)+   F_{ x}^{>r}( x,y), \\
F_y( x,y) &= y+ q( x,y) +  F_{y}^{M+1}( x,y) + \cdots+
 F_{y}^{r}( x,y)+   F_{y}^{>r}( x,y),
\end{aligned}
\end{equation}
with $ F_{ x}^{l},  F_{y}^l \in \Hog{l}$ and $ F_{ x}^{>r},
 F_{y}^{>r} \in \ho{r}$. Moreover, $\KMil{j-1}$ and
$\RMil{j+N-2}$ are sums of homogeneous functions.  By the induction
hypothesis it is easily checked that
\begin{align*}
\Em{j-1}_x &= F_x \circ \KMil{j-1} - \KMil{j-1}_x \circ \RMil{j+N-2}
= \El{j+N-1}{x}+  \Ew{j}{x},
\\
\Em{j-1}_y &= F_y \circ \KMil{j-1} - \KMil{j-1}_y \circ \RMil{j+N-2}
= \El{j+L-1}{y}+ \Ew{j}{y}
\end{align*}
with $\El{l}{x,y} \in \Hog{l}$ and $\Ew{j}{x} \in \ho{j+N-1}$,
$\Ew{j}{y} \in \ho{j+L-1}$ and hence \eqref{decompresidue} is
satisfied. We decompose $ F\circ \Kff - \Kff \circ \Rff$ as
\begin{align*}
 F\circ \Kff - \Kff \circ \Rff = &\Em{j-1}{} +\big [  F \circ \Kff -  F \circ \KMil{j-1} - D F(\KMil{j-1})\cdot \kf\big ] \\
&+D F(\KMil{j-1}) \cdot \kf - \kf \circ \RMil{j+N-2}\\& -\big [\Kff
\circ \Rff-\Kff \circ \RMil{j+N-2}\big ] .
\end{align*}
Next we study each term of the above decomposition. In doing that we
introduce several new remainders $ e_i$. By Taylor's theorem, and
using that $j-1\geq 1$,
\begin{align*}
 e_1 &:=  F \circ \Kff -  F \circ \KMil{j-1} - D F(\KMil{j-1})\cdot \kf \in \Hom{N-2 + 2j}\times \Hom{M-2+ 2j}\\
&\subset \ho{j+N-1}\times \ho{j+L-1}.
\end{align*}
We denote $\id(x)=(x,0)$. Taking into account that $\KMil{j-1}-\id
\in \Hom{2}$ we can write
\begin{equation*}
D F(\KMil{j-1}) \cdot \kf = D F \circ \id \cdot \kf + e_2=
\left (\begin{array}{c} \big [ \Id + D_{x} p \circ \id \big]\cdot \kfp{x} + D_{y}p \circ \id\cdot \kfp{y} \\
\big [\Id + D_y q \circ \id\big ] \cdot \kfp{y}\end{array} \right )
+  e_2,
\end{equation*}
with $ e_2 \in \Hom{j+N}\times \Hom{j+M}\subset \ho{j+N-1}\times
\ho{j+L-1}$. Since $\RMil{j+N-2}( x) - x-p( x,0) \in \Hom{N+1}$ and
$N\geq 2$,
\begin{equation*}
\kf\circ \RMil{j+N-2}( x) = \kf( x) + D\kf( x)\cdot  p( x,0) +
 e_{3}( x)
\end{equation*}
with $ e_{3} \in \Hom{j-2+2N} \cup \Hom{j+N} \subset \ho{j+N-1}$. Finally
\begin{align*}
\Kff \circ \Rff-\Kff \circ \RMil{j+N-2} &= D\Kff (\RMil{j+N-2}) \cdot
\rf +  e_{4} \\ &= \left ( \begin{array}{c} \rf \\ 0 \end{array}\right )
+  e_{5} +  e_{4},
\end{align*}
where $ e_{4} \in \Hom{2(j+N-1)}\subset \ho{j+N-1}$ and $ e_{5} \in
\Hom{j+N}\subset \ho{j+N-1}$.

In conclusion, $ e_{l} \in \ho{j+N-1} \times \ho{j+L-1}$ for
$l=1,\cdots, 5$. Using \eqref{decompresidue} and the previous
computations, we have that
\begin{multline*}
 F\circ \Kff - \Kff \circ \Rff \\ = \begin{pmatrix} \El{j+N-1}{x} \\
\El{j+L-1}{y}
\end{pmatrix}+ \left (\begin{array}{c} D_{x} p \circ \id
\cdot \kfp{x} + D_{y}p \circ \id \cdot \kfp{y} -\rf \\
D_y q \circ \id\cdot \kfp{y}\end{array} \right )
- D \kf \cdot  p \circ \id + \Et{j}{},
\end{multline*}
where $\Et{j}{} = \Em{j-1}{} - (\El{j+N-1}{x}, \El{j+L-1}{y})^{\top}
+  e_1+  e_2- e_3- e_4 - e_5 \in \ho{j+N-1}\times \ho{j+L-1}$.

In order to get property \eqref{HIKlbis} for $j$, we have to choose
$\Kl{j}\in \Hog{j}$ and $\Rl{j+N-1}\in\Hog{j+N-1}$ such that
\begin{equation}\label{eqlKxfirst}
D\Kl{j}_{ x}( x) \cdot  p( x,0) -D_x p( x,0)\cdot \Kl{j}_{ x}( x) -
D_{y}p( x,0)\cdot \Kl{j}_{y}( x)+ \Rl{j+N-1}( x)= \Ef{ x}^{j+N-1}(
x)
\end{equation}
and, taking into account that $M$ and $N$ may be different,
\begin{equation}\label{eqlKy}
D\Kl{j}_{y}( x) \cdot  p( x,0) -D_y q( x,0)\cdot \Kl{j}_{y}( x)
-\Ef{y}^{j+L-1}( x) \in \ho{j+L-1}.
\end{equation}

As usual in the parametrization method we have a lot of freedom to
choose solutions of the above equations. On the one hand, we expect
that equation \eqref{eqlKy} for $\Kl{j}_y$ has a unique homogeneous
solution. On the other hand, it is clear that equation
\eqref{eqlKxfirst} for $\Kl{j}_x$ and $\Rl{j+N-1}$ admits several
homogenous solutions. Despite the fact that we could solve
first~\eqref{eqlKy} for $\Kl{j}_{y}\in \Hog{j}$ and then, take
$\Kl{j}_{ x}\equiv 0$ and
\[
\Rl{j+N-1}(x) = \Ef{ x}^{j+N-1}( x)+ D_{y}p( x,0)\cdot \Kl{j}_{y}(
x)
\]
to solve~\eqref{eqlKxfirst}, we are also interested in looking for
the simplest representation of the dynamics on the stable manifold,
that is, we ask $\Rff$ to be as simple as possible, for instance
taking $\Rl{j+N-1}=0$ if we can solve the following equation
\begin{equation}\label{eqlKxn}
D\Kl{j}_{ x}( x) \cdot  p( x,0) -D_x p( x,0)\cdot \Kl{j}_{ x}( x) =
\Ef{ x}^{j+N-1}( x)+ D_{y}p( x,0)\cdot \Kl{j}_{y}( x).
\end{equation}

We distinguish three cases to obtain an equation for $\Kl{j}_{y}$ so
that condition~\eqref{eqlKy} holds:
\begin{itemize}
\item If $N<M$, then condition~\eqref{eqlKy} is satisfied if
\begin{equation} \label{eqlKyN<Mn}
D\Kl{j}_{y}( x) \cdot  p( x,0) = \Ef{y}^{j+L-1}( x).
\end{equation}
\item If $N=M$,
\begin{equation}\label{eqlKyN=Mn}
D\Kl{j}_{y}( x) \cdot  p( x,0) -D_y q( x,0)\cdot \Kl{j}_{y}( x)
=\Ef{y}^{j+L-1}( x).
\end{equation}
\item If $N>M$, then we get an algebraic equation:
\begin{equation}\label{eqlKyN>Mn}
-D_y q( x,0)\cdot \Kl{j}_{y}( x) =\Ef{y}^{j+L-1}( x)
\end{equation}
which can be solved by using that, by hypothesis H2, $D_y q( x,0)$
is invertible. We also have that $[D_y q( x,0)]^{-1}\in \Hog{-M+1}$.
This equation clearly illustrates the fact that the solutions
$\Kl{j}$ are not necessarily polynomials.
\end{itemize}

Assume that we are able to find appropriate solutions $\Kl{j}_{x}$
of equation~\eqref{eqlKxn} and $\Kl{j}_{y}$ of~\eqref{eqlKyN<Mn},
\eqref{eqlKyN=Mn} or~\eqref{eqlKyN>Mn}. We recall that we were
dealing with values of $j=2,\cdots, \ell-N+1$. When $L=N\leq M$,
\eqref{HIKl} and~\eqref{formf} follows from \eqref{HIKlbis} by
taking $j=\ell-N+1$ so in this case we are done. However, in the
case $L=M<N$ we also have to deal with the equation for $\Kl{j}$
when $j=\ell-N+2,\cdots,\ell-L+1$. That is, we need to add some
extra homogeneous terms to $ K_y$ to obtain \eqref{HIKl} and
\eqref{formf}. Indeed, for any given $\ell$, assume that
$\KMil{\ell-N+1}, \RMil{\ell}$ are of the form~\eqref{Step1KfRf} and
they satisfy~\eqref{HIKlbis} for $j=\ell-N+1$. We prove by induction
on $j$ that, for any $\ell-N+2\leq j \leq \ell-L+1$, we can find
$$
\KMil{j} = \KMil{\ell-N+1} + \sum_{l=\ell-N+2}^{j} \Kl{l}, \qquad
\Kl{l} \in \Hog{l},\quad \text{with} \quad \Kl{l}_x\equiv 0
$$
in such a way that $\Em{j} =F\circ \KMil{j} - \KMil{j} \circ
\RMil{\ell} \in \ho{\ell}\times \ho{j+L-1}$.

Assume that the result holds for $j-1$. Then, since $j+L-1\leq
\ell\leq r$, decomposition \eqref{decompresidue} of $\Em{j-1}_y$ is
also true in this case. Taking $\kfp{x},\rf \equiv 0$ in the above
computations we also have that
$$
 F_y \circ \KMil{j} -  \KMil{j}_y \circ  R = -D\Kl{j}_y \cdot p
\circ \id + D_y q\circ \id\cdot \Kl{j}_y +\Ef{y}^{j+L-1}+ \Et{j}{y}
$$
with $\Ef{y}^{j+L-1} \in \Hog{j+L-1}$, $\Et{j}{y}\in \ho{j+L-1}$ and
$$
\Em{j}_x =  F_x \circ \KMil{j} -  \KMil{j}_x \circ  R = D_{y}p
\circ \id\cdot \Kl{j}_{y} + \Et{j}{x},
$$
with $\Et{j}{x} \in \ho{j+N-1}\subset \ho{\ell}$.

Since $M<N$, if $\Kl{j}_{y}\in \Hog{j}$ and satisfies the equation
$$
D_y q( x,0)\cdot \Kl{j}_y( x)=-\Ef{y}^{j+L-1}( x),
$$
then $D_{y}p \circ \id \cdot \Kl{j}_{y} \in \Hom{j+N-1}\subset
\ho{\ell}$ and $\Em{j}_x\in \ho{\ell}$. Therefore, we can follow
this procedure $N-L$ times until \eqref{decompresidue} holds true.
After that, the order of the remainder $\Em{\ell-L+1}$ will be
$\ell$ and $ K_x$ will have the form given in
Theorem~\ref{formalsolutionprop} and property~\eqref{HIKl} will be
satisfied.

We remark that the equation for $\Kl{j}_{y}$,
$j=\ell-N+2,\cdots,\ell-L+1$, is the same algebraic equation
\eqref{eqlKyN>Mn} as the one corresponding to $j=2,\cdots, \ell-N+1$.

\subsection{Resolution of the linear equations~\eqref{eqlKyN<Mn}-\eqref{eqlKyN>Mn} for $\Kl{j}_y$}\label{subsectionstep3}
We take $2\leq j \leq \ell-L+1$. In the case $M<N$, $\Kl{j}_y$ is a
solution of the algebraic equation~\eqref{eqlKyN>Mn}. Since $D_{y}q(
x,0)$ is invertible, the unique solution of this equation is
\begin{equation*}
\Kl{j}_y( x) = -\big ( D_{y}q( x,0) \big )^{-1} \Ef{y}^{j+L-1}( x).
\end{equation*}
Clearly, $\Kl{j}_y$ is a homogeneous function of order $j$ which is
analytic in $V$. Nevertheless, it is only $j-1$ times differentiable
at the origin according to Definition~\ref{difrem}.

Let $M\geq N$. In this case $\Kl{j}_{y}$ has to satisfy either
equation~\eqref{eqlKyN<Mn}, if $N<M$, or~\eqref{eqlKyN=Mn}, if
$N=M$. We write them in a unified way as
$$
D\Kl{j}_{y}( x) \cdot  p( x,0) - \Qa( x)\cdot \Kl{j}_{y}( x)
=\Ef{y}^{j+L-1}( x),
$$
where $\Qa( x)=0$ if $N<M$ and $\Qa( x) = D_y q( x,0)$ if $M=N$.
Hence this case follows from Theorem~\ref{solutionlinearequation}
taking $\pa( x) =  p( x,0)$ and $\Qa$ as indicated. We claim that
under the current hypotheses, $\pa$ and $\Qa$ satisfy the conditions
of Theorem~\ref{solutionlinearequation}. Indeed, the constants
$\Apa,\apa,\bpa$ in Theorem~\ref{solutionlinearequation} are
$$
\apa = \ap >0 \quad \text{(by H1)}, \quad \bpa = \bp >0 \quad
\text{(by definition)}\quad \Apa = \Ap.
$$
As for $\BQ$, by definition~\eqref{defconstants}, if $M>N$, $\BQ=0$.
If $M=N$, $\BQ=\Bq$ and by hypotheses H1 and H2 the condition
$j+\frac{\BQ}{\cpa} > \max\left\{1-\frac{\Apa}{\dpa},0\right \}$ is
satisfied in both cases. Then Theorem~\ref{solutionlinearequation}
provides a solution $\Klp{j}{y} \in \Hog{j}$ for $2\leq j \leq \ell-L+1$.

\subsection{Resolution of the linear equation~\eqref{eqlKxfirst} for $\Kl{j}_x$}\label{sec:linearequationx}
Consider $2 \leq j \leq \ell-N+1$. We have to find $\Klp{j}{x}$
satisfying equation~\eqref{eqlKxfirst} which we recall here:
\begin{equation*}
D\Kl{j}_{ x}( x) \cdot  p( x,0) -D_x p( x,0)\cdot \Kl{j}_{ x}( x) +
\Rl{j+N-1}(x)= \Ef{ x}^{j+N-1}( x)+D_{y}p( x,0)\cdot \Kl{j}_{y}( x)
\end{equation*}
being $\Ef{ x}^{j+N-1}$ a homogenous function of order $j+N-1$ and
$\Klp{j}{y}\in \Hog{j}$ the
solution of the linear equation considered in
Section~\ref{subsectionstep3}. Since $D_y p \circ \id  \cdot
\Klp{j}{y}\in \Hog{j+N-1}$ we can add this term to $\Ef{ x}^{j+N-1}$
and denote the resulting term again by $\Ef{ x}^{j+N-1}$ and hence
we end up with equation
\begin{equation}\label{eqx}
D\Kl{j}_{ x}( x) \cdot  p( x,0) -D_x p( x,0)\cdot \Kl{j}_{ x}( x) +
\Rl{j+N-1}( x)= \Ef{ x}^{j+N-1}( x).
\end{equation}
As we mentioned in Section~\ref{subsectionstep1}, to
solve~\eqref{eqx}, one possibility is to take $\Klp{j}{x}$ as any
function in $\Hog{j}$ and $\Rl{j+N-1}$ as the solution of the
resulting equation. If we proceed in this form, we are always able
to solve the equation, but we do not have a normal form result for $
R$ in the sense that $ R$ is not simple at all. In the other
extreme, we can try to choose $\Rl{j+N-1}=0$ and use
Theorem~\ref{solutionlinearequation} with $\pa( x) =  p( x,0)$ and
$\Qa( x)=D_{x}p( x,0)$ to solve
\begin{equation}\label{eqR}
D\Klp{j}{x}( x) \cdot p( x,0)-D_{x} p( x,0) \cdot \Klp{j}{x}( x)
=\Ef{ x}^{j+N-1}( x), \;\;\;\; \text{for  }\Klp{j}{x}.
\end{equation}
However, this equation may not have solutions if $j$ is not large
enough. Indeed, in this case, since $\pa(x) = p(x,0)$ and $\Qa( x) =
D_x p( x,0)$, by hypothesis H1 and Lemma~\ref{defKapwell} the
constant $\BQ = -\Bp \leq -N\ap<0$ and hence equation~\eqref{eqR}
can not be solved unless $j$ is large enough. Concretely, the sufficient
condition to have solutions is $j-\frac{\Bp}{\ap} >
\max\left\{1-\frac{\Ap}{\ddp},0\right\}$. Therefore, if $j\in \NN$
satisfies
\begin{equation*}
j> \frac{\Bp}{\ap}+\max\left\{1-\frac{\Ap}{\ddp},0\right\},
\end{equation*}
equation~\eqref{eqR} has a unique homogeneous solution $\Klp{j}{ x}
\in \Hog{j}$.

In conclusion, if $j>\frac{\Bp}{\ap}+\max\{1-\frac{\Ap}{\ddp},0\}$,
we take $\Rl{j+N-1} \equiv 0$ and $\Klp{j}{ x}$ a homogeneous
solution of \eqref{eqR}. Otherwise, $\Klp{j}{ x}$ is free and we
take as $\Rl{j+N-1}$ the solution of~\eqref{eqx}.

\subsection{Regularity of $K^{j}$ and $R^{j+N-1}$}\label{sec:regularity}
When $\Ap>\ddp$, since $(\El{N+1}{x},\El{M+1}{y})$ is analytic, from Theorem~\ref{solutionlinearequation}, $\Kl{2}$ is an analytic function in
$V$ and consequently, by induction $\Kl{j}$ is also analytic.

If $M<N$, we solve equation~\eqref{eqx} for $j=2$ by taking
$\Klp{2}{x}\equiv 0$ and $\Rl{N+1}=\Ef{ x}^{N+1}$. Hence $\Kl{2}$ is
analytic, since $(\El{N+1}{x},\El{M+1}{y})$ is analytic. Then by
induction $\Kl{j}$ is also analytic provided we solve
equation~\eqref{eqx} in some appropriate way, for instance, by taking $\Klp{j}{x}\equiv 0$
and $\Rl{j+N-1}=\Ef{ x}^{j+N-1}$.

In the case $\Ap= \ddp$ and $M\geq N$, even if
$(\El{N+1}{x},\El{M+1}{y})$ is analytic,
Theorem~\ref{solutionlinearequation} only provides $\CC^{\infty}$
solutions in $V$. Consequently, $\Kl{2}$ is only $\CC^{\infty}$ and
inductively we obtain that $\Kl{j}$ is $\CC^{\infty}$.

Finally we consider the case $\Ap< \ddp$ and $M\geq N$, where we lose
regularity. Concretely, $\Klp{2}{y} \in C^{\gdf}$, on $V$ where
$\gdf$, given in~\eqref{defminimaregularitat}, is the maximum
integer $k$ such that
$$
\left(1- \frac{\Ap}{\ddp}\right) k <   2 + \frac{\Bq}{\cp}.
$$
In order to deal with the cases $M=N$ and $M>N$ jointly, from now on
we understand that $\Bq=0$ if $M>N$. Recall that $\Klp{j}{x}\equiv
0$ for $j=2, \dots , \df-N+1$ so that the differentiability of
$\Klp{j}{y}$ for these values of $j$ only depends on the smoothness
of $\Klp{i}{y}$, for $i=2,\cdots, j-1$, which is $\gdf$ by induction.

When $j=\df-N+2$, from Theorem~\ref{solutionlinearequation},
we have that $\Klp{\df-N+2}{ x}\in \CC^{r_{x} }$, with $r_{x} $ the
maximum integer $k$ satisfying
$$
\left(1- \frac{\Ap}{\ddp}\right) k<   \df -N+2 - \frac{\Bp}{\cp}.
$$
The maximum differentiability, $\gdf$, is obtained by choosing $\df$ to be
the smallest integer satisfying
$$
\df > \left(1- \frac{\Ap}{\ddp}\right) \gdf +N-2 + \frac{\Bp}{\cp}
$$
which justifies the definition~\eqref{degreefreedom} of $\df$ in the present case.

By induction one checks that $\Kl{j}=(\Klp{j}{x},\Klp{j}{y})$ is
also a $\CC^{\gdf}$ function.

By construction, $\Rl{j+N-1}$ has the same regularity in all cases.

\subsection{The flow case. Proof of Theorem~\ref{maintheoremflow} without parameters}\label{formalsolutionflow}
In the case of flows we have to find $K^{\leq j}(x,t)$ and $Y^{\leq j+N-1}(x)$ of the form
\begin{equation*}
\KMil{j}( x,t) = \sum_{l=1}^{j}  \Kalg{l}( x,t), \qquad
\YMil{j+N-1}( x) = \sum_{l=N}^{j+N-1} \Yl{l}( x)
\end{equation*}
being $ K_{1}( x,t) = (x,0)^{\top}$, $\Yl{N}( x)= p( x,0)$. For
technical reasons, we look for $\Kalg{l}$ as a sum of two
homogeneous functions: one of degree~$l$ independent of~$t$ and the
other belonging to $\ho{l+N-1} \times \ho{l+L-1}$. The homogeneous
terms $K^l$ in the statement of the theorem are obtained by
rearranging the sum above. $\KMil{j}$ have to satisfy the invariance
condition~\eqref{homequationflow} up to some order~$j$ in the sense
that the error term
\begin{equation*}
\Em{j}( x,t):= X(\KMil{j}( x,t),t)- D \KMil{j}( x,t) \YMil{j+N-1}(
x) -
\partial_t \KMil{j}( x,t)
\end{equation*}
satisfies
\begin{equation}\label{HIKlbisflow}
\Em{j} = (\Em{j}_x,\Em{j}_y) \in \ho{j+N-1} \times \ho{j+L-1}.
\end{equation}
As we have noticed in Section~\ref{subsectionstep1}, in the case $L=N$
condition~\eqref{HIKlbisflow} implies \eqref{HIklflow}.

Following the same induction arguments as in
Section~\ref{subsectionstep1} we obtain that $\Yl{j}$ and
$\Kalg{j}=(\Kalg{j}_x, \Kalg{j}_y)$ have to satisfy the conditions:
\begin{multline}
\label{Kxellflow} D\Kalg{j}_{x} (x,t)p(x,0) - D_x p(x,0)
\Kalg{j}_{x}(x,t) -D_y p(x,0)\Kalg{j}_{y}(x,t) \\ + \Yl{j+N-1} (x) +
\partial_t \Kalg{j}_{x}(x,t)  - \Ef{x}^{j+N-1}(x,t) \in
\ho{j+N-1},
\end{multline}
and
\begin{equation}
\label{Kyellflow} D\Kalg{j}_{y}(x,t) p(x,0) - D_y q(x,0)
\Kalg{j}_{y}(x,t)  +\partial_t \Kalg{j}_{y}(x,t) -  \Ef{y}^{j+L-1}(
x,t) \in \ho{j+L-1}
\end{equation}
which are the analogous in the case of flows
for~\eqref{eqlKxfirst} and~\eqref{eqlKy} respectively. We will skip
the computations which are pretty similar as the ones in the
previous section. However we will not ask $\Kalg{j}_x, \Kalg{j}_y$ to
satisfy their corresponding partial differential equation (vanishing
the terms $\Ef{x}^{j+N-1}$ and $\Ef{y}^{j+L-1}$) but we allow them
to include new terms of higher order.

With this strategy in mind, we are going to explain how to solve
these equations. For a given $T$-periodic function $h$, we denote by
$\ME{h}$ its average and by $\ZE{h}=h-\ME{h}$ its oscillatory part (with zero average).
Clearly, since we look for $\Kalg{j}$ periodic, one choice is to ask that the average $\ME{\Kalg{j}}$ satisfies the equations
\begin{equation}
\label{eq:inductionflowsmean}
\begin{aligned}
D\ME{\Kalg{j}_{x}} (x)p(x,0) - D_x p(x,0) \ME{\Kalg{j}_{x}}(x)  -D_yp(x,0)\ME{\Kalg{j}_y}(x) & \\
 + \Yl{j+N-1} (x) - \ME{\Ef{x}^{j+N-1}}(x) & =0,\\
D\ME{\Kalg{j}_{y}}(x) p(x,0) - D_y q(x,0)  \ME{\Kalg{j}_{y}}(x)  -
\ME{\Ef{y}^{j+L-1}}( x) & \in \ho{j+L-1}.
\end{aligned}
\end{equation}
We can solve equations~\eqref{eq:inductionflowsmean} as in the map
case, following the arguments in Sections~\ref{subsectionstep3}
and~\ref{sec:linearequationx} for solving
equations~\eqref{eqlKxfirst} and \eqref{eqlKy}. Concerning
regularity, the arguments in Section~\ref{sec:regularity} leads to
the same regularity as in the map case for the average of $\Kalg{j}$
and $\Yl{j}$. As a conclusion, we have solutions of
equations~\eqref{eq:inductionflowsmean} $\ME{\Kalg{j}}$ and
$\Yl{j+N-1}$ belonging to $\Hog{j}$ and $\Hog{j+N-1}$ respectively.

We take the oscillatory part $\ZE{\Kalg{j}}$ with zero average and
satisfying
\begin{equation}
\label{eq:oscillatorypart}
\partial_ t \ZE{\Kalg{j}}( x,t) = (\ZE{\Ef{x}^{j+N-1}}( x,t),\ZE{\Ef{y}^{j+L-1}}( x,t) ).
\end{equation}
Consequently, $\ZE{\Kalg{j}}\in \Hog{j+N-1}\times \Hog{j+L-1}$.

It only remains to see that $\Kalg{j}=\ME{\Kalg{j}} + \ZE{\Kalg{j}}$
and $\Yl{j+N-1}$ satisfy equations~\eqref{Kxellflow} and
\eqref{Kyellflow}. Indeed, when we compute the left-hand side of,
for instance, equation~\eqref{Kxellflow} we obtain
$$
D\ZE{\Kalg{j}_x}( x,t)p(x,0)-D_x p(x,0) \ZE{\Kalg{j}_x}(
x,t)-D_yp(x,0) \ZE{\Kalg{j}_y}( x,t)
$$
which belongs to $\Hog{j+L-1+N-1} \subset \ho{j+N-1}$ since $L\geq
2$. Analogously for equation~\eqref{Kyellflow}. Therefore, we
conclude that $\Kalg{j}=\ME{\Kalg{j}} + \ZE{\Kalg{j}}$ and
$\Yl{j+N-1}$ satisfy equations \eqref{Kxellflow} and
\eqref{Kyellflow} and then \eqref{HIKlbisflow} is satisfied.

The regularity of the oscillatory part follows from the fact that it
satisfies equation~\eqref{eq:oscillatorypart}.

As in Section~\ref{subsectionstep1}, if $L=N$, we are done. The case
$L=M$ needs an extra argument which is totally analogous to the one
in Section~\ref{subsectionstep1}.

\begin{remark}
The vector field $Y$ can be chosen independent of $t$. This is due
to the fact that we can perform the averaging procedure so that for
any given $\ell$ we can move the dependence on $t$ of the vector
field $X$ up to order $\Vert z \Vert^{\ell}$. If we take $\ell\geq
\ell_*+1$, then the formal procedure is independent of $t$ and we
obtain (for the averaged vector field $\overline{X}$) a
parametrization $\overline{K}^{\leq}$ and a vector field
$\overline{Y}$ satisfying the invariance
condition~\eqref{homequationflow} up to order $\Vert x \Vert^{\ell}$
which do not depend on $t$.

Nevertheless we can add $t$-depending terms to $Y$ in order to have a more simple $K_x$.
\end{remark}

\section{Dependence on parameters. Proof of Theorems~\ref{prop:param} and~\ref{maintheoremflow}.}
\label{sec:parameters}

In this section we prove Theorems~\ref{prop:param}
and~\ref{maintheoremflow} which give us the dependence on parameters
of the functions $ K$ and $ R$ given in Theorem~\ref{formalsolutionprop}
as a sum of homogeneous functions.

We first emphasize that the methodology developed in
Section~\ref{formalsolutionsection} can also be applied in the
parametric case so that the cohomological equations~\eqref{eqlKxfirst}
and~\eqref{eqlKy} for $\Kl{j}$ are the same in this context but
involving the dependence on $\lambda$. For a given value of the
parameter $\lambda$, the discussion about how to solve the
cohomological equations for $ \Kl{j}_{y}$ distinguishing the different
cases ($N>M$, $N=M$ and $N<M$) and the different strategies to solve the cohomological
equations for $\Kl{j}_{x}$ are also valid. Therefore, even in the parametric case,
the existence of $\Kl{j}(\cdot, \lambda)$ is already proven.
Next we study the regularity with respect to $\lambda$
both for maps and flows.

\subsection{The cohomological equation in the parametric case}

The case $N\geq M$, can be treated by using the auxiliary
equation~\eqref{modellinearequation} studied in
Section~\ref{subsectionstep2}. See the strategy of how to proceed in
Sections~\ref{subsectionstep3} and~\ref{sec:linearequationx}. As a
consequence, we are lead to deal with the dependence on parameters
of the homogeneous solution $h$
$$
h(x,\lambda) = \int_{\infty}^0 M^{-1}(t,x,\lambda) \T(\fpa(t,x,\lambda),\lambda)\dd t
$$
of the auxiliary equation:
\begin{equation}\label{modellinearequation:param}
\Dx   h( x,\lambda) \pa( x,\lambda) - \Qa( x,\lambda)   h(
x,\lambda)= \T( x,\lambda),
\end{equation}
given by Theorem~\ref{solutionlinearequation} for any $\lambda\in
\Lambda$, where $\pa,\Qa$ and $\T$ are homogeneous functions of
degree $N$, $N-1$, $\m+N$ respectively. We will write $\pa\in
\Hog{N}$, $\Qa\in \Hog{N-1}$ and $\T\in \Hog{\m+N}$.

In this setting, the constants defined in~\eqref{defconstantspa},
HP1 and HP2 depend on $\lambda$. We denote them by $\Apa^{\lambda}$,
$\AQ^{\lambda}$, $\BQ^{\lambda}$, $a_V^{\pa,\lambda}$, $\apa^{\lambda}$,
$\bpa^{\lambda}$, $\cpa^{\lambda}$ and $\dpa^{\lambda}$. In order to
obtain uniform bounds with respect to $\lambda$ we redefine
\begin{equation}\label{defconstantspa:param}
\begin{aligned}
&\apa=\inf_{\lambda \in \Lambda} \apa^{\lambda},\qquad &
\bpa&=\sup_{\lambda \in \Lambda} \bpa^{\lambda},\qquad
&\Apa&=\inf_{\lambda \in \Lambda} \Apa^{\lambda}, \\
&\BQ=\inf_{\lambda \in \Lambda} \BQ^{\lambda},\qquad
&\AQ&=\sup_{\lambda \in \Lambda} \AQ^{\lambda},\qquad &\CIP
&=\inf_{\lambda \in \Lambda} a_{V}^{\pa,\lambda}
\end{aligned}
\end{equation}
and $\cpa, \dpa$ as in~\eqref{defconstantspa}. Notice that, with
this definition of the constants, all the bounds in Section
\ref{subsectionstep2} will be also true uniformly for any $\lambda
\in \Lambda$.

To study equation~\eqref{modellinearequation:param}, we will assume
the following:
\begin{enumerate}
\item[HP$\lambda$] Hypotheses HP1 and HP2 hold true for
$\apa,\CIP$ defined in~\eqref{defconstantspa:param}.
\end{enumerate}

To deal with the analytic case, for $\gamma>0$, we define the
complex extension of~$\Lambda$
$$
\Lambda(\gamma)=\{ \lambda \;:\; \re \lambda \in \Lambda ,\; \Vert
\im \lambda \Vert < \gamma \}.
$$
\begin{lemma}\label{lemma:param}
Let $\pa \in \Hog{N}$, $\Qa \in \Hog{N-1}$ and $\T\in \Hog{\m+N}$.
Assume that $\pa, \Qa, \T \in \CC^{\DS}$ and  that $\pa$ satisfies
hypothesis HP$\lambda$ for $\r_0>0$.

Then, if
$$
\m+1+\frac{\BQ}{\cpa} > \max\left \{1-\frac{\Apa}{\dpa},0\right\},
$$
the solution $  h:V \times
\Lambda \to \RR^k$ of \eqref{modellinearequation:param}  provided by Theorem \ref{solutionlinearequation}
satisfies $h\in \Hog{\m+1}$ and we have the regularity results according to the cases:
\begin{enumerate}
\item[(1)] $\Apa \geq \dpa$. If $1\leq r\leq \infty$, then $  h\in \CC^{\DS}$ in
$V\times \Lambda$.
\item[(2)] $\Apa<\dpa$. Let $\kappa_0$ be the maximum of $1 \leq i \leq r+s$
such that
\begin{equation*}
\m + 1+\frac{\BQ}{\cpa}-i \left ( 1-\frac{\Apa}{\dpa}\right)>0.
\end{equation*}
Then $  h\in \CC^{\Sigma_{s_0,r_0}}$ in $V\times \Lambda$ with
$s_0=\min\{s,\kappa_0\}$ and $r_0=\kappa_0-s_0$.
\item[(3)] $\Apa>\dpa$. If $\pa, \Qa,\T$ are real analytic functions in
$\Omega(\gamma_0)\times \Lambda(\gamma_0^2)$ for some $\gamma_0$ then
$  h$ is analytic in $\Omega(\gamma)\times  \Lambda(\gamma^2)$ for
$\gamma$ small enough. In particular it is real analytic in $V\times
\Lambda$.
\end{enumerate}
\end{lemma}
To have an unified notation for all cases of the lemma, we introduce the differentiability degrees $\rpa, s_{\pa}$
as:
\begin{equation}\label{defkappapa}
\rpa=r, \;\; s_{\pa} = s,\quad\text{if } \Apa\geq \bpa,\qquad
\rpa=r_0,\;\; s_{\pa}=s_0, \quad \text{otherwise},
\end{equation}
where $r_0,s_0$ are defined in Lemma~\ref{lemma:param}.
In this way, in all cases $h\in \CC^{\Sigma_{\rpa,s_{\pa}}}$.

We proceed in a similar way as in Section~\ref{subsectionproofTheoremauxiliaryequation}.
We introduce the function
\begin{equation}\label{defgind:param}
 g( x,\lambda):=\int_{\infty}^0 M^{-1}( t, x,\lambda) \Ti{}(\fpa( t, x,\lambda),\lambda)\, d t,
\end{equation}
where $\fpa( t, x,\lambda)$ is the solution of
$\dot{x}=\pa(x,\lambda)$ such that $\fpa(0,x,\lambda)=x$, $M$ is the
fundamental matrix of $\dot{\psi} = \Qa(\fpa( t,
x,\lambda),\lambda)\psi$ such that $M(0, x,\lambda)=\Id$ and $\Ti{}$
satisfies appropriate conditions to be specified later on.
We first deal with the continuity of $g$ and then with the differentiability with respect to the parameter $\lambda$.
For that we check that the formal derivative $\Dl  g$ is of the same form as $ g$ with a suitable different $\Ti{}$
which implies the differentiability with respect to $\lambda$. This
is done jointly in Lemma~\ref{lemma:hdif:param}. Then using an
induction argument we deal with the general differentiable case.
Finally we deal with the analytic case, which needs an extra
argument in this parametric setting.

For a given set $\mathcal{U} \subset \RR^n \times \RR^{n'}$, it will
be useful to consider the functional spaces:
$$
\BH{\nu}{\sigma}{\kappa} = \{ h: \mathcal{U} \to \RR^k \;: \; h \in
\CC^{\Sigma_{\sigma,\kappa-\sigma}} \text{ and } h\in \Hog{\nu}\}
$$
if $\kappa,\sigma\in \ZZ^+$ and $\kappa\geq \sigma$.

\begin{remark} Note that $\BH{\nu}{s}{r+s}=\CC^{\DS} \cap \Hog{\nu}$.
\end{remark}
\begin{lemma}\label{lemma:hdif:param} Let $\kappa\geq \sigma$ with $\sigma=0,1$.
Assume that $\pa\in \BH{N}{\sigma}{\kappa}$, $\Qa\in \BH{N-1}{\sigma}{\kappa}$
and $\Ti{}\in \BH{\nu+N}{\sigma}{\kappa}$. Let
$\ro>0$ be such that H$\lambda$ holds true.

Then, if $ \nu+1+\frac{\BQ}{\cpa} > \max\left
\{1-\frac{\Apa}{\dpa},0\right\}$, the function $ g$ defined
by~\eqref{defgind:param} belongs to $\BH{\nu+1}{\sigma}{\kappa_{\pa}}$,
where $\kappa_{\pa}=\rpa+s_{\pa}$ and $\rpa, s_{\pa}$ are defined in~\eqref{defkappapa}.

In addition, when $\sigma=1$, $\Dl g$ exists and
\begin{equation}\label{lemma:expDh:param}
\Dl g( x,\lambda)= \int_{\infty}^0 M{}^{-1}( t, x,\lambda)
\Ti{1}(\fpa( t, x,\lambda),\lambda)\, d t
\end{equation}
with $\Ti{1}:V\times \Lambda \to \L{}(\RR^{n'},\RR^{k})$ (recall
that $\Lambda \subset \RR^{n'}$), given by:
\begin{equation}\label{defDifT:param}
\Ti{1}( x,\lambda)= \Dl \Ti{}( x,\lambda) + \Dl \Qa( x,\lambda)  g(
x,\lambda) - \Dx  g( x,\lambda) \Dl \pa( x,\lambda).
\end{equation}
\end{lemma}

\begin{remark}\label{remarkderlambda}
We observe that $\kappa_{\pa}$, the degree of differentiability
stated for the case $\sigma=0$, is the same as the one given in Theorem
\ref{solutionlinearequation}.
\end{remark}
\begin{remark}
Note that $\Dl \Qa (x,\lambda) g(x,\lambda) \in \L{}(\RR^{n'}
,\RR^k)$ having the $i$-th column
$$
\Di{\lambda_i}\Qa ( x,\lambda)  g( x,\lambda).
$$
The same happens for $\Dx \Qa(x,\lambda) g(x,\lambda)$.
\end{remark}
\begin{proof}
The case $\sigma=0$ follows from Theorem~\ref{solutionlinearequation},
the dominated convergence theorem and the fact that the bounds are
uniform in $\lambda\in \Lambda$.

The case $\sigma=1$ is more involved. Its proof is analogous to the proof of Lemma~\ref{lemma:hdif}.
To shorten the notation we introduce $\fpa_{\lambda}^x(t) :=\fpa(t,x,\lambda)$.
First we check that
$$
G^x_{\lambda}(\tau):=\int_{\tau}^0 \Dl \big [M^{-1}(t,x,\lambda) \Ti{}(\fpa_{\lambda}^x(t),\lambda)\big ]\dd t
$$
can be written as:
\begin{align}\label{dlamdahfirst}
G^x_{\lambda}(\tau):=& -\Dl M^{-1}(\tau,x,\lambda) g(\fpa_{\lambda}^x(\tau),\lambda)+\int_{\tau}^0 M^{-1}(t,x,\lambda) \Dl \big [\Ti{}(\fpa_{\lambda}^x(t),\lambda)]\dd t \notag\\
&+\int_{\tau}^0 M^{-1}(t,x,\lambda) \Dl \big [\Qa(\fpa_{\lambda}^x(t),\lambda)] g(\fpa_{\lambda}^x(t),\lambda) \dd t
\end{align}
Indeed, since $M^{-1}$ is the fundamental matrix of $\dot{\psi}=-\psi \Qa(\fpa_{\lambda}^x(t),\lambda) $, we can use the variation of constants formula to
the variational equation for $D_{\lambda_j}M^{-1}$ to obtain:
$$
D_{\lambda_j} M^{-1}(t,x,\lambda) M(t,x,\lambda)=\int_{t}^0 M^{-1}(s,x,\lambda)D_{\lambda_j} \big [\Qa(\fpa^{x}_{\lambda}(s),\lambda)\big ] M(s,x,\lambda) \dd s
$$
and, using expression~\eqref{exphfpa} of $g(\fpa_{\lambda}^x(t),\lambda)$,
$$
M^{-1}(t,x,\lambda) \Ti{}(\fpa_{\lambda}^x(t),\lambda) = \frac{d}{dt}\big[ M^{-1}(t,x,\lambda) g(\fpa_{\lambda}^x(t),\lambda)\big ].
$$
Then the result follows by integrating by parts
$$
\int_{\tau}^0 \big [\Dl M^{-1}(t,x,\lambda)M(t,x,\lambda)\big ]\big [M^{-1}(t,x,\lambda )\Ti{}(\fpa_{\lambda}^x(t),\lambda)\big ]\dd t.
$$
Next we prove that
$$
\widetilde{G}^x_{\lambda}(\tau):=\int_{\tau}^0 M^{-1}(t,x,\lambda) \Dx g(\fpa_{\lambda}^x(t),\lambda) \Dl \pa(\fpa_{\lambda}^x(t),\lambda)\dd t
$$
can be written as
\begin{align}\label{expdglambda}
\widetilde{G}^x_{\lambda}(\tau)=& -M^{-1}(\tau,x,\lambda) \Dx g(\fpa_{\lambda}^x(\tau),\lambda)
\Dl \fpa_{\lambda}^x(t) \notag \\
 &- \int_{\tau}^0 M^{-1}(t,x,\lambda) \Dx \Qa(\fpa_{\lambda}^x(t),\lambda) g(\fpa_{\lambda}^x(\tau),\lambda)  \Dl \fpa_{\lambda}^x(t)\\
&-\int_{\tau}^0 M^{-1}(t,x,\lambda) \Dx \Ti{}(\fpa_{\lambda}^x(t),\lambda) \Dl \fpa_{\lambda}^x(t).\notag
\end{align}
In order to prove~\eqref{expdglambda}, we will also integrate by parts. By using that $\Dl\fpa$ is the solution of
$$
\frac{d}{d t} \psi= \Dx \pa(\fpa_{\lambda}^x(t),\lambda) \psi + \Dl
\pa(\fpa_{\lambda}^x(t),\lambda),\qquad \psi(0, x,\lambda)=0
$$
we deduce that:
$$
 \Dl \pa(\fpa_{\lambda}^x(t),\lambda)= \Dx \fpa_{\lambda}^x(t)\frac{d}{d t} \left[(\Dx \fpa_{\lambda}^x(t))^{-1} \Dl
\fpa_{\lambda}^x(t)\right ].
$$
Therefore, since $ \Dx \big [g(\fpa_{\lambda}^x(t),\lambda)\big ]= \Dx g(\fpa_{\lambda}^x(t),\lambda) \Dx \fpa_{\lambda}^x(t)$,
\begin{equation}\label{expGfirstlambda}
\widetilde{G}^x_{\lambda}(\tau)= \int_{\tau}^0 M^{-1}(t,x,\lambda) \Dx \big[g(\fpa_{\lambda}^x(t),\lambda)\big ]\frac{d}{d t} \left[(\Dx \fpa_{\lambda}^x(t))^{-1} \Dl
\fpa_{\lambda}^x(t)\right ]\dd t.
\end{equation}
Applying~\eqref{edohfpa} with $h=g$ we have
$$
\frac{d}{d t} \big [ g(\fpa_{\lambda}^x(t),\lambda)\big ] =
\Qa(\fpa_{\lambda}^x(t),\lambda)  g(\fpa_{\lambda}^x(t),\lambda) +
\Ti{}(\fpa_{\lambda}^x(t),\lambda)
$$
which implies
\begin{align*}
\frac{d}{d t} \big (M^{-1}(t,x,\lambda) \Dx  \big [g(\fpa_{\lambda}^x(t),\lambda)\big ] \big ) =&
-M^{-1}(t,x,\lambda) \Qa(\fpa_{\lambda}^x(t),\lambda) \Dx  \big [g(\fpa_{\lambda}^x(t),\lambda)\big ]
\\ &+M^{-1}(t,x,\lambda)\Dx \big [ \Qa(\fpa_{\lambda}^x(t),\lambda) g(\fpa_{\lambda}^x(t),\lambda)\big ]\\
&+ M^{-1}(t,x,\lambda)\Dx \big [\Ti{}(\fpa_{\lambda}^x(t),\lambda)\big ].
\end{align*}
Finally, expression~\eqref{expdglambda} follows from integrating by parts in~\eqref{expGfirstlambda}. To do so we use that
if $H(x,\lambda):=\Qa(x,\lambda) g(x,\lambda)$ we have that $\Dx\big [H(\fpa_{\lambda}^x(t),\lambda)\big ] = \Dx H (\fpa_{\lambda}^x(t),\lambda)\Dx\fpa_{\lambda}^x(t) $
with
$$
\Dx H (\fpa_{\lambda}^x(t),\lambda) =
\Dx \Qa(\fpa_{\lambda}^x(t) ,\lambda) g(\fpa_{\lambda}^x (t),\lambda) + \Qa(\fpa_{\lambda}^x(t),\lambda) \Dx g(\fpa_{\lambda}^x(t),\lambda).
$$

Now we are going to relate expression~\eqref{dlamdahfirst} with~\eqref{expdglambda}. It is an straightforward computation (see Remark~\ref{remarkderlambda}) to check that
\begin{align*}
\Dl [\Qa(\fpa_{\lambda}^x(t),\lambda)] g(\fpa_{\lambda}^x(t),\lambda) =& \Dx \Qa(\fpa_{\lambda}^x(t),\lambda) g(\fpa_{\lambda}^x(t),\lambda) \Dl \fpa_{\lambda}^x(t)
\\ &+\Dl \Qa(\fpa_{\lambda}^x(t),\lambda) g(\fpa_{\lambda}^x(t),\lambda).
\end{align*}
Substituting the above expression of $\Dl [\Qa(\fpa_{\lambda}^x(t),\lambda)] g(\fpa_{\lambda}^x(t),\lambda)$ into~\eqref{dlamdahfirst}, using~\eqref{expdglambda} and the
definition of $\widetilde{G}^x_{\lambda}$ we have
\begin{align*}
G^x_{\lambda}(\tau)=& -\Dl M^{-1}(\tau,x,\lambda) g(\fpa_{\lambda}^x(\tau),\lambda) -M^{-1}(\tau,x,\lambda) \Dx g(\fpa_{\lambda}^x(\tau),\lambda)
\Dl \fpa_{\lambda}^x(\tau)\\
&+\int_{\tau}^0  M^{-1}(\tau,x,\lambda)\Ti{1}(\fpa_{\lambda}^x(t),\lambda) \dd t
\end{align*}
with $\Ti{1}$ defined in~\eqref{defDifT:param}.

To prove that $\lim_{\tau \to \infty} G^{x}_{\lambda}(\tau) = \int_{\infty}^0  M^{-1}(\tau,x,\lambda)\Ti{1}(\fpa_{\lambda}^x(t),\lambda) \dd t$ it remains to check that
$$
\overline{h}^x_{\lambda}(\tau):=\Dl M^{-1}(\tau,x,\lambda) g(\fpa_{\lambda}^x(\tau),\lambda) + M^{-1}(\tau,x,\lambda) \Dx g(\fpa_{\lambda}^x(\tau),\lambda)
\Dl \fpa_{\lambda}^x(\tau)
$$
goes to $0$ as $\tau\to \infty$ uniformly in $(x,\lambda)\in V\times \Lambda$. Indeed, the result follows from
\begin{align*}
\Vert \Dl \fpa_{\lambda}^x(\tau) \Vert &\leq K \Vert x \Vert \big (1+\dpa (N-1) \tau \Vert x \Vert^{N-1}\big )^
{-\a\left (1-\max\left \{0,1-\frac{\Apa}{\dpa}\right \}\right )},\qquad \\
\Vert \Dl M^{-1}(\tau,x,\lambda) \Vert&\leq K \big (1+\cpa(N-1)  \tau \Vert x \Vert^{N-1}\big )^{-\a\left (\frac{\BQ}{\cpa} - \max\left \{0,1-\frac{\Apa}{\dpa}\right \}\right )}.
\end{align*}
These bounds are obtained in a similar way as the ones of the corresponding derivatives with respect to $x$ in Lemma~\ref{lemma:hdif}.
First we write adequately $\Dl \fpa_{\lambda}^x$ and $\Dl M^{-1}$ by taking into account the differential equations that they both satisfy and property~\eqref{secondpropertychi}:
\begin{align*}
\Dl\fpa_{\lambda}^x(\tau)&=\int_{0}^{\tau} \Dx \fpa(\tau-s,\fpa^x_{\lambda}(s),\lambda) \Dl \pa(\fpa_{\lambda}^x(s),\lambda)\dd s, \\
\Dl M^{-1}(\tau,x,\lambda)&=-\int_{0}^{\tau} M^{-1}(s,x,\lambda) \Dl \big [\Qa(\fpa_{\lambda}^{x}(s),\lambda)\big ]M^{-1}(\tau-s,\fpa_{\lambda}^x(s),\lambda)\dd s.
\end{align*}
Bound~\eqref{Variationalbound} of $\Vert \Dx \fpa_{\lambda}^x(s)\Vert$, bound of $\Vert \fpa_{\lambda}^x(s)\Vert$ in Lemma~\ref{cotafpa} and the fact that $\Dl \pa \in \Hog{N}$, lead to
\begin{align*}
\Vert &\Dl \fpa_{\lambda}^x(\tau) \Vert \\ &\leq \frac{K\Vert x \Vert^{N}}{\big (1+ (N-1)\dpa  \tau \Vert x \Vert^{N-1}\big )^{\a\frac{\Apa}{\dpa}}}
\int_{0}^{\tau} \frac{1}{\big (1+ (N-1)\dpa  s \Vert x \Vert^{N-1}\big )^{\a \left (N-\frac{\Apa}{\dpa}\right )}}\dd s
\end{align*}
which gives the bound for $\Vert \Dl \fpa_{\lambda}^x(\tau)\Vert$.
Since
$$
\Dl \big [\Qa(\fpa_{\lambda}^x (\tau), \lambda)\big ]= \Dx \Qa(\fpa_{\lambda}^x (\tau),\lambda) \Dl \fpa_{\lambda}^x(\tau) +\Dl \Qa(\fpa_{\lambda}^x(\tau),\lambda),
$$
we have that
$$
\Vert \Dl \big [\Qa(\fpa_{\lambda}^x (\tau), \lambda)\big ]\Vert \leq K \big ( 1+ (N-1) \dpa \tau \Vert x \Vert^{N-1}\big )^{-\a \left (N-1 -\max\left \{0,1-\frac{\Apa}{\dpa}\right \}\right )}.
$$
Then, using~\eqref{boundXtXs} with $\MP=M$, we obtain the bound for $\Vert \Dl M^{-1}(\tau,x,\lambda)\Vert$.

Finally we easily check that the three terms in $M^{-1}(t,x,\lambda) \Ti{1}(\fpa_{\lambda}^x (t),\lambda)$ have a uniform
behavior of the form $t^{-a(\frac{\BQ}{\cpa} +N-1)}$ when $t$ is big and $\a(\frac{\BQ}{\cpa} +N-1)>1$. This
proves that indeed, $g$ is differentiable with respect to $\lambda$ and formula~\eqref{lemma:expDh:param} holds true.

Now assume that $\Ti{} \in \BH{\nu+N}{1}{\kappa}$. Applying the result when $\sigma=0$,
we get that $g\in \BH{\nu+1}{0}{\kappa_{\pa}}$ and in particular $\Dx g\in \BH{\nu}{0}{\kappa_{\pa}-1}$. Then we deduce that
$\Ti{1} \in \BH{\nu+N}{0}{\kappa_{\pa}-1}$. Therefore, using again
the present result for $\sigma=0$, $\Dl  g \in
\BH{\nu+1}{0}{\kappa_{\pa}-1}$, that is: $\Dx^j\Dl  g( x,\lambda)$
for $j\leq \kappa_{\pa}-1$ are continuous and bounded
and as a consequence $ g \in \BH{\nu+1}{1}{\kappa_{\pa}}$.
\end{proof}

\begin{proof}[End of the proof of Lemma \ref{lemma:param}]
We consider the differentiable and the analytic cases separately.

Assume that $\pa \in \BH{N}{s}{\kappa}$, $\Qa \in
\BH{N-1}{s}{\kappa}$ and $\T\in \BH{\m+N}{s}{\kappa}$ with
$\kappa=r+s$. We apply Lemma \ref{lemma:hdif:param} with $\Ti{}=\T$
and $\nu=\m$ and we obtain that the function $  h$ belongs to $\BH{\m+1}{1}{\kappa_{\pa}}$ with $\kappa_{\pa}=\rpa+s_{\pa}$
defined in~\eqref{defkappapa}. To finish the proof in the differentiable case we use induction.
Assume that $  h \in \BH{\m+1}{\sigma -1}{\kappa_{\pa}}$ with
$\sigma \leq s_{\pa}$. By definition of
$\BH{\m+1}{\sigma-1}{\kappa_{\pa}}$, we have that if $i+j\leq
\kappa_{\pa}$ and $i\leq \sigma-1$, then $\Dl^{i} \Dx^{j}   h$  are
continuous and bounded functions. We have to prove that indeed, $  h
\in \BH{\m+1}{\sigma}{\kappa_{\pa}}$.

We define $\HH{0}=  h$, $\TT{0}=\T$ and recurrently, for $1\leq i
\leq \sigma-1$:
\begin{align*}
\HH{i}( x,\lambda)&=\Dl \HH{i-1}( x,\lambda), \\
\TT{i}( x,\lambda)&= \Dl \TT{i-1}( x,\lambda) + \Dl \Qa( x,\lambda)
\HH{i-1}( x,\lambda) - \Dx \HH{i-1}( x,\lambda) \Dl \pa( x,\lambda).
\end{align*}
Note that by expression \eqref{lemma:expDh:param} in Lemma
\ref{lemma:hdif:param}, we have that
$$
\HH{\sigma-1} ( x,\lambda)= \int_{\infty}^0 M^{-1}( t, x,\lambda)
\TT{\sigma-1}(\fpa( t, x,\lambda),\lambda)\, d t.
$$
Since by induction hypothesis $\HH{0}\in \BH{\m+1}{\sigma-1}{\kappa_{\pa}}$ then $\HH{i}\in \BH{\m+1}{\sigma-1-i}{\kappa_{\pa}-i}$ and $\Dx \HH{i-1}
\in\BH{\m}{\sigma-i}{\kappa_{\pa}-i}$. These facts imply that $\TT{i}\in \BH{\m+N}{\sigma-i}{\kappa_{\pa}-i}$. Applying the
last formula for $i=\sigma-1$, one has that $\TT{\sigma-1}\in
\BH{\m+N}{1}{\kappa_{\pa}-\sigma+1}$. Therefore, applying Lemma
\ref{lemma:hdif:param} with $s=1$, one concludes that $\HH{\sigma-1}
\in \BH{\m+1}{1}{\kappa_{\pa}-\sigma+1}$.

Now we are almost done because, on the one hand, if $1\leq i\leq
\sigma-1$ and $1\leq i+j\leq \kappa_{\pa}$, all the derivatives
$\Dl^i \Dx^j   h$ are bounded and continuous by induction hypothesis
and on the other hand, since $\HH{\sigma-1} =\Dl^{\sigma-1}   h\in
\BH{\m+1}{1}{\kappa_{\pa}-\sigma+1}$ the same happens for
$$
\Dl \Dx^j \HH{\sigma-1} = \Dl \Dx^j \left (\Dl^{\sigma-1}   h\right
)
$$
if $1+j\leq \kappa_{\pa}-\sigma+1$, hence $\Dl^{\sigma} \Dx^j   h$
is continuous and bounded if $\sigma + j\leq \kappa_{\pa}$.

It remains to deal with the analytic case. We denote by
$\fpa(t,x,\lambda)$ the flow of $\dot{x}=p(x,\lambda)$. We claim
that, if $\r,\gamma$ are small enough, the complex set
$\Omega(\r,\gamma)$ is invariant by $\fpa(t,x,\lambda)$ for any
$\lambda\in \Lambda(\gamma)$. Indeed, first we note that
\begin{align*}
\pa(x,\lambda) =& \pa(\re x ,\re \lambda) + \ii D\pa(\re x,\re
\lambda)[\im x, \im \lambda] \\&- \int_{0}^1 (1-\mu)
D^2\pa(x(\mu),\lambda(\mu)) [\im x, \im \lambda]^2 \dd \mu,
\end{align*}
with $x(\mu)=\re x+ \ii \mu \im x$ and $\lambda(\mu)=\re \lambda+
\ii \mu \im \lambda$. We observe that, writing
$z_\mu=(x(\mu),\lambda(\mu))$:
\begin{align*}
D\pa (\re x,\re \lambda)[\im x, \im \lambda] =& \Dx \pa (\re x,\re \lambda) \im x + \Dl \pa(\re x,\re \lambda) \im \lambda, \\
D^2\pa(z_\mu) [\im x, \im \lambda]^2=& \Dx^2\pa(z_\mu)[\im x,\im x] + 2 \Dx\Dl\pa(z_\mu) [\im x ,\im \lambda] \\&+
\Dl^2\pa(z_\mu) [\im \lambda, \im \lambda].
\end{align*}
Then, since $\pa$ is homogeneous and analytic, we have that $\Dl
\pa, \Dl^2 \pa\in \Hog{N}$. Then, if $\lambda \in \Lambda(\gamma^2)$:
$$
\pa(x,\lambda)=\pa(\re x,\re \lambda) + \ii \Dx \pa(\re x,\re
\lambda) \im x + \gamma^2 \OO(\Vert x \Vert^N).
$$
From the above equality we can proceed as in the proof of
Lemma~\ref{invflux} to prove that $\Omega(\r,\gamma)$ is invariant
if $\r,\gamma$ are small enough. Then, the proof of the analytic case is completely analogous to
the one of Theorem~\ref{solutionlinearequation}, using the dominated convergence
theorem and the fact that the bounds are uniform for $\lambda\in \Lambda$.
\end{proof}

\subsection{End of the proof of Theorems~\ref{prop:param} and~\ref{maintheoremflow}}

First we discuss the case of maps. No matter what strategy we choose for
solving the cohomological equations for $\Kl{j}$ we have to deal with the remainders
$\Ef{x}^{j+N-1}$ and $\Ef{y}^{j+L-1}$ (Sections~\ref{subsectionstep3} and \ref{sec:linearequationx}).
Therefore, the first thing we need to do is to check what regularity
with respect to $(x,\lambda)$ they have. We deal with
$\Ef{x}^{j+N-1}$ being the case for $\Ef{y}^{j+L-1}$ analogous.
Recall that, as we prove in~\eqref{decompresidue}, $\Ef{x}^{j+N-1}$
was the homogeneous part of the error term $\Em{j-1}_x = F_x\circ
\KMil{j-1}  - \KMil{j-1}_x \circ \RMil{j+N-2} $. To prove this we used
that by induction $\KMil{j}$ and $\RMil{j+N-1}$ are sums of homogeneous functions and
Taylor's theorem by decomposing $F_x$ as in~\eqref{Taylor:homog}:
 $$
 F_x(x,y,\lambda)= x +  p( x,y,\lambda)+  F_{ x}^{N+1}( x,y,\lambda) + \cdots +   F_{ x}^{r}( x,y,\lambda)+   F_{ x}^{>r}( x,y,\lambda).
 $$
Since $p$ and $F_{x}^l$, $l=N+1,\cdots, r$, are homogeneous
polynomials with respect to $(x,y)$ and moreover $F\in \CC^{\DS}$, we have
that $p, F_{x}^l\in \CC^{\Sigma_{s,\infty}}$ for $l=N+1, \cdots, r$.
In fact they are analytic with respect to $x$ and $\CC^s$ with
respect to $\lambda$. Analogously for $\Ef{y}^{j+L-1}$.

The cases $M<N$ or $\Ap\geq \bp$ follows immediately from the
strategy in Section~\ref{sec:regularity} and Lemma~\ref{lemma:param}.

When, $M\geq N$ and $\Ap<\bp$ the first cohomological equation we
solve is
$$
\Dx \Kl{2}_y (x,\lambda) p(x,0,\lambda) - \Qa(x,\lambda)
\Kl{2}_y(x,\lambda) = \Ef{y}^{M+1},
$$
with $\Qa\equiv 0$ if $M>N$ or $\Qa(x,\lambda)=D_y q(x,0,\lambda)$ if $N=M$. Using
Lemma~\ref{lemma:param} with $\pa(x,\lambda)= p(x,0,\lambda)$, $\m=1$
and $\T=\Ef{y}^{M+1}$, we have that $\Kl{2}_y \in
\CC^{\Sigma_{s_*,\gdf-s_*}}$ where $s_{*}, \gdf$ are the given in
Theorem~\ref{prop:param}. Proceeding by induction as in
Section~\ref{sec:regularity}, we prove Theorem~\ref{prop:param}.

The proof of Theorem~\ref{maintheoremflow} is straightforward. Indeed,
following the strategy in Section~\ref{formalsolutionflow}, we
decompose $\Kalg{j}=\ME{\Kalg{j}} + \ZE{\Kalg{j}}$ where
$\ME{\Kalg{j}}$ is the time average of $\Kalg{j}$ which satisfies
equation~\eqref{eq:inductionflowsmean}. The same argument as in the
case of maps leads to conclude that $\ME{\Kalg{j}} \in
\CC^{\Sigma_{s_*,\gdf-s_*}}$. Finally, $\ZE{\Kalg{j}}$ satisfies the
equation
$$
\partial_t \ZE{\Kalg{j}} = (\ZE{\Ef{x}^{j+N-1}},\ZE{\Ef{y}^{j+L-1}})
$$
and therefore it is $\CC^{s}$ with respect to $(t,\lambda)$ and
analytic with respect to $x$.

\section{Examples}\label{sectionexamples}

In this section we are going to see that our hypotheses are all of
them necessary in order to be able to solve the cohomological
equations for $\Kl{j}_y$.

In Section~\ref{sec:examples:special}, we present an alternative
(and easy) way for solving the cohomological equations in a particular
setting. We also provide two examples of analytic maps (or even
analytic vector fields) satisfying all the hypotheses, where the
solution of the corresponding cohomological equations are only $\CC^r$
in $\text{int}(V)$. One of these examples, satisfies that $\Ap=0$
and the other one is such that $\Ap>0$. We will also check that the
condition $\Ap>\ddp$ is essential to obtain analyticity. Moreover, we
will also check that, when $\Ap<\ddp$, $\gdf$ is the maximum degree
of differentiability.

Recall that the cohomological equations for $\Kl{j}_x$ can be always
solved by choosing $\Rl{j}$ properly. However, it is interesting to
obtain the simplest normal form, to be able to solve the cohomological
equations for $\Kl{j}_x$ with $\Rl{j}\equiv 0$. We present an
example where the cohomological equation for $\Kl{j}_x$ can not be
solved with $\Rl{j}\equiv 0$ if the degree $j\leq \df$ with $\df$
the degrees of freedom to chose $\Kl{j}_x$ defined in
\eqref{degreefreedom}. In consequence, the normal form $\Rl{j}$
stated in the main result, is the simplest one, generically.

\subsection{Example 1. A particular form of $p$}\label{sec:examples:special}
Let $F$ be a map of the form~\eqref{defF}, satisfying hypotheses H1, H2 and H3.

\begin{claim}
Let $\pa(x)=p( x,0)$. Assume that $\pa( x)= \pa_0( x)  x$, with
$\pa_0:V\to \RR$ and $\pa$ and $V$ satisfy hypotheses HP1, HP2. Then
the approximate parametrization $K^{\le}$ and the reparametrization
$R$ are rational functions (which in general are not polynomials).
Moreover $R$ can be chosen to be of the form $R(x) = x + \pa_0(x)x +
R^{2N-1}(x)$, as in the one dimensional case.
\end{claim}

\begin{proof}
HP1
implies $-2<\pa_0( x) <0$, $x\in V$. Then, the auxiliary
equation~\eqref{modellinearequation} reads
$$
D  h( x) \pa_0( x)  x  - \Qa( x)   h( x) = \T ( x).
$$
Since we look for homogeneous solutions of degree $\m+1$, using
Euler's identity, namely $D  h( x) x = (\m+1)   h(x)$, if $  h$ is
homogeneous of degree $\m+1$, equation~\eqref{modellinearequation}
can be written as:
$$
\big[ (\m+1) \pa_0(x)  \Id - \Qa( x) \big ]   h( x) = \T( x).
$$
Consequently, we can solve this equation for any homogeneous
function $\T \in \Hog{\m+N}$ if and only if the matrix $ (\m+1)
\pa_0(x)  \Id - \Qa( x)$ is invertible for all $x\in V$. Assume the
contrary, that is, there exists $x\in V$ and a eigenvector $v$, with
$\Vert v \Vert=1$, of the eigenvalue $0$.
For the next computations we assume that $\m>0$ and $V$ is small enough
so that $-1<\pa_0(x)<(\m+1)^{-1}$ if $x\in V$.
Then $\Qa( x) v =
(\m+1)\pa_0(x) v$ and
$$
\Vert \big (\Id - \Qa( x)\big ) v \Vert = 1-(\m+1) \pa_0(x).
$$
By definition \eqref{defconstantspa} of $\BQ$:
$$
\Vert \big (\Id -\Qa( x)\big )v \Vert  \leq \Vert \Id -\Qa( x)\Vert
\leq 1- \BQ \Vert  x \Vert^{N-1}
$$
and by definition \eqref{defKAS} of $\apa$, $\pa_0( x)\leq -\apa
\Vert  x \Vert^{N-1}$, then, we deduce that, if the matrix $ (\m+1)
p_0(x)  \Id - \Qa( x)$ is not invertible,
$$
\m+1 +\frac{\BQ}{\apa}\leq 0.
$$
Consequently, if $\m+1 + \frac{\BQ}{\apa}>0$, for any $x\in V$, the
matrix $ (\m+1) \pa_0(x)  \Id - \Qa( x)$ is invertible and moreover,
the solution of the auxiliary equation is
$$
  h( x) = \big[ (\m+1) \pa_0(x)  \Id - \Qa( x) \big ]^{-1}\T( x).
$$

Depending on the values of $M,N$, $\Kl{j}_y$ has to satisfy the
cohomological equations~\eqref{eqlKyN<Mn} if $N<M$,
equation~\eqref{eqlKyN=Mn} if $N=M$ and~\eqref{eqlKyN>Mn} when
$N>M$. Then, taking in the auxiliary equation $\T(x) =
\Ef{y}^{j+L-1}$, and either $\Qa(x)=0$ if $N<M$ or $\Qa(x) = D_y
q(x,0)$ if $N\geq M$, we have that
\begin{equation*}
\Kl{j}_y( x) =
\begin{cases}
{\displaystyle \frac{1}{j \pa_0( x)} \Ef{y}^{j+N-1}( x), }& N<M, \\ \\
{\displaystyle \big[ j \pa_0(x)  \Id - D_yq( x,0) \big ]^{-1} \Ef{y}^{j+N-1}( x), } & N=M,\\ \\
{\displaystyle -D_y q( x,0)^{-1} \Ef{y}^{j+M-1}( x), }& N>M.
\end{cases}
\end{equation*}
To obtain $\Kl{j}_x$ and $\Rl{j+N-1}$ we have to deal
with~\eqref{eqlKxfirst} which in abstract form reads
\begin{align*}
Dh(x) \pa_0(x) x - D(\pa_0(x) x) h(x) + \eta(x) &= \big( j\pa_0(x)
\Id - D(\pa_0( x)x) \big ) h(x) + \eta(x)\\ &= \T(x),
\end{align*}
where $h=\Kl{j}_x$, $\eta=\Rl{j+N-1}$ and $\T(x)=\Ef{ x}^{j+N-1}(
x)+D_{y}p( x,0) \Kl{j}_{y}( x)$. Assume that the matrix in the above
equation is not invertible for some $x\in V$. Then there exists
$v\in \RR^n$ with $\Vert v \Vert=1$ such that
$$
(j-1)\pa_0( x) v  = (D\pa_0( x) v ) x.
$$
This implies that $x$ and $v$ are linearly dependent: $ v=\lambda x$
for some $\lambda\in \RR\backslash \{0\}$. Then
$$
(j-1)\pa_0( x) \lambda x  = \lambda (D\pa_0( x) x ) x=\lambda(N-1)
\pa_0(x) x
$$
and hence $j=N$.
As a consequence, for $j\geq 2$, $j\neq N$, the previous matrix is
invertible, we can take $\Rl{j+N-1}\equiv 0$ and
$$
\Kl{j}_{ x} ( x)  = \big [j\pa_0(x)  \Id - D_x p( x,0)\big ]^{-1}
(\Ef{ x}^{j+N-1}( x)+D_{y}p( x,0) \Kl{j}_{y}( x)).
$$
When $j=N$, we can take $\Kl{N}_{ x}$ as any function in $\Hog{N}$
and then
$$
\Rl{2N-1}( x) = \Ef{ x}^{2N-1}( x) -D\Kl{N}_{ x}( x)  \pa_0( x) x  +
D_x p( x,0) \Kl{N}_{ x}( x) + D_{y}p( x,0) \Kl{N}_{y}( x).
$$
\end{proof}

\subsection{Example 2. On the necessity of hypothesis H3}

Consider the system of ordinary differential equation in $\RR^2
\times \RR$
\begin{equation*}
\dot{x}_1 = -x_1^2, \qquad \dot{x}_2 = - a x_1 x_2, \qquad
\dot{y}=bx_1 y + x_2^3
\end{equation*}
with $a,b>0$ and $b+3a\leq 1$. This system was also considered in
Section 5.1 of~\cite{BFM2015a}. There it was shown that the time $1$
map $F$ of the flow defined by the above system satisfies hypotheses
H1 and H2 in a suitable domain $V$ but that it has no invariant
manifold over $V$.

\begin{claim}
There exist $V\subset \RR^2$, star-shaped with respect to the
origin, where $F$ satisfies hypotheses H1 and H2 but in which the
cohomological equations~\eqref{eqlKyN=Mn} have no homogeneous
solution in~$V$. That is, H3 is needed both at a formal and at an
analytical level.
\end{claim}

It is clear that $F$ is a map of the form~\eqref{defF} with $N=M=2$,
$p(x,y)=(-x_1^2,  -ax_1x_2)$ and $q(x,y)=bx_1y$.

We denote $x=(x_1,x_2)$. Let $\varphi$ be the flow of $\dot x =
p(x,0)$, which can be explicitly computed:
\begin{equation*}
\varphi(t,x)= \big (\varphi_{x_1}(t,x,y),\varphi_{x_2}(t,x)\big
)=\left (\frac{x_1}{1+tx_1},\frac{x_2}{(1+tx_1)^a}\right ).
\end{equation*}

\begin{proof}
Hypotheses H1 and H2 are satisfied for $F$ in the convex domain
\begin{equation*}
W = \left \{ x=(x_1, x_2) \in \RR^2 :  |x_2| <(1-a)x_1<\frac{2}{a+1}
\right \}
\end{equation*}
with the supremum norm. Actually, $\Ap =a^2$, $\ap = 1$ and $\Bq = b$.
However there is no open invariant set for $F_x$ contained in $W$
and, as a consequence, hypothesis H3 is not satisfied. Indeed, assume there
is such open set and that $x^0 =(x_1^0,x_2^0) \in W$, $x_2^0\neq 0$, and let
\begin{equation*}
x^n = F_x(x^{n-1},0) = (F^n)_x(x^0,0) = (F_x)^n(x^0,0) =\left
(\frac{x_1^0}{1 + nx_1^0}, \frac{x_2^0}{\big (1 + nx_1^0\big
)^{a}}\right ) .
\end{equation*}
If the sequence $x^n \in W$, $\forall n\geq 0$, then  $(1-a)x_1^0
\geq |x_2^0| \big (1 + n x_1^0 \big )^{1-a}$, $\forall n\geq 0$,
which is false since $a<1$.

Following the algorithm described in
Section~\ref{formalsolutionsection}, we compute
$$
\Em{1}( x) = F\circ \KMil{1}( x)  - \KMil{1} \circ \RMil{N}( x) = F(
x,0) - \big ( x+p( x,0), 0\big ) = (0,0,x_2^3)+ \dots.
$$
Therefore, the first cohomological equation that we need to solve is
\begin{equation}\label{eq:example1}
D\Kl{2}_{y}( x)   p( x,0)   -D_y q( x,0) \Kl{2}_{y}( x)  =  x_2^3.
\end{equation}
Let $\My(t, x)=(1+tx_1)^b$ be the fundamental matrix of $\dot{z} =
D_y q(\varphi(t,x),0)z = b \varphi_{x_1}(t,x) z$. Formula
\eqref{defh0} applied to $\pa(x) = p(x,0)$, $\Qa(x) = D_yq(x,0)$ and
$\T(x) =x_2^3$  states that
$$
\Kl{2}_y( x) = \int_{\infty}^0 \My^{-1}( t, x) \T(\fpa( t, x))\,d t
= x_2^3 \int_{\infty}^0 \frac{1}{(1+ t x_1)^{b+3a}} \,d t
$$
which, obviously, is not convergent if $b+3a \leq 1$. In conclusion,
our algorithm can not be applied if H3 does not hold. Finally, we
remark that, by Corollary~\ref{cor:theoremlinearequation},
equation~\eqref{eq:example1} has no homogeneous solution.
\end{proof}

\subsection{Example 3. The loss of differentiability}
We consider the map $(x,y)\in \RR^2 \times
\RR \mapsto F(x,y)\in \RR^3$ given by
\[
F(x,y)= \left (\begin{array}{c} x+p(x) \\ y+q_1(x)y +
g(x)\end{array} \right ),  \qquad x = (x_1,x_2) \in \RR^2, \; y
\in\RR,
\]
where
\[
 p(x) = \left (\begin{array}{c} -x_1^3\\ - cx_2^3\end{array}\right ),  \qquad q_1(x) = d (x_1^2 + x_2^2),\qquad g(x) = x_1^i x_2^j,
\]
with $i+j\geq 4$ and $c,d>0$.

\begin{claim}
There exists $V\subset \RR^2$, star-shaped with respect to the
origin, where $F$ satisfies hypotheses H1, H2 and H3.

Let $K$ be any approximate solution of~\eqref{HIKl} provided by
Theorem~\ref{formalsolutionprop}. If the choice of $i,j,c,d$ is such
that $i+d=j+d/c=4$, then $K$ is only $j+1$ times differentiable.
This is the optimal regularity claimed by
Theorem~\ref{formalsolutionprop}.
\end{claim}

Possible choices are $i=j=d=2$, $c=1$ and $i=3,j=d=1$, $c=1/3$.

\begin{proof}
We will compute the term $K^2_y$ explicitly and check that if has
precisely the claimed regularity.

Let $V=B_{\r_0}\setminus\{0\} \subset \mathbb{R}^2$ with $\r_0$
small. We claim  that, hypotheses H1, H2, H3 are satisfied  in $V$
for the Euclidean norm $\Vert \cdot \Vert_2 $ (in fact, they are
satisfied with any norm). Indeed, we have that $V$ is invariant by
$x\mapsto x+p(x)$ if $\r_0$ is small and
\begin{align*}
\ap=\frac{c}{1+c}+\OO(\r_0^2)>0 ,\qquad \Ap=0,\qquad \bp =
\max\{1, c\},\qquad \Bq=d>0.
\end{align*}

We have that $E^{>1}(x) = (E^4_x,E^4_y)(x) =
F(x,0)-(x+p(x),0)=(0,g(x))$. Then, the first cohomological equation we
have to solve is
\begin{equation*}
D\Kl{2}_{y}( x)   p( x)   - q_1( x) \Kl{2}_{y}( x)  =  g(x)=x_1^i
x_2^j ,
\end{equation*}
which, according to~\eqref{defh0}, gives
\[
\Kl{2}_y (x) = x_1^i x_2^j \int_{\infty}^0 \frac{1}{\big (1+ 2 t
x_1^2\big )^{\frac{i+d}{2}} \big (1+ 2 t c x_2^2\big )^{\frac{j}{2}
+ \frac{d}{2c}}} \,d t.
\]
According to Theorem~\ref{formalsolutionprop}, the degree of
differentiability of $K$, given in~\eqref{defminimaregularitat}, is
the maximum integer satisfying
\[
\gdf < 2+\frac{\Bq}{\bp}  = 2 + \frac{d}{\max\{1,c\}}.
\]
Now we take values of $i,j,c,d$ such that $i+d=j+d/c=4$. It is a
calculation to check that
\[
\Kl{2}_y(x) = x_1^i x_2^j \left [\frac{cx_2^2 +
x_1^2}{2(cx_2^2-x_1^2)^2} - c\frac{x_1^2 x_2^2}{(cx_2^2-x_1^2)^3}
\log \left (\frac{cx_2^2}{x_1^2}\right )\right ].
\]
We study $\Kl{2}_y$ in the subdomain $W=\{|\sqrt{c} x_2| <|x_1|\}$
of $V$. On $W$, $\Kl{2}_y$ is
\begin{align*}
\Kl{2}_y(x) =&   x_1^{i-j-2} \frac{x_2^j}{x_1^j}\left
(1+\frac{cx_2^2}{x_1^2}\right )\left (1 -\frac{cx_2^2}{x_1^2}\right)^{-2} \\
&+ 2 x_1^{i-j-6} \frac{x_2^{j+2}}{x_1^{j+2}} \left (1-\frac{cx_2^2}{x_1^2}\right )^{-3}
 \log\left (\frac{\sqrt{c}|x_2|}{|x_1|}\right ).
\end{align*}
To study the differentiability of $\Kl{2}_y$ on $W$ is equivalent to
study the derivability of $\chi(z)=z^{j+2}\log(|z|)$, which is only
$\CC^{j+1}$ at $z=0$ but it is not $\CC^{j+2}$. Consequently,
$\Kl{2}_y$ is only $\CC^{j+1}$ at the points $(x_1,0)\in W\subset
V$. Note that, with the two choices of the parameters $i,j,d,c$, we
have that, $d=j$ and $c\leq 1$. Then, $\gdf<2+j$, that is,
$\gdf=1+j$ which coincides with regularity of $K^2_y$ at $x_2=0$.
\end{proof}

\subsection{The reparametrization $R$}

We consider the map given by
\begin{equation*}
F(x,y)= \left (\begin{array}{c} x+p(x)+f(x)\\ y+q_1(x)y +
g(x)\end{array} \right ),\qquad (x,y)\in \mathbb{R}^2 \times \RR,
\end{equation*}
with $p(x) = ( -x_1^N, - cx_1^{N-1}x_2)$, $N\geq 2$, $q_1(x) =
(x_1^2 + x_2^2)^{(M-1)/2}$, $M$ odd and $M\geq 3$, $g\in \Hom{M+1}$
and $f\in \Hom{N+1}$.

\begin{claim}
Assume $c>1$. $F$ satisfies hypotheses H1, H2 and H3 with the
supremum norm in the set
\begin{equation}\label{defVexample}
V=\{ x\in \RR^2 : |x_2|<x_1\}.
\end{equation}
For any approximate solutions $K$ and $R$ given by
Theorem~\ref{formalsolutionprop}, $R$ has the form
$$
R(x) = x+p(x) + \sum_{j=2}^{N} R^{j+N-1}(x),\qquad R^{j+N-1} \neq 0
,\quad j=2,\cdots, N.
$$
\end{claim}

In the case of one dimensional manifolds, it was proven
in~\cite{BFdLM2007}) that one can always take $R^{j+N-1}=0$ if
$j=2,\dots,N-1$.

\begin{proof}
It is easy to see that Hypotheses H1, H2 and H3 hold in~$V$, as well
as to compute the value of the constants $\Bp=Nc$, $\ap=1$,
$\Ap=-c(N-2)$ and $\bp=c$. Consequently we have that
$$
\df> N-1 + [ Nc]\geq 2N-1.
$$

What we are going to check is that, necessarily, for solving the
cohomological equations~\eqref{eqlKxfirst} for $\Kl{j}_x$ in~Section
\ref{formalsolutionsection} for values of $2\leq j \leq \df-N+1$, we
have to take $\Rl{j+N-1}\nequiv 0$. Indeed, if not, the cohomological
equations~\eqref{eqlKxfirst} for $2\leq j\leq N$ are
\begin{equation*}
\begin{aligned}
D\Kl{j}_x (x)p(x) - & Dp(x) \Kl{j}_x(x)
\\ =& D\Kl{j}_x (x)\left
(\begin{array}{c} -x_1^N \\ -c x_1^{N-1} x_2 \end{array}\right )+ \left ( \begin{array}{cc} Nx_1^{N-1} & 0 \\ c (N-1)x_1^{N-2} x_2  & c x_1^{N-1} \end{array}\right ) \Kl{j}_x(x)
\\ = &E^{j+N-1}_x(x),
\end{aligned}
\end{equation*}
where $E^{j+N-1}_x$ is a homogeneous function of degree $j+N-1$.

We focus our attention to the equation for the first component of
$\Kl{j}_{x}$,
\begin{equation}\label{exampleeqKx1}
x_1 \Di{1}\Kl{j}_{ x_1}( x) + c  x_2 \Di{2}\Kl{j}_{ x_1}( x)- N
\Kl{j}_{x_1}( x)= -x_1^{1-N}E^{j+N-1}_{x_1}( x).
\end{equation}
We introduce the auxiliary functions $h(z)= \Kl{j}_{ x_1}(1,z)$ and
$T(z) =E^{j+N-1}_{x_1}(1,z)$. Notice that we can recover
$\Kl{j}_{x_1}( x)$ from the identity:
\begin{equation}\label{relhKj}
\Kl{j}_{ x_1}(x_1,x_2) = x_1^j h\big( x_2/ x_1\big ).
\end{equation}
Using Euler's identity $j \Kl{j}_x (x) = D\Kl{j}_x (x) x$ and
rearranging terms in~\eqref{exampleeqKx1}, we obtain that $h$ is a
solution of the differential equation:
\begin{equation}\label{edohKj}
(c-1)\frac{d}{dz} h(z) = \frac{N-j}{z} h(z) - \frac{T(z)}{z}.
\end{equation}
We study the solutions of~\eqref{edohKj}. Assume the easiest case,
that is $E^{j+N-1}_{x_1}$ is a homogeneous polynomial. Then $T(z)$
is a polynomial of degree $j+N-1$ which we write as: $T(z) =
\sum_{l=0}^{j+N-1} a_l z^l$. From the form of~\eqref{edohKj} the solutions are defined for $z\in(0,\infty)$ and for $z\in (-\infty,0)$. When $j=N$, equation~\eqref{edohKj}
yields:
$$
(c-1) h(z) = C- a_0 \log |z| - \sum_{l=1}^{2N-1} \frac{a_l}{l}z^l
$$
for some constant $C$. Then, by~\eqref{relhKj}
$$
\Kl{N}_{ x_1}(x) = \frac{x_1^N}{c-1}\left (C - a_0 \log \left |
\frac{x_2}{x_1}\right | - \sum_{l=1}^{2N-1}
a_l\frac{x_2^l}{lx_1^l}\right )
$$
which is not defined for $x_2=0$ contained in the set $V$
in~\eqref{defVexample}. So that equation~\eqref{exampleeqKx1} can
not be solved in $V$ for $j=N$. Even more, when $j\neq N$, denoting
$\beta= (N-j)/(c-1)$
$$
h(z)=|z|^{\beta} C  - |z|^{\beta} \int_{1}^z w^{-\beta -1} T(w) \,
dw =|z|^{\beta} C  - |z|^{\beta} \sum_{l=0}^{j+N-1} \int_{1}^z a_l
w^{-\beta-1+l}\,dw.
$$
When $\beta=l\in\{0,\cdots, j+N-1\}$, $h$ will have the term $a_l
\log|z|$ and, as in the case $j=N$, $\Kl{j}_{ x_1}$ will have the
term $\log(|x_2|/|x_1|)$ which, again, is not defined in the set $V$. We realize this case for $j<N$ taking, for instance, $c=2$
and $l=N-j$. On the contrary, $\Kl{j}_{ x_1}$ is well defined if
$j>N$.
\end{proof}

\section{Acknowledgments}
I.B and P.M. have been partially supported by the Spanish Government
MINECO-FEDER grant MTM2015-65715-P and the Catalan Government grant
2014SGR504. The work of E.F. has been partially supported by the
Spanish Government grant MTM2013-41168P and the Catalan Government
grant 2014SGR-1145.

\section*{References}
\bibliography{references}
\bibliographystyle{alpha}
\end{document}